\documentclass[a4paper,12pt]{article}

\usepackage{amsmath,amssymb,amsthm,stmaryrd} 
\usepackage{mathtools}
\usepackage{float} 
\usepackage{color}
\usepackage{graphics}
\usepackage{enumerate}
\usepackage[dvipsnames]{xcolor}
\usepackage[hyperindex]{hyperref}
\usepackage{tikz}
\usetikzlibrary{math}
\usetikzlibrary{calc}
\usetikzlibrary{patterns}
\usetikzlibrary{arrows.meta}

\pgfdeclarelayer{bg}    
\pgfsetlayers{bg,main} 

\numberwithin{equation}{section}

\theoremstyle{plain}
\newtheorem{theorem}{Theorem}[section]
\newtheorem{lemma}[theorem]{Lemma}
\newtheorem{proposition}[theorem]{Proposition}

\theoremstyle{definition}
\newtheorem{definition}[theorem]{Definition}

\newtheorem{notation}[theorem]{Notation}

\theoremstyle{remark}
\newtheorem{remark}[theorem]{Remark}
\newtheorem{example}[theorem]{Example}
\newtheorem*{claim}{Claim}

\newcommand{\les}{\lesssim}
\newcommand{\ges}{\gtrsim}

\newcommand{\bks}{\backslash}
\newcommand{\R}{\mathbb{R}}
\newcommand{\N}{\mathbb{N}}
\newcommand{\C}{\mathbb{C}}
\newcommand{\Z}{\mathbb{Z}}
\newcommand{\X}{\mathbb{X}}

\newcommand{\CC}{\mathcal{C}}
\newcommand{\QQ}{\mathcal{Q}}

\newcommand{\MM}{\mathcal{M}}
\newcommand{\HH}{\mathcal{H}}
\newcommand{\DD}{\mathcal{D}}

\newcommand{\II}{\mathcal{I}}
\newcommand{\JJ}{\mathcal{J}}
\newcommand{\Qrb}{Q_{\ov r}}
\newcommand{\Qrbq}{Q_{\ov r/4}}
\newcommand{\Qtrb}{{\widetilde Q}_{\ov r}}
\newcommand{\sh}{^{\mathrm h}}
\newcommand{\sv}{^{\mathrm v}}
\newcommand{\we}{\wedge}
\newcommand{\hel} {
\hskip2.5pt{\vrule height7pt width.5pt depth0pt}
\hskip-.2pt\vbox{\hrule height.5pt width7pt depth0pt}
\, }
\newcommand{\restr}{\hel}
\newcommand{\pf} {\,\!_\#\,} 

\newcommand{\mc}{\mathcal}

\newcommand{\sm}{\setminus}
\newcommand{\pt}{\partial}
\newcommand{\vhi}{\varphi}
\newcommand{\eps}{\varepsilon}
\newcommand{\te}{\theta}
\newcommand{\oo}{\infty}
\newcommand{\nb}{\nabla}
\newcommand{\ov}{\overline}

\newcommand{\wt}{\widetilde}
\newcommand{\dw}{\downarrow}
\newcommand{\up}{\uparrow}
\newcommand{\cd}{\cdot}
\newcommand{\longto}{\longrightarrow}

\newcommand{\un}{\bs1}
\newcommand{\Om}{\Omega}
\newcommand{\om}{\omega}

\newcommand{\lb}{[\hspace{-1.68pt}[}
\newcommand{\rb}{]\hspace{-1.65pt}]}
\newcommand{\sub}{\subset}
\newcommand{\be}{\begin{equation}}
\newcommand{\ee}{\end{equation}}

\newcommand{\U}{S}

\newcommand{\XXint}[3]{{\setbox0=\hbox{$#1{#2#3}{\int}$} 
      \vcenter{\hbox{$#2#3$}}\kern-.5\we0}}

\newcommand{\void}{\varnothing}

\newcommand{\st}{\stackrel}
\newcommand{\lt}{\left}
\newcommand{\rt}{\right}
\newcommand{\bs}{\boldsymbol}

\newcommand{\M}{\mathbf{M}}

\newcommand{\E}{\mathcal{E}}

\usepackage{xspace}
\makeatletter
\DeclareRobustCommand\onedot{\futurelet\@let@token\@onedot}
\newcommand{\@onedot}{\ifx\@let@token.\else.\null\fi\xspace}
\renewcommand{\ae}{a.e\onedot}
\newcommand{\ie}{\textit{i.e}\onedot}
\makeatother

\DeclareMathOperator{\Lip}{Lip}

\DeclareMathOperator{\supp}{supp}

\DeclareMathOperator{\diam}{diam}

\DeclareMathOperator{\Span}{span}

\DeclareMathOperator{\Vol}{Vol}
\DeclareMathOperator{\dens}{dens}
\DeclareMathOperator{\sign}{sign}
\DeclareMathOperator{\Sl}{Sl}
\newcommand{\cM}{\MM}

\newcounter{proof-step}

\title{Non-convex functionals penalizing simultaneous oscillations along two independent directions: structure of the defect measure}

\author{M. Goldman\footnote{CMAP, CNRS, \'Ecole polytechnique, Institut Polytechnique de Paris, 91120 Palaiseau,
France, email: michael.goldman@cnrs.fr} \and B. Merlet\footnote{Univ. Lille, CNRS, UMR 8524, Inria - Laboratoire Paul Painlev\'e, F-59000 Lille, email: benoit.merlet@univ-lille.fr}}

\begin{document}

 \maketitle
\begin{abstract}
We continue the analysis of a family of energies penalizing oscillations in oblique directions: they apply to functions $u(x_1,x_2)$ with $x_l\in\R^{n_l}$ and vanish when $u(x)$ is of the form $u_1(x_1)$ or $u_2(x_2)$. We mainly study the rectifiability properties of the \emph{defect measure} $\nb_1\nb_2u$ of  functions with finite energy. 

The energies depend on a parameter $\theta\in(0,1]$ and the set of functions with finite energy  grows with $\theta$.  For $\theta<1$ we prove that the defect measure is $(n_1-1,n_2-1)$-tensor rectifiable in $\Om_1\times\Om_2$. We first get the result for $n_1=n_2=1$ and deduce the general case through slicing using White's rectifiability criterion.

 When $\theta=1$ the situation is less clear as measures of arbitrary dimensions from zero to $n_1+n_2-1$ are possible. We show however, in the case $n_1=n_2=1$ and for Lipschitz continuous functions, that the defect measures are $1\,$-rectifiable. This case bears strong analogies with the study of entropic solutions of the eikonal equation.
\end{abstract}

 \section{Introduction} \label{SIntro}
 
As in~\cite{GM1}, we decompose the euclidean space $\X=\R^n$ as,
\[
\R^n=\X_1\oplus \X_2,\quad\text{with}\quad n_l:=\dim \X_l\ge 1\ \text{ for }l=1,2.
\]
The spaces $\X_1$ and $\X_2$ are assumed to be orthogonal. We then set 
\[
 K:=\X_1\cup \X_2.
\]

We consider a domain $\Om\sub\R^n$ which writes as $\Om=\Om_1+\Om_2$ where, for $l=1,2$, $\Om_l\sub \X_l$ is a nonempty bounded domain. In this paper, we push further the analysis of the differential inclusion

\be\label{eq:difincl1}
 \nb \times v=0 \qquad \text{and }\qquad v\in K\ \text{\ae} 
\ee

Here $\nb \times v=0$ means that $\pt_i v_j= \pt_j v_i$ for $1\le i,j\le n$. Writing then $v=\nb u$ and using the notation $\nb_l$ for the derivatives with respect to the variable in $\X_l$,~\eqref{eq:difincl1} is equivalent to 
\be\label{eq:difincluintro}
 |\nb_1 u||\nb_2 u|=0\quad\ \text{almost everywhere in }\Om.
\ee
As noticed in~\cite{GM1}, even under an $L^\oo$ assumption on $v$ (correspondingly in the class of Lipschitz functions $u$),~\eqref{eq:difincl1} is far from rigid. In particular it is not strong enough to characterize the subset of functions $u\in\Lip(\Om)$ of the form $u(x_1+x_2)=u_1(x_2)$ or $u(x_1+x_2)=u_2(x_2)$ for $x_1\in\Om_1$, $x_2\in\Om_2$. This motivates the introduction of an energy based on a discrete version of~\eqref{eq:difincluintro} (see also~\cite{GRun} for another motivation).
We fix $\te_1,\te_2>0$ and assume
\[
\te:=\te_1+\te_2\le1.
\]
We also fix a radial function $\rho\in L^1(\R^n,\R_+)$ supported in $B_1$ and such that $\int \rho=1$. Denoting $L(\Om)$ the space of measurable functions over $\Om$, for $\eps>0$ and $u\in L(\Om)$, we define the energy:
\be\label{Eeps}
\E_\eps(u):=\int_{\R^n}\int_{\Om^\eps}\dfrac{|Du(x,z_1)|^{\te_1}|Du(x,z_2)|^{\te_2}}{|z|^2}\,dx\, \rho_\eps(z)\, dz,
\ee
where we use the following conventions.
\begin{enumerate}[(i)] 
\item $\rho_\eps(z):=\eps^{-n}\rho(\eps^{-1} z)$;
\item for $l=1,2$, $z_l$ denotes the component in $\X_l$ of $z\in \X_1+\X_2$;
\item we use the notation
\[
Du(x,z):=u(x+z)-u(x);
\]
\item and we integrate over the restricted domain\footnote{Notice that  $\Om^\eps$ is empty when $\eps$ is too large.}
\[
\Om^\eps:=\Om_1^\eps +\Om_2^\eps,\qquad\text{where}\qquad  \Om_l^\eps:=\{x_l\in \Om_l: d(x_l, \X_l\sm\Om_l)>\eps\}.
\]
\end{enumerate}
We then send $\eps$ to 0 to obtain the energy:
\[
\E(u):=\liminf_{\eps\dw 0} \E_{\eps}(u).
\]
\begin{remark}
The definition of~\cite{GM1} allows for $\te>1$ and involves another parameter $p>0$ which corresponds to the exponent of $|z|$ in the denominator of~\eqref{Eeps}. Here we have fixed $p=2$ which is the relevant value in the case $\te\le1$ as shown in~\cite{GM1}.
\end{remark}

\begin{remark}
 As we will see in Proposition~\ref{prop:differincl}, if $u$ is Lipschitz continuous with $\E(u)<\oo$ then $v=\nb u$ satisfies~\eqref{eq:difincl1}. However, while~\eqref{eq:difincl1} requires at least $v\in L^1$, \ie $u\in W^{1,1}$ to make sense, $\E(u)$ is well defined as soon as $u$ is measurable.
\end{remark}

\begin{remark} Consider some parameters $\te'_l\ge \te_l>0$ for $l=1,2$ and assume that $\te':=\te'_1+\te'_2\le1$. For $u\in L^\oo(\Om)$ we have, with obvious notation,
\[
\E^{(\te_1,\te_2)}(u)<\oo\ \implies\ \E^{(\te'_1,\te'_2)}(u)<\oo.
\]
In particular, the larger $\theta_1,\theta_2$ are, the less coercive the energy is. In the limit case $\te=1$, the set of functions with finite energy is much larger and we will see that the results are of a different nature than in the case $0<\te<1$. 
\end{remark}

In~\cite[Theorem~I]{GM1}, we established the equivalence $\E(u)=0$ if and only if $u$ depends only on $x_1$ or only on $x_2$. That is, $u$ lies in the \emph{non convex} set
\[
\U(\Om):=\{u\in L(\Om) : \exists\,l\in\{1,2\},\ \exists\,u_l\in L(\Om_l) \text{ such that } u(x)=u_l(x_l) \text{ in }\Om \}.
\] 
We also established some quantitative versions of this fact by showing that $\E(u)$ controls the distance of $u$ to $\U(\Om)$ in a strong sense, see~\cite[Theorem~R $\And$ Theorem~S]{GM1}.  
 In the proofs of these results a key step is the control of the $\X_1\otimes \X_2$-valued distribution
\[
\mu[u]:=\nb_1\nb_2 u.
\] 
 We established~\cite[Proposition~M(a)]{GM1} that if $u\in L^\oo(\Om)$ has finite energy then $\mu[u]$ is a Radon measure with
\be\label{controlmu}
|\mu[u]|(\Om)\les \|u\|_\oo^{1-\te}\E(u).
\ee 
Obviously, the functions for which $\mu[u]=0$ are exactly the functions of the form $u_1(x_1)+u_2(x_2)$, that is the elements of $\Span (\U(\Om))$. The distribution $\mu[u]$ measures how much the function $u$ deviates from $\Span (\U(\Om))$. \medskip

The aim of this paper, is to study the  regularity and geometric structure of $\mu[u]$ that we call from now on: the \emph{defect measure}. As a by-product of our analysis, for $n=2$ and $\theta<1$, we are able to improve the quantitative results obtained in~\cite[Theorem~R]{GM1}.\\

Before going further let us recall a useful result from~\cite{GM1}.

\begin{proposition}[{\cite[Lemma~3.4, Remark~3.5]{GM1}}]
\label{prop:goodqk}
 If $u\in L^1_{loc}(\Om)$ is such that $\E(u)<\oo$, then up to a change of variables,  there exist:
\begin{enumerate}[(i)] 
\item sequences  $\eps_k\ge r_k>0$ tending to $0$;
\item  orthonormal bases $(e_1,\cdots,e_{n_1})$ of $\X_1$ and  $(f_1,\cdots,f_{n_2})$ of $\X_2$;  
\end{enumerate}
 such that\footnote{We use the notation $a\les b$ to indicate that there exists a constant $C>0$ which can depend only on $n_1$, $n_2$, $\te_1$, $\te_2$ and $\rho$, such that $a\le Cb$.}
\be\label{eq:goodqk}
\E'(u):=\limsup_{k\up \oo}\sum_{1\le i \le n_1,1\le j\le n_2} \int_{\Om^{\eps_k}} \dfrac{q(x,  r_k(e_i+f_j))}{ r_k^2} \,dx 
 \les  \E(u),
\ee
with the notation 
\be\label{defq}
 q(x,z):=\lt(|Du(x+ z_2,z_1)|^{\te_1} +|Du(x,z_1)|^{\te_1}\rt)\lt(|Du(x+z_1,z_2)|^{\te_2}+|Du(x,z_2)|^{\te_2}\rt).
\ee
As a consequence, if $u\in L^\oo$,
\be\label{eq:goodEk}
 \E''(u):=\liminf_{k\to \oo} \sum_{1\le i \le n_1,1\le j\le n_2} \int_{\Om^{\eps_k}} \dfrac{|D[Du(\cdot, r_k f_j)](x,r_k e_i) |}{ r_k^2} \,dx 
 \les \|u\|_\oo^{1-\theta}  \E(u).
\ee

\end{proposition}
Throughout the article we implicitly assume that the sequences $r_k$, $\eps_k$ and the bases $(e_1,\cdots,e_{n_1})$, $(f_1,\cdots, f_{n_2})$ are those provided by the proposition. There are however some exceptions where the symbols $\eps_k$, $r_k$ are used for other purposes (as in Lemma~\ref{lem:Diracmeasure}, Remark~\ref{remark_not_eps_kr_k_of_prop} or Proposition~\ref{prop:counterexamp}) but this is clear from the context.\\
When $n_1=n_2=1$, we can take $(e_1,f_1)$ to be the standard basis of $\R^2$ and we denote it by $(e_1,e_2)$.

\begin{remark}\label{rem:controlmu}
 In~\cite{GM1} we actually derive~\eqref{controlmu} from the stronger estimate, $ |\mu[u]|(\Om)\les \E''(u)$, see~\cite[Lemma~3.6]{GM1}.
\end{remark}

\subsection{The case $\theta<1$}
We first consider the case $\theta<1$ with $n_1=n_2=1$ so that $\Om=I_1\times I_2$ with $I_l$ open intervals. As shown in~\cite[Proposition~P]{GM1}, the typical example of a function with $\E(u)<\oo$ in this case is given by the characteristic function of a polyhedron with sides parallel to the axes (see Figure~\ref{Figure_polyhedron}, left). The defect measure is then a sum of Dirac masses sitting at the vertices of the polyhedron that lie in $\Om$.\\
In the setting of characteristic functions the situation simplifies in the sense that we only need the following consequence  of~\eqref{controlmu}.
\be\label{eq:diffinclu}
\mu[u]= \pt_1\pt_2 u\in \MM(\Om)\   \text{ and } \  u(x)\in \{0,1\}\  \text{ almost everywhere in } \Om.  
 \ee

\begin{figure}[h]
\begin{tikzpicture}[scale=1.1]
\begin{scope}
\clip (0,.1) rectangle (6.2,5.2);
\draw (1.5,4) node{$\Om$};
\draw[fill, color=gray!45] (3,5) -- (3,3) -- (6,3) -- (6,4) -- (4,4)  -- (4,5) -- cycle;
\draw (4,4) node[above right]{$x^7$};
\draw[fill, color=gray!45] (1,1) -- (5,1) -- (5,2) -- (3,2) -- (3,3) --(1,3) -- cycle;
\draw (1,1)  node {$\bullet$} node[below left]{$x^1$}  (5,1) node {$\bullet$} node[below right]{$x^2$} (5,2) node {$\bullet$} node[above right]{$x^3$} (3,2) node {$\bullet$} node[above right]{$x^4$} (3,3) node {$\bullet$} node[above left]{$x^5$}(1,3) node {$\bullet$} node[above left]{$x^6$};
\draw (2,2) node{$A$};

\draw (4,4) node{$\bullet$};
\draw[thin] (0.2,0.3) rectangle (6,5);
\end{scope}

\begin{scope}[xshift=.5\textwidth]
\clip (0,.1) rectangle (6.2,5.2);
\draw (1.5,4) node{$\Om$};
\draw[fill, color=gray!45] (1,1)   -- (5,1)  -- (5,2)  -- (3,2) -- (3,3) --(1,3) -- cycle;
\draw (2,2) node{$A'$};
\draw[fill, color=gray!45] (3,5) -- (3,3) -- (6,3) -- (6,4) -- (4,4)  -- (4,5) -- cycle;
\draw[fill, color=gray!45] (.2,.35) rectangle (6,.55) ;
\draw[fill, color=gray!45] (.2,.65) rectangle (6,.75) ;
\draw[fill, color=gray!45] (.2,.8) rectangle (6,.85) ;
\draw[fill, color=gray!45] (.2,.875) rectangle (6,.9) ;
\draw[fill, color=gray!45] (.2,.9125) rectangle (6,.925) ;
\draw[thin] (0.2,0.3) rectangle (6,5);
\end{scope}
\end{tikzpicture}
\caption{\label{Figure_polyhedron}A polygon $A$ with $\un_A$ of finite energy and with
$ \mu[\un_A]=\delta_{x^1}-\delta_{x^2}+\delta_{x^3}-\delta_{x^4}+2\delta_{x^5}-\delta_{x^6}-\delta_{x^7}$
and a set $A'$ with the same defect measure but with infinite perimeter.}
\end{figure}
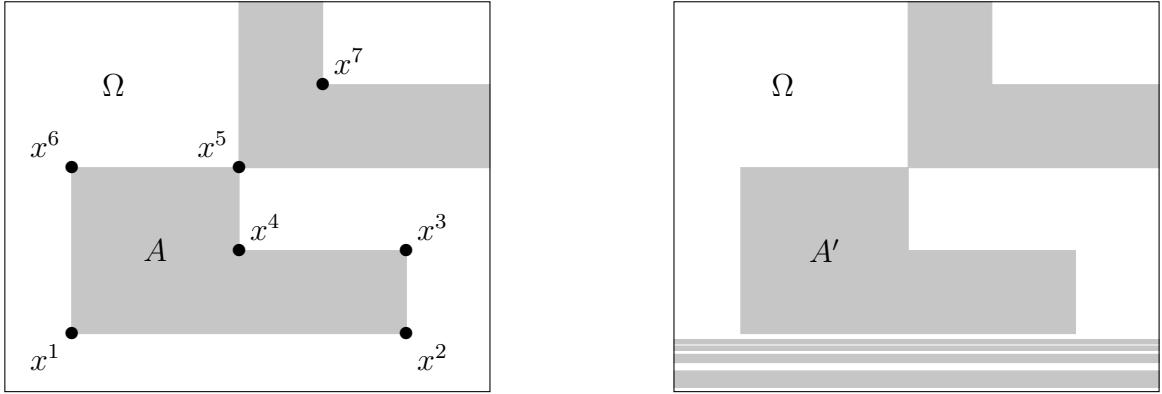

\begin{theorem}\label{theo:dirac1}
Let $n_1=n_2=1$, $\Om=I_1\times I_2\sub \R^2$ be a nonempty open box and let $A\sub \Om$ be measurable. If  $u=\un_A$  satisfies~\eqref{eq:diffinclu}  then the following properties hold.
 \begin{enumerate}[(i)]
 \item There exist finite sequences $m_1,\cdots,m_N\in\{\pm1,\pm2\}$ and $x^1,\cdots,x^N\in \Om$ with the $x^i$'s pairwise distinct such that
\[ \mu[u]=\sum_{j=1}^N m_j \delta_{x^j}.
 \]
 \item The set  $A$ is the union of finitely many polygons with sides parallel to the axes and stripes which are either all vertical or all horizontal. We can thus find measurable subsets  $A_1\sub I_1$, $A_2\sub I_2$ with either $\HH^0(\pt A_1)<\oo$ or $\HH^0(\pt A_2)<\oo$ and such that 
\[
u(x)=\un_{A_1}(x_1)\pm\un_{A_2}(x_2)+w(x)\qquad
\text{where}\quad 
w(x):=\mu[u]( (0,x_1]\times(0,x_2]).
\]
Besides $w=0$ (and $\mu[u]=0$) whenever  $|\mu[u]|(\Om)<1$ (see Figure~\ref{Fig:Lshape} in Section~\ref{sec:th<1d2} for an example with $u(x)=\un_{A_1}(x_1)-\un_{A_2}(x_2)$).

\item As a consequence, there exists $c>0$ such that if $\E(u)<c$ then $u\in \U(\Om)$, that is, up to a negligible set,
\[
A= I_1\times A'\quad \text{ or }\quad  A= A'\times I_2\quad \text{ for some measurable set } A'.
\] 
\end{enumerate}
\end{theorem}
\begin{remark}\label{rem1.5}
Notice that point (iii) improves~\cite[Theorem~R]{GM1} when $u$ is a characteristic function. Such result does not hold for functions which can take values into a set with more than two elements. For instance,  if for $l=1,2$, $A_l$ is a finite union of intervals with $\void\ne A_l\varsubsetneq I_l$, the function $u(x):=\un_{A_1}(x_1)+\un_{A_2}(x_2)$ satisfies $\mu[u]=\pt_1\pt_2 u\equiv 0$ (and even $0<\E(u)<\oo$) but $u\notin\U(\Om)$. 
\end{remark}
\begin{remark}
Let us stress that the theorem does not state that assumption~\eqref{eq:diffinclu} implies that $u$ is the characteristic function of a finite union of polygons. This is only true up to vertical or horizontal stripes since such stripes do not contribute to $\mu[u]$ (see Figure~\ref{Figure_polyhedron}). 
\end{remark}

We now turn to the case of generic functions $u$ when  $\te<1$. Unlike the case of characteristic functions where~\eqref{eq:diffinclu} is rigid, we need the stronger assumption $\E(u)<\oo$. 

\begin{theorem}\label{theo:dirac2}
Let $n_1=n_2=1$ and let $\Om=I_1\times I_2\sub \R^2$ be a nonempty open box. Assume that $\te<1$ and let $u\in L^\oo(\Om)$ with $\E(u)<\oo$.\\
Then $\mu[u]=\sum_{j\ge1} m_j \delta_{x^j}$ for some $x^j\in \Om$ and $m_j\in \R\bks\{0\}$ and we have the estimate
\be\label{boundai2d}
\sum_{j\ge 1} |m_j|^{\te}\les \E(u).
\ee
Moreover, if we assume that $u$ is integer-valued then the $m_j$'s are integers.
\end{theorem}
The  proof of Theorem~\ref{theo:dirac2} is based on the following observation. For every $r>0$ and almost every $x\in\Om$ such that $Q_{x,r}:=x+[0,r)^2\subset\Om$, there holds,
\begin{align*}
\mu[u](Q_{x,r})=& u(x+r(e_1+e_2))-u(x+r e_1)-u(x+r e_2)+u(x)\\
=&D[Du(\cdot, r e_2)](x,re_1).
\end{align*}
For a rigorous justification, see Lemma~\ref{lemma:4points}.  
By Proposition~\ref{prop:goodqk}, this shows that 
\[
 \liminf_{k\up \oo}\int_{\Om^{\eps_k}} \dfrac{|\mu[u]( Q_{x,r_k})|^{\te}}{r_k^2} \,dx\les \E(u).
\]
Lemma~\ref{lem:Diracmeasure} then implies that $\mu[u]$ is atomic and that the estimate~\eqref{boundai2d} holds true.
\begin{remark}\label{rem1.6}
We see from Remark~\ref{rem1.5} that $\mu[u]=0$ does not imply $\E(u)=0$. In particular we cannot expect the energy $\E$ to concentrate  on $\supp \mu[u]$.
\end{remark}

\smallskip

Building on the structure of $\mu[u]$ provided by the theorem, we are able to  improve~\cite[Theorem~R]{GM1}. In order to state the result let us introduce the set $SBV_\te(\Om)$ as the subset of functions of bounded variation  in $\Om $ (see~\cite{Am_Fu_Pal}) whose distributional derivative has only jump part, \ie $\nb u= (u_+-u_-) \nu_u \HH^{n-1}\restr J_u$, and such that 
\[
 |\nb  u |_\te:=\int_{J_u} |u_+-u_-|^\te\,d\HH^{n-1}<\oo.
\]  
\begin{theorem}\label{thmBVtheta}
Assume that $n_1=n_2=1$, that $\Om=I_1\times I_2$ is an open box and that $\te<1$. Let $u\in L^\oo(\Om)$ with $\E( u)<\oo$ and $\|u\|_\oo\le 1$.
\begin{enumerate}[(i)]
\item   There exists $\bar u\in \U(\Om)$  such that $u-\bar u\in SBV_\te(\Om)\cap L^\oo(\Om)$ with the estimate
\be\label{quantrig2dtheta}
\|u-\bar u\|_\oo+|\nb[u-\bar u]|_\te(\Om) \les (1+ \HH^1(I_1)+ \HH^1(I_2))\lt(\E( u)+ \sqrt{\E( u)\,}\rt).
\ee
\item  If moreover $\mu[u]=0$ and $u\notin \U(\Om)$, then there exist two non-constant functions  $u_l\in SBV_{\theta_l}(I_l)$ for $l=1,2$  such that $u(x)=u_1(x_1)+u_2(x_2)$ and 
\be\label{quantrig2dtheta2}
|\nb u_1|_{\te_1}(I_1)\, |\nb u_2|_{\te_2}(I_2)\,\les\,\E(u).
\ee  
\item As a consequence, if $u\in L^\oo(\Om)\cap BV(\Om)$ is such that $\E(u)<\oo$ and $\nb u$ has no jump part then $u\in \U(\Om)$.
\end{enumerate}
\end{theorem}

We now turn to the  higher dimensional case $n\ge 3$.  Our main result is the $(n-2)$- rectifiability of $\mu[u]$, more precisely, its $(n_1-1,n_2-1)\,$-tensor rectifiability.

\begin{theorem}\label{coromurect}
Assume $\te<1$ and $u\in L^\oo(\Om)$ with $\E(u)<\oo$ then the following hold.
\begin{enumerate}[(i)]
\item $\mu[u]$ is a $(n-2)\,$-rectifiable measure, \ie there exist a $(n-2)\,$-rectifiable set $\Sigma\sub \Om$ and a Borel mapping $m:\Om\to \R$  such that:\\ 
($*$) the approximate tangent space at $\HH^{n-2}\,$-almost every $x\in\Sigma$ is of the form $(\nu_1(x),\nu_2(x))^\perp$ where $\nu_l(x)\in \X_l\sm\{0\}$ for $l=1,2$,\\
($*$) \[
\mu [u]= m (\nu_1\otimes\nu_2)\, \HH^{n-2}\restr \Sigma.
\]
\item We have the estimate, 
\[
\M_\theta(\mu[u]):=\int_{\Sigma} |m|^\te\,d\HH^{n-2}\ \les\ \E(u). 
\]
\item We can choose $\Sigma$ such that $\Sigma\sub \Sigma_1 + \Sigma_2$ for some $(n_l-1)\,$-rectifiable subsets $\Sigma_l\sub\Om_l$ for $l=1,2$. We say that $\mu[u]$ is $(n_1-1,n_2-1)$-tensor rectifiable.  
\end{enumerate}
\end{theorem}
The main two observations in the proof of (i) and (ii) in Theorem~\ref{coromurect} are the following. First, we can identify the defect measure $\mu[u]$ with a $(n-2)\,$-current $T[u]$ (see Proposition~\ref{equivTmu}). Moreover, since $\mu[u]=\nb_1\nb_2 u$ is a finite measure, we get that $T[u]$ is a finite mass cycle, \ie $\pt (T[u])=0$ and $\M(T[u])<\oo$.\\
 We then argue by slicing (Theorem~\ref{thm:thetamass}) and apply the slicing rectifiability criterion of White~\cite{White1999-2}, see also~\cite{Jerrard}. For this we have to show that that the $0\,$-slices of $T[u]$ are rectifiable. 
 \begin{enumerate}[($*$)]
 \item  We first notice that as a consequence of the formula $\mu[u]=\nb_1\nb_2 u$, all the  slices of $T[u]$ with respect  to  coordinate $(n-2)\,$-spaces orthogonal to a plane of the form $\Span(e_{i_1},e_{i_2})$ or $\Span(f_{j_1},f_{j_2})$ vanish. 
 \item Next, we consider instead a $2\,$-plane of the form $\Span(e_i,f_j)$. On the one hand, slicing and partial boundary operations commute, see~\eqref{commuteslice}. On the other hand, by Fubini and Fatou, the energy $\E''(u)$ (recall~\eqref{eq:goodEk}) also behaves well with respect to slicing. We may thus apply the two dimensional result Theorem~\ref{theo:dirac2} to conclude that the $0\,$-slices of $T[u]$ are rectifiable. 
 \end{enumerate}
    Once rectifiability is obtained, (ii) follows from the corresponding bound~\eqref{boundai2d} in the case $n_1=n_2=1$. The proof of (iii) is considerably more involved  and motivated the development of the theory of tensor-rectifiable flat chains in~\cite{GM_tfc}. In a nutshell, the idea of the proof of~\cite[Theorem~1.3]{GM_tfc} is to first treat the case $n_1=1$, $n_2\ge 2$. In this case, based on the decomposition of $T[u]$ in indecomposable components proven in~\cite{GM_decomp} we obtain in~\cite[Proposition~6.2]{GM_tfc} the stronger result below (rephrased in the language of the present paper),
\begin{proposition}\label{prop:n1=1}
 Let $n_1=1$ and $n_2\ge 2$ so that $\Om=I_1+\Om_2$ for some open interval $I_1$. Assume that $\te<1$ and that $u\in L^\oo(\Om)$ is such that $\E(u)<\oo$.\\ 
 Then there exist  sequences $y^j_1\in I_1$ and $u_2^j\in SBV^\theta(\Om_2)$ for $j\ge1$ such that 
 \be\label{eq:mun1=1}
  \mu[u]=\sum_{j\ge1} \delta_{x^j_1}\otimes \nb_2 u_2^j
 \ee
 and 
 \be\label{Mthetan1=1}
  \M_\theta(\mu[u])=\sum_{j\ge1} |\nb_2 u_2^j|_\theta(\Om_2).
 \ee
 As a consequence, there exist $u^0_1\in L^\oo(I_1)$ and $u^0_2\in L^\oo(\Om_2)$ such that 
 \[
 u(x)= u^0_1(x_1)+u^0_2(x_2) +\sum_{j\ge1}\un_{(-\oo, x^j_1)}(x_1)\,  u_2^j(x_2).
\]
\end{proposition}
In the general case $n_1, n_2\ge 2$, we can  formally interpret $T[u]$ as a $(n_1-1)\,$-flat chain over $\Om_1$ with coefficients in the infinite dimensional space of $(n_2-1)\,$-flat chains over $\Om_2$. By~\eqref{eq:mun1=1} we see that if we slice this flat-chain with respect to a hyperplane of $\X_1$, we obtain a $0\,$-rectifiable flat chain in $\Om_1$ (still with coefficients in the space of $(n_2-1)\,$-flat chains over $\Om_2$). Applying White's rectifiability criterion then concludes the proof.
\begin{remark}\label{rem:counterdecomp}
 Let us point out that for $n_1, n_2\ge 2$ we cannot expect a decomposition analog to~\eqref{eq:mun1=1} satisfying also the counterpart of identity~\eqref{Mthetan1=1} (see the counterexample of~\cite[Proposition~6.5]{GM_tfc}).
\end{remark}

\begin{remark}\label{rem:hmass}
 The energy $\M_\theta$ coincides with a so-called $h\,$-mass studied for instance in~\cite{White1999-1,CdRMS2017}. Partly due to its connection with branched transport models, this type of functionals has received a lot of attention in the past few years, see \textit{e.g.}~\cite{BraWir,CFM2019a,colombo2021well}. It would be interesting to understand further this connection with the energy $\E(u)$.
\end{remark}


\begin{remark}\label{rem:DeP}
 Let us finally observe that since $\mu=\mu[u]=\nb_1\nb_2 u$, it satisfies the linear PDE constraints
\[\nb_1\times \mu=0 \qquad \text{and} \qquad \nb_2\times \mu=0,\]
where for instance $\nb_1\times \mu=0$ means that for every $i,k\in[1,n_1]$ and every $j\in [1,n_2]$,
\[
 \dfrac{\pt \mu_{i,j}}{\pt e_k}= \dfrac{\pt \mu_{k,j}}{\pt e_i}.
\]
Letting $\mathbb{A}$ be the associated symbol, we have for $\xi=\xi_1+\xi_2\in \R^n$ with $\xi_i\neq 0$
\[\ker(\mathbb{A}(\xi))=\R\, \xi_1\otimes \xi_2\]
and for $\xi_1\neq 0$, $\ker(\mathbb{A}(\xi_1))= \xi_1\otimes \X_2$ and similarly for $\ker(\mathbb{A}(\xi_2))$. In the language of~\cite{DePal}, we thus find $\Lambda_{\mathcal{A}}^{n-2}=\{0\}$ so that  we could appeal to~\cite{DePRin,DePal} and obtain that the most singular part of $\mu$ is $(n-2)\,$-rectifiable. In comparison with Theorem~\ref{coromurect} this would however not exclude the presence of a more diffuse part of the measure.
\end{remark}

\subsection{The case $\theta=1$ (with $n_1=n_2=1$)}

As proven in~\cite[Proposition~P]{GM1}, when $\te=1$ the set of possible measures $\mu[u]$ is much richer. The study being much more difficult, we restrict ourselves to the case $n_1=n_2=1$. 
Indeed, besides atomic measures we can also have measures concentrated on lines (or a mixture of both). The typical example is given by the ``roof'' function $u(x_1,x_2):= \min(x_1,x_2)$. It has finite energy (when restricted to cubes)  and $\mu[u]=(1/\sqrt2)\HH^1\restr L$ where $L=\{x_1=x_2\}$ is the diagonal.  It turns out that  $\mu[u]$ cannot be more diffuse. We establish this through a far-reaching refinement of the method used for Theorem~1.8 (but now in the case $\theta=1$).
\begin{theorem}
\label{theo:lines}
Let $n_1=n_2=1$ and  $\te=1$. Let  $u\in L^1_{loc}(\Om)$  be such that $\E(u)<\oo$. Then $\mu[u]=\mu[u]\restr A$ for some Borel subset $A\sub \Om$ such that $\HH^1\restr A$ is $\sigma\,$-finite.
\end{theorem}
\begin{remark}\label{rq:alldim}
 Measures of all dimensions between zero and one are possible. Indeed, fix  $\nu$ a finite measure on $(-1,1)$ and, for $l=1,2$, let $y_l:(-1,1)\to(-1,1)$  be a smooth and  increasing function. Setting for $x=(x_1,x_2)\in(-1,1)^2$,
\[
u(x) :=\int_{-1}^1\un_{(y_1(t),1)}(x_1)\un_{(1,y_2(t))}(x_2)\, d\nu(t),
\]
then $u$ has finite energy and $\mu[u]=(y_1\otimes y_2)\pf\nu$.
\end{remark}

From the examples discussed above it seems then natural to restrict ourselves to the case of Lipschitz functions. Our last main result states that for such functions, the defect measure is $1\,$-rectifiable.
\begin{theorem}\label{thm:rectiftheta=1}
 Let $n_1=n_2=1$ and $\theta=1$. Let $u\in L^\oo(\Om)$ be such that $\|\nb u\|_\oo\le 1$ and $\mu:=\mu[u]$ is a Radon measure. Then $\mu$ is $\HH^1$-rectifiable, that is 
 \[
 \mu=m \HH^1\restr \Sigma
 \]
 for some $1\,$-rectifiable set $\Sigma\sub \Om$ and some Borel measurable function $m:\Sigma\to \R$.\\
 Moreover, $\nb u$ has strong traces on $\Sigma$ of the form 
  \be\label{eq:tracesintro}
  (v^\oo_1,0) \qquad  \text{and } \qquad (0,v^\oo_2)
 \ee
with $v^\oo_1\ne0$, $v^\oo_2\ne0$ and
\[
 |m|=\dfrac{|v^\oo_1| |v^\oo_2|}{\sqrt{|v^\oo_1|^2+ |v^\oo_2|^2}}.
\]
\end{theorem}

For the proof of Theorem~\ref{thm:rectiftheta=1}, we first show in Proposition~\ref{prop:differincl} that if $u$ satisfies the hypothesis of the theorem then $v=\nb u$ satisfies the differential inclusion~\eqref{eq:difincl1} with the additional constraints that $|v|\le 1$ and that $\mu[v]$ is a measure. We then prove in Theorem~\ref{theo:rectifLipmain} the analog of Theorem~\ref{thm:rectiftheta=1} but for  bounded vector fields $v$ satisfying~\eqref{eq:difincl1} and $\mu[v]\in \MM(\Om)$. Writing $v=\nb u$ for some  Lipschitz function $u$, the idea is to use the layer-cake formula to decompose it on its superlevel-sets $\om_t:=\{u>t\}$. Defining
\[
 \kappa_t:=\mu[\un_{\om_t}]=\pt_1\pt_2 \un_{\om_t},
\]
we show in Proposition~\ref{prop:decompmu} that
\[
 \mu[v]=\int \kappa_t\,dt.
\]
In particular, for almost every $t$, $\kappa_t$ is a finite measure. Therefore, $\un_{\om_t}$ satisfies the differential inclusion~\eqref{eq:diffinclu}. Applying Theorem~\ref{theo:dirac1} we conclude that $\om_t$ is a finite union of polygons with sides parallel to the axes and $\kappa_t$ is nothing else than the sum of the Dirac masses located at the corners. Moreover, $|\mu[v]|\,$-almost every point $\bar x$ is given by such a corner corresponding to a level $\bar t$ at which the function $t\mapsto |\om_t|$ is continuous. As a consequence, up to discarding sets of small measures, for $t$ close to $\bar t$, every $\om_t$ contains exactly one corner $x(t)$ in a small neighborhood of $\bar x$, see Lemma~\ref{lem:corner}. This gives a sort of  local parametrization of $\mu[v]$ around $\bar x$. The main point is then to prove that $x$ is differentiable in an appropriate sense at $t=\bar t$. 
The central insight is that since $u(x(t))=t$, the velocity $x'(t)$ is governed by $\nb u$. We prove in Lemma~\ref{lem:h_i'} that, with the notation~\eqref{eq:tracesintro},
\[
 x'_l(\bar t\,)= \dfrac1{v_l^\infty(x(\bar t\,))}\qquad\text{ for }l=1,2.
\]

\begin{remark}
We call entropy any mapping $\Phi:\R^2\to\R^2$ which writes as $\Phi(v)=\vhi_1(v_2) e_1 +\vhi_2(v_1) e_2$ for some pair of smooth functions $\vhi_1,\vhi_2:\R\to \R$. We say that $v$ is an entropy solution to~\eqref{eq:difincl1}
if for every entropy $\Phi$,
\[
 \mu_\Phi:=\nb\cdot [\Phi(v)]\in \MM(\Om).
\]
It is then not hard to see that  if $v$ satisfies~\eqref{eq:difincl1} it is an entropy solution if and only if  $\mu[v]:=\mu_{\Phi^0}$ is a measure for the entropy $\Phi^0$ associated with $(\vhi_1^0,\vhi_2^0)=(\text{Id},0)$. In this light our result can be compared with~\cite{DelOt} where a similar question is addressed for entropic solutions of the eikonal equation. This corresponds to replacing 
\[
K\cap B_1=([-1,1]\times\{0\}) \cup(\{0\}\times[-1,1])
\] by $\pt B_1$ as the nonlinear constraint. Thanks to the strong constraint on the level-sets of the stream function $u$ the analysis in our case turns out to be substantially simpler than for~\cite{DelOt}. In particular we are able to obtain a stronger result which is the rectifiability of the defect measure $\mu[v]$. This is still a major open problem for the eikonal equation, see~\cite{Marconi2,Marconi1} for recent results for related models.
\end{remark}
Partly motivated by this analogy with the Aviles-Giga functional, we investigate in Section~\ref{sec:compactness} the compactness properties of sequences in 
\[
S^\oo(\Om):= \{v\in L^\oo(\Om,\R^2) : \|v\|_\oo\le 1,\  \nb\times v=0,\ v\in K \text{ \ae\,and }\,\mu[v]\in \MM(\Om)\}.
\]
\begin{proposition}\label{prop:compactnessintro}
 If $v^k\in S^\oo(\Om)$ is such that 
 \[
  \sup_k |\mu[v^k]|(\Om)<\oo,
 \]
\begin{enumerate}[(i)]
\item Then, up to extraction, $v^k\longto v$ for some $v\in S^\oo(\Om)$ in the weak-$*$ topology of $L^\oo$.
\item  If moreover
\[
 \limsup_{k\to \oo} |\mu[v^k]|(\Om)=0
\]
then for some $l\in \{1,2\}$, $v= v_l e_l$ with $v_l(x)=v_l(x_l)$ and $v_{\bar l}^k$ converges strongly to $0$ in $L^1$ (here $\{l,\bar l\}=\{1,2\}$).
\end{enumerate}
\end{proposition}
The proof follows the same scheme as in~\cite{DKMO01} using the div-curl lemma and Young measures. \bigskip

For $x\in \R^2$ and $r>0$ we set $Q_r(x):=x+(-r,r)^2$ and simply write $Q$ for $(-1,1)^2$. Returning to the setting (and to the notation) of Theorem~\ref{thm:rectiftheta=1}, we introduce for $x^*\in\Om$ the family of blow-ups of $v:=\nb u$ at $x^*$ as the functions $v^{x^*,r}\in S^\oo(Q)$  defined for $0<r<d(x^*,\R^2\sm\Om)$ by
\[
v^{x^*,r}(x):=v(x^*+r x)=\nb u(x^*+r x)\qquad\text{for }x\in Q.
\]
We notice that for every Borel subset $A\sub Q$ we have
\[
\mu\lt[v^{x^*,r}\rt](A)=\dfrac1r\mu[v]({x^*+r A})
\]
so that the assumption $|\mu|(\Om)<\oo$ implies by~\cite[Theorem~2.56]{Am_Fu_Pal} that for $\HH^1$-almost every $x^*\in\Om\sm\Sigma$ there holds
\[
\limsup_{r\dw0} |\mu[v^{x^*,r}]|(Q)=\limsup_{r\dw0} \dfrac{|\mu[v]|(Q_{r}(x^*))}r=0.
\]
Applying Proposition~\ref{prop:compactnessintro} to the family $v^{x^*,r}\in S^\oo(Q)$ we deduce that for $\HH^1$-almost every $x^*\in\Om\sm\Sigma$, there exists $l\in \{1,2\}$ such that, \emph{up to extraction} of a subsequence $r_k\dw0$, there hold
\be\label{compactness_intro}
v^{x^*,r_k}_l\ \text{ converges weakly-$*$ in }L^\oo(Q)\qquad\ \text{ and }\ \qquad v^{x^*,r_k}_{\bar l}\to0\text{ in }L^1(Q).
\ee
These points parallel the ``VMO points''  of~\cite{DelOt} in the setting of the eikonal equation which are conjectured to be $\HH^1$-almost all Lebesgue points of $v=\nb u$ (see~\cite{lamy2022lebesgue} for recent results in this direction).\\
A natural question is whether we can choose the integer $l=1,2$ in~\eqref{compactness_intro} independently of the choice of the subsequence.  In such a case, $x^*$ would be a Lebesgue point of $v_{\ov l}$. The answer is no in general. Indeed, we construct in Proposition~\ref{prop:counterexamp} a vector field $v\in S^\oo(Q)$ satisfying
\[
\limsup_{r\dw0} \dfrac{|\mu[v]|(Q_r(0))}r=0
\]
but such that the integer $l\in\{1,2\}$ in~\eqref{compactness_intro} does depend on the choice of the subsequence. In particular 0 is neither a Lebesgue point of  $v_1$ nor of $v_2$.

\subsection{Conventions and notation}
In all the paper, we consider $\te_1,\te_2>0$, we note $\te=\te_1+\te_2$ their sum and we assume that $\te\le1$. For $x,z\in \R^n$ and $u:\Om\to\R$ we define 
\[
Du(x,z):=u(x+z)-u(x).
\]  
We denote by $(e_1,\cdots, e_{n_1})$ (respectively $(f_1,\cdots, f_{n_2})$)  an orthonormal basis of $\X_1$ (respectively of $\X_2$). \\
If $n_1=n_2=1$,  for $x\in\R^2$ and $z\in(0,+\oo)^2$, we write
\[
Q_{x,z}:=[x_1,x_1+z_1)\times [x_2,x_2+z_2).
\]
We also use the notation  $Q_{x,r}:= x+ [0,r)^n$  for $x\in \R^n$ and $r>0$. \\
The open ball in $\R^n$ with radius $r>0$ and centered at $x$ is denoted $B_r(x)$ and we write $B_r$ when $x=0$. The dimension $n$ is always clear from the context.\\
Given $\mathcal{S}\subset \X$ where $\X$ is some vector space, $\Span(\mathcal{S})$ is the space spanned by $\mathcal{S}$.
We denote $\HH^k(A)$ the $k$-dimensional Hausdorff measure of a set $A\sub\R^n$ and if $A\sub \R^n$ is a measurable subset, we note $|A|$ its volume.\\
We write  $a\les b$ when  $a\le C b$ for some  $C>0$ which may only depend on $\te_1,\te_2$, $n$ or on the kernel $\rho$.\\
We use standard notation for functional spaces such as $L^p(\om)$, $W^{1,p}(\om)$, $BV(\om)$. $\MM(\om)$ is the space of Radon measures on $\om$.\\
Unless otherwise specified the sequences are indexed from $1$. We often write $\sup a_j$  as a shortcut for $\sup_{j\ge1}a_j$ and similarly for series, we write $\sum a_j$ for $\sum_{j\ge1} a_j$.

\subsection{Outline of the paper}
The paper is organized as follows. In Section~\ref{sec:th<1} we consider the case $\theta<1$ and start by treating the two dimensional case in Section~\ref{sec:th<1d2}. We first prove Theorem~\ref{theo:dirac1} on the structure of the defect measures for characteristic functions. We then consider the case of arbitrary functions, \ie Theorem~\ref{theo:dirac2}. Finally, we use it to prove Theorem~\ref{thmBVtheta}. In Section~\ref{S3} we consider the higher dimensional case and prove Theorem~\ref{coromurect}. In Section~\ref{sec:theta=1} we turn to the case $\theta=1$. We first prove Theorem~\ref{theo:lines} for arbitrary functions and in Section~\ref{sec:Lip}, we consider the case of Lipschitz functions. We derive the differential inclusion satisfied by functions of finite energy in Section~\ref{sec:DerivdifLip} before proving Theorem~\ref{theo:rectifLipmain} about the rectifiability of the defect measure in Section~\ref{sec:rectifLip}. We finally prove the compactness result, Proposition~\ref{prop:compactnessintro} in Section~\ref{sec:compactness}.

\section{The case $\theta<1$}\label{sec:th<1}
\subsection{The two-dimensional case}\label{sec:th<1d2}
In this section we assume that $n_1=n_2=1$.
We first show the following simple lemma used in several places. 
\begin{lemma}\label{lemma:4points}
Let $\Om= \mathring{Q}_{\bar x,\bar z}\sub \R^2$ be a nonempty open box (that is $\bar x\in\R^2$ and $\bar z\in(0,+\oo)^2$) and let $u\in L^1_{loc}(\Om)$ be such that $\mu[u]=\pt_1\pt_2 u$ is a Radon measure.\\
Defining $w$ as 
\be\label{defw}
 w(x):=\mu((\bar x_1, x_1]\times(\bar x_2, x_2]) \qquad \text{ for } x\in \Om,
\ee
we have $w\in BV(\Om)$ with the estimate
\be\label{estimw}
 \|w\|_\oo+ |\nb w|(\Om)\les (1+ |\bar z|)|\mu[u]|(\Om),
\ee
and there exist functions $u_1(x_1)$ and $u_2(x_2)$ in $\U(\Om)$ (which are bounded if $u$ is bounded) such that 
\be\label{decompu}
 u(x)=u_1(x_1)+u_2(x_2)+w(x) \qquad \text{ for  almost every } x\in \Om.
\ee
As a consequence, for every $z\in (0,+\oo)^2$,  we have for almost every $x\in \Om$ with $Q_{x,z}\sub \Om$, 
\be
\label{id4pts}
\mu[u](Q_{x,z})=D[Du (\cdot, z_2)](x,z_1) \qquad \text{ and } \qquad |\mu[u]|(\pt Q_{x,z})=0.
\ee
In particular,  $|\mu[u](Q_{x,z})|\le 4\|u\|_\oo$.
\end{lemma}
\begin{proof}~
Estimate~\eqref{estimw} is a direct consequence of the definition of $w$. Since $\pt_1\pt_2 (u-w)=0$,~\eqref{decompu} is also readily obtained (see for instance the proof of~\cite[Theorem~3.7]{GM1}). We thus only need to show~\eqref{id4pts}.

Let $I_l:=(\bar x_l,\bar x_l+\bar z_l)$ for $l=1,2$ (so that $\Om=I_1\times I_2$).
 Since $\mu:=\mu[u]$ is a Radon measure, there exist $J_1\sub I_1$ and $J_2\sub I_2$ with full Lebesgue  measures such that for every $x_1\in J_1$, $|\mu|(\{x_1\}\times I_2)=0$ and for every $x_2\in J_2$, $|\mu|(I_1\times\{x_2\})=0$. 
 Hence, for every box $Q$ with vertices  in $J_1\times J_2$, we have 
$|\mu|(\pt Q)=0$. We may assume moreover that all the points of $J_l$ are Lebesgue points of $u_l$ for $l=1,2$. \\
Then, using~\eqref{decompu}, for every box $Q_{x,z}$ with vertices in $J_1\times J_2$ we have
\begin{align*}
D[Du (\cdot, z_2)](x,z_1)
&=D[Dw (\cdot, z_2)](x,z_1)\\
&=w(x)-w(x_1,x_2+z_2)-w(x_1+z_1,x_2)+w(x+z).
\end{align*}
By definition of $w$ this implies~\eqref{id4pts}. Eventually, if  $z\in (0,+\oo)$ is fixed then almost every box  $Q_{x,z}\sub \Om$ has vertices in $J_1\times J_2$. This concludes the proof.
\end{proof}

Let us prove Theorem~\ref{theo:dirac1} which deals with the case of characteristic functions. 
\begin{proof}[Proof of Theorem~\ref{theo:dirac1}]~
Let $u=\un_A$ be such that $\mu:=\pt_1\pt_2 u\in \MM(\Om)$ as in the theorem. 
\smallskip

(i) We first prove that $\mu$ is atomic. Notice that since $u\in\{0,1\}$,~\eqref{id4pts} implies that for every fixed $z$, Lebesgue almost every  rectangle $Q_{x,z}\sub \Om$ satisfies
\be\label{-2,2}
\mu(Q_{x,z})\in\{0,\pm1,\pm2\}
\ee
and by approximation this actually holds for every $Q_{x,z}\sub \Om$.
Therefore, if $m\delta_x$  is a non trivial atom of $\mu$, we deduce that $m=\mu(\{x\})=\lim_k \mu(Q_{x^k,z^k}) \in\pm \{1,2\}$,
where we take the limit over a sequence of boxes $Q_{x^k,z^k}$ with $x\in Q_{x^k,z^k}$ and  $|z^k|\to0$.
Since $\mu$ is a finite measure on $\Om$, its atomic part writes as a finite sum $\mu^a=\sum_{j=1}^Nm_j\delta_{x^j}$ with $m_j\in\{\pm1,\pm2\}$, $x^j\in \Om$ and  $N\le\sum  |m_j| \le|\mu|(\Om)$. \\
Let us prove that $\mu(Q_{x,z})=0$ for every rectangle $Q_{x,z}\sub \Om':=\Om\bks \supp\mu^a$.
Let us assume by contradiction that $\mu(Q_{x,z})\neq 0$ so that $\mu(Q_{x,z})\in\{\pm1,\pm2\}$. We split $Q_{x,z}$ into four equal disjoint
rectangles $Q_{x^1,z^1},\cdots,Q_{x^4,z^4}$.
By~\eqref{-2,2}, we have $\mu(Q_{x^j,z^j})\in \pm\{0,1,2\}$ for $j=1,\cdots,4$. Since  $\sum  \mu(Q_{x^j,z^j})=\mu(Q_{x,z})\neq0$,  we can pick $j$ with $\mu(Q_{x^j,z^j})\in \{\pm1,\pm2\}$. 
We note  $Q^1:=Q_{x^j,z^j}$ and iterate the construction to produce a sequence of nested boxes $Q^k$ with $\mu(Q^k)\in\{\pm1,\pm2\}$ for $k\ge1$ and $\diam Q^k\to 0$. We deduce from the monotone convergence theorem that $\cap_k Q^k=\{x^*\}$ for some $x^*\in Q_{x,z}$ with $\mu(\{x^*\})\ne0$. This contradicts $Q_{x,z}\cap \supp\mu^a=\void$ and proves the claim. Eventually, since the rectangles of the form $Q_{x,z}\sub \Om'$ 
generate the $\sigma\,$-algebra of Borel sets of $\Om'$, we conclude that $\mu\restr\Om'=0$ so that $\mu=\mu^a$ as claimed and (i) is proved.

\smallskip 
(ii) We turn to the proof of  (ii). Without loss of generality, we may assume that $I_l=(0,\ell_l)$ for $l=1,2$. \\
We start with the following simple observation. If $Q_{0,z}\sub \Om$ is such that $\mu\restr Q_{0,z}=0$, then  $w\restr Q_{0,z}=0$ (recall the definition~\eqref{defw}) and thus by~\eqref{decompu},
\[
 u(x)=u_1(x_1)+u_2(x_2) \qquad \text{for }x\in Q_{0,z}
\]
for some functions $u_1$, $u_2 \in \U(Q_{0,z})$. However, since $u$ takes only two values this implies that $u\in \U(Q_{0,z})$ (see~\cite{DolMul} for an application of this argument in a different context). 
Notice that since $\mu$ is quantized this implies in particular (iii).\\
Let $z$ be such that $\mu$ does not contain any Dirac mass in the $L\,$-shaped domain $\mathcal{L}:=Q_{0,(\ell_1,z_2)}\cup Q_{0,(z_1,\ell_2)}$, see Figure~\ref{Fig:Lshape}. 
\begin{figure}[h]
 \begin{center}
\begin{tikzpicture}[scale=1]
\pgfmathsetmacro{\L}{10}
\pgfmathsetmacro{\l}{6}
\pgfmathsetmacro{\xa}{.18}
\pgfmathsetmacro{\xb}{.41}
\pgfmathsetmacro{\xc}{.52}
\pgfmathsetmacro{\xd}{.68}
\pgfmathsetmacro{\xe}{.88}
\pgfmathsetmacro{\ya}{.24}
\pgfmathsetmacro{\yb}{.45}
\pgfmathsetmacro{\yc}{.68}
\pgfmathsetmacro{\yd}{.88}
\draw[very thin] (0,0) rectangle (\L,\l);

\draw[fill, color=gray!45]   (0,0) rectangle (\xa*\L,\ya*\l);
\draw[fill, color=gray!45]   (0,\yb*\l) rectangle (\xa*\L,\yc*\l);
\draw[fill, color=gray!45]   (0,\yd*\l) rectangle (\xa*\L,\l);
\draw[fill, color=gray!45]   (\xa*\L,\ya*\l) rectangle (\xb*\L,\yb*\l);
\draw[fill, color=gray!45]   (\xa*\L,\yc*\l) rectangle (\xb*\L,\yd*\l);
\draw[fill, color=gray!45]   (\xb*\L,0) rectangle (\xc*\L,\ya*\l);
\draw[fill, color=gray!45]   (\xb*\L,\yb*\l) rectangle (\xc*\L,\yc*\l);
\draw[fill, color=gray!45]   (\xb*\L,\yd*\l) rectangle (\xc*\L,\l);
\draw[fill, color=gray!45]   (\xc*\L,\ya*\l) rectangle (\xd*\L,\yb*\l);
\draw[fill, color=gray!45]   (\xc*\L,\yc*\l) rectangle (\xd*\L,\yd*\l);
\draw[fill, color=gray!45]   (\xd*\L,0) rectangle (\xe*\L,\ya*\l);
\draw[fill, color=gray!45]   (\xd*\L,\yb*\l) rectangle (\xe*\L,\yc*\l);
\draw[fill, color=gray!45]   (\xd*\L,\yd*\l) rectangle (\xe*\L,\l);
\draw[fill, color=gray!45]   (\xe*\L,\ya*\l) rectangle (\L,\yb*\l);
\draw[fill, color=gray!45]   (\xe*\L,\yc*\l) rectangle (\L,\yd*\l);
\draw[very thick, dashed] (0,0) rectangle  (\xa*\L,\ya*\l);
\draw[very thick] (0,0) -- (\L,0) -- (\L,\ya*\l) -- (\xa*\L,\ya*\l) -- (\xa*\L,\l) -- (0,\l) -- cycle;
\draw (\xa*\L,\ya*\l) node[below right]{$z$};
\draw (0,\ya*\l) node[left]{$(0,z_2)$};
\draw (\xa*\L,0) node[below]{$(z_1,0)$};
\draw (\L,0) node[below]{$(l_1,0)$};
\draw (0,\l) node[left]{$(0,l_2)$};
\draw (0,0) node[left]{$0$};
\draw  (\L*1.11, .5*\ya*\l) node {$\Bigg\}\!\!\!\!\!\Bigg\}\, Q_{0,(l_1,z_2)}$};
\draw (.5*\xa*\L,\l+0.085*\L) node {$Q_{0,(z_1,l_2)}$};
\draw(.5*\xa*\L,\l+0.025*\L) node {$\overbrace{\phantom{a\,aaaaaa}}$};   
\end{tikzpicture}
\caption{\label{Fig:Lshape} The $L\,$-shaped domain $\mathcal{L}=Q_{0,(\ell_1,z_2)}\cup Q_{0,(z_1,\ell_2)}$. Here  $u=1$ on the gray zones.} 
\end{center}
\end{figure}
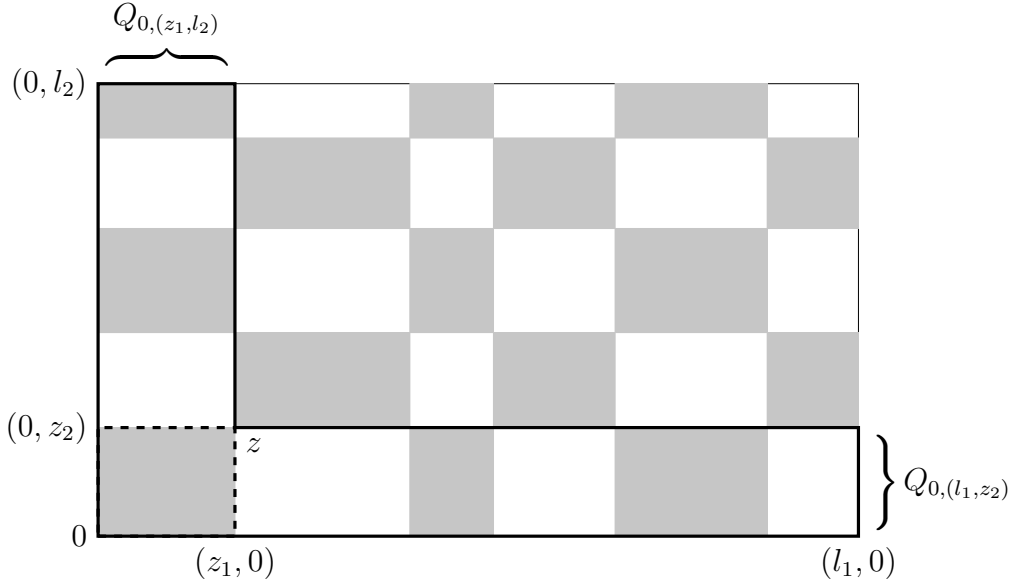
 By the above observation applied in $Q_{0,(\ell_1,z_2)}$ and then in $Q_{0,z_1,\ell_2}$ and using obvious notation, there exist $u\sh_1,u\sh_2\in\U(Q_{0,(\ell_1,z_2)})$ and $u\sv_1,u\sv_2\in\U(Q_{0,z_1,\ell_2})$ such that\footnote{Throughout the article, the superscripts $\_\sh$ and $\_\sv$ stand for ``horizontal'' and ``vertical''.} 
\begin{align*}
u(x)&=u_1\sh(x_1)+u_2\sh(x_2)\qquad\text{for }x\in Q_{0,(\ell_1,z_2)},\\
u(x)&=u_1\sv(x_1)+u_2\sv(x_2)\qquad\text{for }x\in Q_{0,(z_1,\ell_2)}.
\end{align*}
In the intersection $Q_{0,z}$, there holds
\[
u_1\sh(x_1)+u_2\sh(x_2)=u(x)=u_1\sv(x_1)+u_2\sv(x_2),
\]
and substituting $(u_1\sv-c,u_2\sv+c)$ for $(u_1\sv,u_2\sv)$ for some $c\in\R$ we may assume that $u_l\sv=u_l\sh$ in $Q_{0,z}$ for $l=1,2$. Then, setting $u_1=u\sh_1$ and  $u_2=u\sv_2$ we have 
\[
u(x)=u_1(x_1)+u_2(x_2) \quad\text{for }x\in\mathcal{L}\qquad\text{and}\qquad u\in \U(Q_{0,(\ell_1,z_2)})\cap \U(Q_{0,(z_1,\ell_2)}).
\]
If $u_1$ is not constant in $(0,\ell_1)$, then $u_2$ is constant in $(0,z_2)$. Up to the addition of a constant, we may assume without loss of generality that $u_2=0$ in $(0,z_2)$. This implies that $u=u_1$ in $Q_{0,(\ell_1,z_2)}$ and since $u$ is a characteristic function we see that in any case $u_1=\un_{A_1}$ for some measurable set $A_1\sub (0,\ell_1)$.   The exact same considerations in $ Q_{0,(z_1,\ell_2)}$ show that $u_2=\pm \un_{A_2}$ for some measurable set $A_2\sub (0,\ell_2)$. Eventually, since $\pt_1\pt_2(u-w)=0=\pt_1\pt_2 (u_1+u_2)$ in the rectangle $\Om$, the identity $u-w=u_1+u_2$, valid in $\mathcal{L}$, propagates to $\Om$. This concludes the proof.
\end{proof}


We now establish an elementary measure theoretical lemma used in the proofs of Theorem~\ref{theo:dirac2} and Theorem~\ref{theo:lines}.

\begin{lemma}\label{lemma:tech}
Let $\mc X$ be a locally compact, complete and separable metric space, let $\mu$ be a finite Radon measure on $\mc X$. For $k\ge 1$, let $\QQ^k$ be a collection of pairwise disjoint Borel subsets of $\mc X$ such that $d_k:=\sup\{ \diam Q \,:\, Q\in \QQ^k\}$ goes to $0$ as $k\up \oo$. Then, 
\[
|\mu|\lt(\cap_{k}\cup \QQ^k\rt)\le  \liminf_{k\up\oo} \sum_{Q\in\QQ^k} |\mu(Q)|.
\]
\end{lemma} 
\begin{remark}
The lemma is false if we do not assume that the elements of $\QQ^k$ are disjoint.
In the sequel we apply the result with $\mathcal{X}=\Om$, a bounded open set of $\R^n$ and with $\QQ^k$ a finite set of disjoint boxes.
\end{remark}
\begin{proof}[Proof of Lemma~\ref{lemma:tech}]~In the proof we denote ${\mc Y}^k:=\cup\QQ^k$ and $\mc Y:=\cap_k{\mc Y}^k$.\\
Writing $\mu=\zeta |\mu|$ the polar decomposition of $\mu$, the function $\zeta$ is Borel measurable and takes values in $\{-1,1\}$. Let $\eps>0$, there exists $\zeta_\eps\in C_c(\mc X,[-1,1])$  such that
\be
\label{lemma:tech_prf1}
\|\zeta_\eps -\mathbf1_{\mc Y}\zeta \|_{L^1(\mc X,|\mu|)} <\eps.
\ee
For  $k\ge 1$ and  $Q\in\QQ^k$, we set
\[
\zeta_\eps(Q):=\dfrac1{|\mu|(Q)} \int_Q \zeta_\eps\, d|\mu|\quad \text{ with the convention }\zeta_\eps(Q)=0\text{ if } |\mu|(Q)=0.
\] 
By construction, $-1\le \zeta_\eps(Q)\le 1$ and we deduce  
\be\label{lemma:tech_prf2}
\sum_{Q\in \QQ^k} |\mu (Q)|  \ge  \sum_{Q\in \QQ^k}  \zeta_\eps(Q) \mu(Q) = \int_{{\mc Y}^k}\zeta_\eps \,d\mu + \sum_{Q\in \QQ^k} \int_Q\lt(\zeta_\eps(Q) -\zeta_\eps\rt)\,d\mu. 
\ee
The last term is bounded from below by 
\[
-|\mu|(\mathcal{X}) \max\{|\zeta_\eps(x)-\zeta_\eps(y)|:x,y\in \mathcal{X},\,|x-y|\le d_k\}.
\]
By uniform continuity of $\zeta_\eps$ and the assumption $d_k\dw0$, this goes to $0$ as $k\up \oo$. \\
Taking the infimum limit of~\eqref{lemma:tech_prf2} as $k\up\oo$, we obtain,
\[
\liminf_{k\up \oo} \sum_{Q\in\QQ^k} |\mu(Q)| \,\ge\,  \liminf_{k\up \oo} \int_{{\mc Y}^k}\zeta_\eps \,d\mu  \st{\eqref{lemma:tech_prf1}}\ge \int_{\mc Y} \zeta\,d\mu-\eps = |\mu|(\mc Y) -\eps.
\]
Recalling that $\eps>0$ is arbitrary, the lemma is proved.
\end{proof}

In the proof of Theorem~\ref{theo:dirac2} we use the following characterization of discrete measures. Let us first recall the notation $Q_{x,r}=x+ r[0,1)^n$ for $x\in \Om$ and $r>0$ and let us notice that for $x\in \Om^\eps$ and $0<\sqrt{n}\,r\le\eps$, we have $Q_{x,r}\sub\Om$.
\begin{lemma}\label{lem:Diracmeasure}
 Let $n\ge 1$,  $\Om\sub \R^n$ be a bounded open set and $\mu$ be a finite Radon measure on $\Om$ such that there exist $0<\te<1$  and  sequences $r_k$, $\eps_k$ with $0<\sqrt{n}\,r_k\le\eps_k\dw0$ for which
 \be\label{hyp:Diracmeasure}
  \liminf_{k\up \oo} \int_{\Om^{\eps_k}} \dfrac{|\mu(Q_{x,r_k})|^\te}{r_k^n} \,dx<\oo. 
 \ee
 Then there exist sequences $x^j\in\Om$, $m_j\in\R\bks\{0\}$ such that $\mu=\sum m_j\delta_{x^j}$ with the estimate
  \be\label{conclusion:Dirac}
  \sum  |m_j|^\te\le \liminf_{k\up \oo} \int_{\Om^{\eps_k}} \dfrac{|\mu(Q_{x,r_k})|^\te}{r_k^n} \,dx.
 \ee
\end{lemma}
\begin{remark}
\label{remark_not_eps_kr_k_of_prop}
In the statement of the lemma, the sequences $\eps_k$ and $r_k$ are not \textit{a priori} those of Proposition~\ref{prop:goodqk}.
\end{remark}
\begin{proof}[Proof of Lemma~\ref{lem:Diracmeasure}]~
 For $y\in\Om$, we define 
 \[
M_k(y):= \sup\lt\{ |\mu|(Q_{x,r_k}): x\in \Om\text{ such that } y\in Q_{x,r_k}\sub\Om\rt\}.
 \]
and then for $\eta>0$,
\[
\Sigma_\eta:=\lt\{y\in \Om\, :\, \limsup_{k\up \oo}M_k(y)\ge\eta\rt\}.
\]
For every $y\in \Sigma_\eta $ we have $|\mu(\{y\})|\ge \eta$ and thus $\Sigma_\eta$ is a finite set. Taking the union of the nested family $\{\Sigma_\eta\}_{\eta>0}$, we  see that the set 
\begin{equation*}
\Sigma:=\,\!\!\bigcup_{\eta\dw 0} \Sigma_{\eta} = \lt\{y\in \Om\, :\, |\mu|(\{y\})>0\rt\}
\end{equation*}
is at most countable.

We claim  that $|\mu|( \Om \sm  \Sigma)=0$.  By definition, for every $\eps>0$, 
\[
\Om^\eps \sm\Sigma = \lt\{y\in \Om^\eps\, :\, \lim_{k\up\oo} M_k(y)=0 \rt\}.
\]
By Egoroff theorem, there exists a Borel set  $ \mc Y \sub\Om^\eps \sm \Sigma$ with $|\mu|(\mc Y)\ge|\mu|( \Om^\eps \sm\Sigma)/2$ and
\be\label{OmepsmN}
 \lim_{k\up \oo} t_k=0\qquad \text{ where }\ t_k:=\sup_{y\in{\mc Y}}M_k(y).
\ee
For $k\ge 1$ and $x\in Q_{0,r_k}$ we consider the set 
\[
 \QQ^k_x:=\{ Q\ : \ Q=Q_{x+r_k z, r_k} \ \text{ for some  }z\in \Z^n,\ Q\sub \Om \text{ and } Q\cap \mc Y \neq \void\}.
\]
Using the definition of $\mc Q^k_x$, we have,
\[
\dfrac1{r_k^n} \int_{Q_{0,r_k}} \sum_{Q\in\QQ^k_x} |\mu(Q)|\,dx\,  \le \, \dfrac{t_k^{1-\te}}{r_k^n}  \int_{Q_{0,r_k}}\sum_{Q\in\mc Q^k_x} |\mu(Q)|^{\te}\, dx\, \le\,  t_k^{1-\te} \int_{\Om^{\eps_k}} \dfrac{|\mu(Q_{y,r_k})|^\te}{r_k^n} \,dy,
\]
where we used Fubini for the last identity. Taking the infimum limit as $k\up\oo$, we see that by assumption~\eqref{hyp:Diracmeasure}, the right-hand side tends to zero. In particular, there exists a sequence  $(x^k)_{k\ge 1}$ with $x^k\in Q_{0,r_k}$ such that 
\be
\nonumber
 \liminf_{k\up\oo} \sum_{Q\in\QQ^k_{x^k}} |\mu(Q)|\, =\, 0.
\ee
By~\eqref{OmepsmN}, $\mc Q^k_{x^k}$ covers $\mc Y$ for $k$ large enough (as soon as $r_k<\eps/\sqrt{n}$). Applying Lemma~\ref{lemma:tech} to the measure $\mu$ and to the sequence $\{\mc Q^k_{x^k}\}$, we obtain 
\[
|\mu|(\Om^\eps\sm\Sigma)/2\le |\mu|(\mc Y)  \le \liminf_{k\up\oo}  \sum_{Q\in\QQ^k_{x^k}} |\mu(Q)|\, =\, 0.
\]
Since $\eps>0$ is arbitrary we conclude that $\mu=\mu \restr \Sigma$.\smallskip

Recalling that $\Sigma$ is at most countable, we write $\Sigma=(x^j)_{j\ge1}$ with $x^j$ pairwise distinct. We then have $\mu=\sum   m_j\delta_{x^j}$ for some summable sequence of real numbers $(m_j)$. For every $N\in \N$ and every $t>0$, there exists $r(N,t)>0$ such that for $0<r\le r(N,t)$ and $x\in \Om$, 
\begin{align*}
&Q_{x,r}\cap \{x^1,\cdots,x^N\} \text{ has at most one element}\\ 
&\text{and if }x^j\in Q_{x,r} \text{ for some }1\le j\le N \text{ then } |\mu(Q_{x,r})|\ge (1-t)|m_j|.
\end{align*}
We thus have 
\[
 \liminf_{k \up \oo}  \int_{\Om^{\eps_k}}  \dfrac{|\mu(Q_{x,r_k})|^\te}{r_k^n} \,dx
 \ge  \liminf_{k\up \oo} \sum_{j=1}^N\int_{\Om}\un_{Q_{x,r_k}}(x^j)\dfrac{|\mu(Q_{x,r_k})|^\te}{r_k^n} \,dx\ge (1-t)^{\te}\sum_{j=1}^N |m_j|^\te.
\]
Sending $t\dw 0$ and $N\up \oo$, we obtain~\eqref{conclusion:Dirac}. 
\end{proof}
\begin{proof}[Proof of Theorem~\ref{theo:dirac2}]~
The proof  goes as explained in the introduction. Using Lemma~\ref{lemma:4points}, the theorem follows from Proposition~\ref{prop:goodqk} and Lemma~\ref{lem:Diracmeasure} together with the triangle inequality in the form of 
\[
 q(x,z)\ges |\mu[u](Q_{x,z})|^\te.
\]
 Finally, if $u$ is integer-valued, we can argue as in the proof of Theorem~\ref{theo:dirac1} that for every $x\in \supp \mu[u]$ there exists a sequence $x^k\to x$ and $z^k\to 0$ such that~\eqref{id4pts} holds. Therefore 
$\mu[u](\{x\})=\lim_{k\to \oo} \mu[u](Q_{x^k,z^k})=\lim_{k\to \oo}  D[Du(\cdot,z_2^k)](x^k,z^k_1)\in \Z.
$
\end{proof}

\begin{proof}[Proof of Theorem~\ref{thmBVtheta}]~
The proof is a refinement of the proof of~\cite[Theorem~4.1]{GM1} using the additional information given by Theorem~\ref{theo:dirac2}.    
Without loss of generality, we assume that $I_l=(0,\ell_l)$ for $l=1,2$.    Let $u\in L^\oo(\Om)$ with $\|u\|_\oo\le 1$ and $\E( u)<\oo$. \medskip

\noindent 
\textit{Proof of (i), Step 1. Decomposition of $u$ and control of $\|w\|_\oo+|\nb w|_\te(\Om)$.}\\
Let us first recall that by~\eqref{decompu}, there exist $u_1\in L^\oo(I_1)$, $u_2\in L^\oo(I_2)$ such that 
\be\label{decompurap}
u(x)=u_1(x_1)+u_2(x_2)+w(x) \qquad\text{\ae in } \Om.\smallskip
\ee
By Theorem~\ref{theo:dirac2}, $\mu:=\mu[u]=\pt_1\pt_2 u$ is a finite Radon measure which writes as
\be\label{encoresommempowteta}
\mu=\sum m_j \delta_{x^j}\qquad\text{ with }\qquad\sum |m_j|^\te\les\E(u).
\ee 
By direct computation, the definition~\eqref{defw}  of $w$ yields  the identities
\[
\pt_1 w=\sum m_j \HH^1\restr I_1^j,\qquad\qquad\pt_2w =\sum m_j \HH^1 \restr I_2^j,
\]
where, for $j\ge 1$, we denote,  $I_1^j:=\{x_1^j\}\times (0,x_2^j)$ and $I_2^j:= (0,x_1^j)\times \{x_2^j\}$. We deduce that $w\in SBV_\te(\Om)$ with the estimate 
\be\label{SBVtew}
|\nb w|_\te(\Om)
=\sum  (|x_1^j|+|x_2^j|)|m_j|^\te 
\le( \ell_1+\ell_2)\sum  |m_j|^\te \stackrel{\eqref{encoresommempowteta}}\les ( \ell_1+\ell_2)\E(u).
\ee
Besides, by definition of $w$ we have $w\in L^\oo(\Om)$ with
\be\label{SBVtew0}
\|w\|_\oo\le |\mu|(\Om)\les\E(u). 
\ee
Eventually, let us notice for later use that by Lemma~\ref{lemma:4points} there holds, for every $j\ge1$, 
\be\label{supml}
|m_j|\le \sup_{Q\sub\Om} |\mu(Q)| \le 4\|u\|_\oo\le4.
\ee


\noindent
\textit{Proof of (i), Step 2. Control of the oscillations of $u_1$ or of $u_2$.}\\
We proceed as in the proof of~\cite[Theorem~4.1]{GM1}. For every fixed  $z\in \R^2$,  almost every $x\in\Om$  with $Q_{x,z}\sub \Om$ satisfies $Du(x,z_l)=Du_l(x_l,z_l)+Dw(x,z_l)$ for $l=1,2$. Since, $Dw(x,z_1)=\mu((x_1,x_1+z_1]\times (0,x_2])$ and $Dw(x,z_2)=\mu((0,x_1]\times (0,x_2+z_2])$, we deduce from the triangle inequality, 
\begin{align}\label{g1}
\psi_1(x_1,z_1)&:= \Big[|D u_1(x_1,z_1)|-|\mu|((x_1,x_1+z_1]\times I_2)\Big]_+\le  |Du(x,z_1)|,\\
\label{g2}
 \psi_2(x_2,z_2)&:=\Big[|D u_2(x_2,z_2)|-|\mu|(I_1\times(x_2,x_2+z_2])\Big]_+\le |Du(x,z_2)|, 
\end{align}
where, as usual, $a_+:=\max(a,0)$ denotes the positive part of  $a$.   
Raising inequality~\eqref{g1} to the power $\te_1$ and~\eqref{g2} to the power $\te_2$, taking the product and integrating, we have by Proposition~\ref{prop:goodqk}, 
\[
\limsup_{k\up \oo} \lt(\int_{I_1^{\eps_k}} \dfrac{[\psi_1(x_1,z_1^k)]^{\te_1}}{r_k} \,dx_1\rt) \lt(\int_{I_2^{\eps_k}} \dfrac{[\psi_2(x_2,z_2^k)]^{\te_2}}{r_k} \,dx_2\rt) \, \les \, \E( u).
\]
We obtain that either for $l=1$ or for $l=2$ and  up to extraction, 
\be\label{borneintgi}
\limsup_{k\up \oo} \int_{I_l^{\eps_k}} \dfrac{\lt[\psi_l(x_l,z_l^k)\rt]^{\te_l}}{r_k} \,dx_l\, \les \, \E( u)^{1/2}.
\ee
Let us assume for instance that~\eqref{borneintgi} holds with $l=1$ and let $\te_* \in [\te,1]$.  Since $\te_*\le1$, the function $s\in\R_+\mapsto s^{\te_*}$ is subaddititive and we have for every $k\ge 1$,
\begin{multline}\label{borne_int_Du1}
 \int_{I_1^{\eps_k}} |D u_1(x_1,r_k)|^{\te_*} \,dx_1\\
 \le  \int_{I_1^{\eps_k}} \lt[\psi_1(x_1, r_k e_1)\rt]^{\te_*} \,dx_1 + \int_{I_1^{\eps_k}} \lt[|\mu|((x_1,x_1+r_k]\times I_2)\rt]^{{\te_*}} \,dx_1.
\end{multline}
To estimate the second term in the right-hand side, we use the atomic decomposition of $\mu$ and the subadditivity of $s\in\R_+\mapsto s^{\te_*}$. We have,
\begin{multline}
\int_{I_1^{\eps_k}} \lt[|\mu|((x_1,x_1+r_k]\times(0,\ell_2))\rt]^{{\te_*}} \,dx_1
= \int_{I_1^{\eps_k}} \lt(\sum_{\{ 0<x_1^j-x_1\le r_k\}} |m_j|\rt)^{\te_*}\,dx_1 \\
\le \int_{I_1^{\eps_k}} \sum_{\{ 0<x_1^j-x_1\le r_k\}} |m_j|^{\te_*}\,dx_1
\le (\sup |m_j|)^{\te_*-\te}  \int_{I_1^{\eps_k}} \sum_{\{j\,:\,0<x_1^j-x_1\le r_k\}} |m_j|^\te\,dx_1\\
\label{borne_mu_x1,x_1+rk}
\le (\sup |m_j|)^{\te_*-\te} \sum\lt[|m_j|^\te \HH^1((x_1^j-r_k,x_1^j])\rt]
\stackrel{\eqref{encoresommempowteta}\eqref{supml}}\les r_k\E(u).
\end{multline}
To recover the conclusion of~\cite[Theorem~4.1]{GM1}, we first pick $\te_*=1$. Dividing~\eqref{borne_int_Du1} by $r_k$, taking the superior limit as $k\up \oo$ and using~\eqref{borneintgi} and~\eqref{borne_mu_x1,x_1+rk}, we obtain
\[
\limsup_{k\up\oo}  \int_{I_1^{\eps_k}} \dfrac{|D_1 u_1(x_1,r_k)|}{r_k} \,dx_1\les \E(u)^{1/2}+\E(u).
\]
This readily yields $u_1\in BV(I_1)$ and up to the subtraction of a constant, we have
\be\label{estimBVu1}
\|u_1\|_\oo+|\nb_1u_1|(I_1)\les \E(u)^{1/2}+\E(u).
\ee
Next, we use $\te_*=\te$ to get 
\[
\limsup_{k\up\oo}  \int_{I_1^{\eps_k}} \dfrac{|D_1 u_1(x_1,r_k)|^\te}{r_k} \,dx_1\les \E(u)^{1/2}+\E(u).
\]
We apply Lemma~\ref{lem:Diracmeasure} with $n=1$ to the measure $\mu_1:= \pt_{e_1}u_1$, we can write
\[
\mu_1=\sum \widetilde m_j \delta_{\widetilde x_1^j}
\]
for some sequences $(\widetilde x_1^j)\sub I_1$ and $\widetilde m_j\in\R$ with $\sum  |\widetilde m_j|^\te<\oo$. We deduce that $u_1\in SBV_\te(I_1)$, with  
\be\label{estimBVteu1}
|\nb_1u_1|_\te(I_1)\les \E(u)^{1/2}+\E(u).
\ee
Eventually, we set $\bar u(x)=u_2(x_2)$, we have $\bar u\in \U(\Om)$ and $u(x)-\bar u(x)=u_1(x_1)+w(x)$.  Collecting~\eqref{SBVtew},~\eqref{SBVtew0},~\eqref{estimBVu1} \&~\eqref{estimBVteu1}, we conclude that $u-\bar u\in SBV_\te(\Om)$ with the estimate~\eqref{quantrig2dtheta}.\medskip

\noindent
\textit{Proof of (ii)}.\\
Let us now turn to the proof of the second part of the theorem and  assume that  $\mu[u]=0$ as well as $u\notin \U(\Om)$. This implies that $w=0$ in~\eqref{decompurap}.  By Proposition~\ref{prop:goodqk}, we have 
\[
\limsup_{k\up \oo} \lt(\int_{I_1^{\eps_k}} \dfrac{|D u_1(x_1,z_1)|^{\te_1}}{r_k} \,dx_1\rt) \lt(\int_{I_2^{\eps_k}} \dfrac{|D u_2(x_2,z_2)|^{\te_2}}{r_k} \,dx_2\rt) \, \les \, \E( u).
\]
Therefore, there exist constants $\lambda_1,\lambda_2\in[0,\oo]$ such that $\lambda_1\lambda_2\les \E(u)$ and for $l=1,2$,
\[
\liminf_{k\up \oo} \int_{I_l^{\eps_k}} \dfrac{|D u_l(x_l,z_l)|^{\te_l}}{r_k} \,dx_l \,\le \lambda_l.
\]
Since both $u_1$ and $u_2$ are not constant, this implies that $\lambda_l>0$ for $l=1, 2$ and thus also $\lambda_1,\lambda_2<\oo$. Arguing exactly as in 
(i)  we get that  $u_l\in SBV_{\te_l}(I_l)$ for $l=1,2$ with $|\nb u_l|_{\te_l}(I_l)\le \lambda_l$ with the estimate~\eqref{quantrig2dtheta2}.\medskip

\noindent
\textit{Proof of (iii)}.\\
We finally assume that $u\in L^\oo(\Om)\cap BV(\Om)$ has finite energy and that moreover\footnote{we note $\nb^{\mathrm j}v$ the jump part of the distributional gradient of a $BV$ function $v$.} $\nb^{\mathrm j}u=0$. 
We have $\nb^{\mathrm j}w =\nb w$,
so that  $0=\nb^{\mathrm j}_1u=\pt_1 w +\nb^{\mathrm j}_1 u_1$. Taking the derivative with respect to $x_2$, we obtain $0=\pt_2\pt_1 w=\mu[u]$. 
By (ii), $u\in \U(\Om)$ since otherwise we could write it as $u=u_1+u_2$ where $u_l\in SBV_{\te_l}(I_l)$ are non constant functions for $l=1,2$ and  would contradict the hypothesis $\nb^j u=0$.  
\end{proof}

%
%

\subsection{The higher dimensional case}\label{S3}
We now turn to the case $n=n_1+n_2>2$. We recall that we have fixed orthonormal bases $(e_1,\cdots,e_{n_1})$ and $(f_1,\cdots,f_{n_2})$ of $\X_1$ and $\X_2$ and have set
\be\label{expressmuu}
\mu[u]=\nb_1\nb_2 u=\sum_{\substack{1\le i\le n_1\\1\le j \le n_2}} \dfrac{\pt^2 u}{\pt x^1_{i}\pt x^2_{j}} \, e_i\otimes f_j=\sum_{\substack{1\le i\le n_1\\1\le j \le n_2}} \mu_{i,j} \, e_i\otimes f_j.
\ee
As outlined in the introduction we will identify $\mu[u]$ with a $(n-2)\,$-dimensional current. To this aim we recall some notation from Geometric Measure Theory, see~\cite{Federer,KrantzPark}. We note $\DD^k(\Om)$ the space of smooth and compactly supported $k\,$-differential forms on $\Om$ on which acts the differential operator $d:\DD^k(\Om)\to\DD^{k+1}(\Om)$. Its dual space is the space of $k\,$-currents $\DD_k(\Om)$ on which acts the dual operator  $\pt :\DD_k(\Om)\to\DD_{k-1}(\Om)$. We use the standard notation $\we$ for the exterior product. We recall that a $k\,$-current $T$ is rectifiable if there exists a $k\,$-rectifiable set  $\Sigma\sub \Om$ oriented by a unitary simple $k\,$-vector field $\xi$ and a Borel measurable multiplicity function $m:\Sigma\to \R$ such that for $\om\in \DD^k(\Om)$,
\be\label{def:rectifiableT}
 \lt<T,\om\rt>=\int_\Sigma m\,\lt<\om,\xi\rt> d\HH^{k}.
\ee
We  introduce the ``partial'' differentials $d_1$, $d_2$ as follows:
\[
\begin{array}{rclcrcl}
d_1:\DD^{k_1}(\Om_1) \we\DD^{k_2}(\Om_2)  &\to& \DD^{k_1+k_2+1}(\Om)&&d_2:\DD^{k_1}(\Om_1) \we\DD^{k_2}(\Om_2) &\to& \DD^{k_1+k_2+1}(\Om) \\
\om_1\we \om_2&\mapsto & d\om_1\we\om_2,&&\om_1\we\om_2&\mapsto & \om_1\we d\om_2.
\end{array}
\]
For $0\le k\le n$, we then extend by linearity and density the operators $d_1$, $d_2$ on
\[
\DD^k(\Om)=\ov{\Span \lt(\oplus_{l=0}^k \DD^l(\Om_1)\we\DD^{k-l}(\Om_2)\rt)}.
\]
  By duality, this defines continuous partial boundary operators on the space of currents:
\[
\lt<\pt_1 T,\om\rt>:=\lt<T,d_1\om\rt>,\quad \lt<\pt_2 T,\om\rt>:=\lt<T,d_2\om\rt>,\quad\text{for }T\in \mc{D}_k(\Om),\ \om\in \mc{D}^{k-1}(\Om).
\]
 For $u\in L^1_{loc}(\Om)$, we define the current
\[
\lb u\rb:=u\,e_1\we\cdots\we e_{n_1}\we f_1\we \cdots \we f_{n_2}\ \in\DD_n(\Om),
\]
from which, we derive the $(n-2)\,$-dimensional current
\[
T[u]:=\pt_1\pt_2 \lb u\rb.
\]
For every $(n-2)\,$-current $T$ we define the $\theta\,$-mass of $T$ by
\[
 \M_\theta(T):=\begin{cases}
\displaystyle  \int_{\Sigma} |m|^\theta\,d\HH^{n-2} & \text{if } T \text{ is rectifiable,}\\[8pt]
\qquad +\oo &\text{otherwise}.
\end{cases}
\]
We say that a rectifiable $(n-2)\,$-current is tensor-rectifiable if we can choose the set $\Sigma$ from~\eqref{def:rectifiableT} such that $\Sigma\sub \Sigma^1+\Sigma^2$ where $\Sigma^l\sub \Om_l$ is $(n_l-1)\,$-rectifiable. Notice that in this case, the $k\,$-vector field $\xi$ tangent to $\Sigma$ must be of the form $\xi=\xi^1\we\xi^2$ with $\xi^l$ tangent to $\Sigma^l$.
For $l=1,2$ we set
\[
 dx^l=dx_1^l\wedge \cdots \wedge dx_{n_l}^l
\]
and introduce the following simple multi-covectors for $1\le i\le n_l$,
\[
d x^l_{\bar\imath}:=dx^l_1\we\cdots\we dx^l_{i-1}\we dx^l_{i+1}\we\cdots\we dx^l_{n_l}.
\]
Similarly, we define the simple multi-vectors
\[
e_{\bar\imath}:=e_1\we\cdots\we e_{i-1}\we e_{i+1}\we\cdots\we e_{n_1},\qquad f_{\bar\jmath}:=f_1\we\cdots\we f_{j-1}\we f_{j+1}\we\cdots\we f_{n_2}.
\]
We define the Hodge star operator on $\X_1$, first on the simple $(n_1-1)\,$-vectors $e_{\bar\imath}$ by $\star e_{\bar\imath}=(-1)^i e_i$, and then extend it by linearity. Similarly we set $\star f_{\bar \jmath}=(-1)^j f_j$.\\
Every $\om\in \DD^{n-2}(\Om)$ decomposes as
\be\label{decompoomega}
\om =\lt[\sum_{\substack{1\le i\le n_1\\1\le j \le n_2}} \om_{i,j}\,d x^1_{\bar\imath}\we dx^2_{\bar\jmath}\rt] +\om_1\, \we dx^2 \,+\,dx^1\we\om_2.
\ee

\begin{proposition}\label{equivTmu}
 For every $u\in L^1_{loc}(\Om)$  we have with  the notation~\eqref{expressmuu},
 \be\label{Tuomegaij}
  T[u]=\sum_{\substack{1\le i\le n_1\\1\le j \le n_2}} (-1)^{i+j} \mu_{i,j}\,  e_{\bar\imath} \wedge f_{\bar\jmath}.
 \ee
As a consequence, $\mu[u]$ is rectifiable (respectively tensor-rectifiable) if and only if $T[u]$ is rectifiable (respectively tensor-rectifiable). Moreover $\pt T[u]=0$ (we say that $T[u]$ is a cycle),
\be\label{eqmass}
 \M(T[u])\les |\mu[u]|(\Om)\qquad\text{and}\qquad \M_\theta(T[u])=\M_\theta(\mu[u]).
\ee

\end{proposition}
\begin{proof}
 We first establish~\eqref{Tuomegaij}. Fix $\om\in \DD^{n-2}(\Om)$ and decompose it as in~\eqref{decompoomega}.  Since
 \be
 \label{d2dx2=0}
 d_2(\om_1\we dx^2)=\om_1\we d(dx^2)=0= d_1(dx^1\we \om_2),
 \ee
 we have 
 \begin{multline*}
  \lt< T[u],\om\rt>=\lt< \lb u\rb, d_2 d_1 \om\rt>=\sum_{i,j} \lt< \lb u\rb, d_2 d_1 \lt( \om_{i,j} \,d x^1_{\bar\imath}\we dx^2_{\bar\jmath}\rt)\rt>\\
  =\sum_{i,j} (-1)^{i+j} \int_\Om \dfrac{\pt^2 \om_{i,j}}{\pt x^1_i \pt x^2_j} u\,dx =\sum_{i,j} (-1)^{i+j}\int_\Om \om_{i,j}\,d\mu_{i,j}.
 \end{multline*}
This proves~\eqref{Tuomegaij}.\\
Assume now that $T[u]$ is rectifiable. Let $(\Sigma, \xi, m)$ be as in~\eqref{def:rectifiableT}. First, by~\eqref{Tuomegaij} we can write $\xi=\xi^1\wedge \xi^2$ for some simple $(n_l-1)\,$-vectors $\xi^l$ of $X_l$. We let $\nu^l= \star \xi^l\in X_l$ be the normals to $\Sigma$. If we decompose $\xi^1$ as
\[
 \xi^1=\sum_i \xi^1_{\bar \imath} e_{\bar \imath}
\]
we have by definition of the Hodge star operator $\xi^1_{\bar \imath}=(-1)^i \nu^1_i$. Similarly $\xi^2_{\bar \jmath}=(-1)^j \nu^2_j$. Using
\eqref{Tuomegaij} we find
\[
 \mu_{i,j}=m (-1)^{i+j} \xi^1_{\bar\imath} \xi^2_{\bar\jmath}\, \HH^{n-2}\restr \Sigma=  m\nu^1_i \nu^2_j\, \HH^{n-2}\restr \Sigma.
\]
This proves that $\mu[u]= m \nu^1\otimes \nu^2 \, \HH^{n-2}\restr \Sigma$ and thus that $\mu[u]$ is rectifiable. If moreover $T[u]$ is tensor rectifiable we see that also $\mu[u]$ is tensor rectifiable. Since we can revert the argument, this also proves that if $\mu[u]$ is rectifiable (respectively tensor rectifiable), then $T[u]$ is rectifiable (respectively tensor rectifiable).\\
Since on $\DD^{k_1}(\Om_1)\we \DD^{k_2}(\Om_2)$, $d =d_1+ (-1)^{k_1} d_2$ we have $d_2d_1 d =0$ and thus $\pt T[u]=0$. We finally prove~\eqref{eqmass}. For the first identify we write using the decomposition~\eqref{decompoomega} and~\eqref{Tuomegaij},
\[
 \M(T[u])=\sup_{|\om|\le 1} \lt<T[u],\om\rt>=\sup_{|\om|\le 1} \sum_{i,j} (-1)^{i+j}\int_{\Om}  \om_{i,j}\,d\mu_{i,j}\les |\mu|(\Om).
\]
By the previous discussion, the equality $\M_\theta(T[u])=\M_\theta(\mu[u])$ is immediate (notice that  both terms are finite only if $T[u]$ is rectifiable). 
\end{proof}

Thanks to Proposition~\ref{equivTmu}, we can reduce the proof of Theorem~\ref{coromurect} to the analog result for $T[u]$.
\begin{theorem}\label{thm:thetamass} 
Let $u\in L^\oo(\Om)$ be such that 
\[\E''(u):= \liminf_{k\up \oo}\sum_{1\le i \le n_1,1\le j\le n_2} \int_{\Om^{\eps_k}} \dfrac{|D[D u(\cdot, r_k f_j)](x,r_k e_i)|^\te}{ r_k^2} \,dx<\oo.\]
Then $T[u]$ is a tensor rectifiable cycle with 
\be\label{estimTheorem}
 \M(T[u])\les \|u\|_\oo^{1-\theta} \E''(u) \qquad \text{and } \qquad \M_\theta(T[u])\les \E''(u).
\ee
Moreover, if  $u$ is integer-valued then $T[u]$ is integer rectifiable.
\end{theorem}
\begin{proof}
 \setcounter{proof-step}{0}
\noindent{\textit{Step \stepcounter{proof-step}\arabic{proof-step}. Preliminary observations.}}
We first notice that by Remark~\ref{rem:controlmu}, we have $|\mu[u]|(\Om)\les \|u\|_\oo^{1-\theta} \E''(u)$ and thus by Proposition~\ref{equivTmu}, $T[u]$ is a normal current and the first inequality in~\eqref{estimTheorem} holds. We thus need to show that $T[u]$ is tensor rectifiable and the second inequality in~\eqref{estimTheorem}.\medskip

\noindent{\textit{Step \stepcounter{proof-step}\arabic{proof-step}. The case where $\Om$ is a cube.}}\\
In this step we prove the claim assuming that $\Om$ is a $n$-cube. By scaling we assume without loss of generality that $\Om=Q^n:=(0,1)^n$.\\
We start by recalling the definition of slicing of currents by coordinate $(n-2)\,$-planes (see~\cite{GMS1,Federer,White1999-2,Jerrard} for more details). For this we introduce some notation. For $\alpha^1\sub \{1,\cdots,n_1\}$ and $\alpha^2\sub\{1,\cdots,n_2\}$ we set 
\[
\alpha:=\alpha_1\cup (n_1+\alpha_2)\qquad\text{and}\qquad\bar{\alpha}:=\{1,\cdots,n\}\sm\alpha.
\]
We then define 
$\X_{\alpha^1}=\Span\{ e_i\}_{i\in \alpha^1}$, $\X_{\alpha^2}=\Span\{ f_j\}_{j\in \alpha^2}$ and set $\X_\alpha=\X_{\alpha^1}\oplus \X_{\alpha^2}$ so that $\R^n=\X_\alpha\oplus \X_{\bar \alpha}$. We decompose correspondingly every $x\in \R^n$ as $x=x^{\alpha}+x^{\bar \alpha}$. With a slight abuse of notation we set $Q^\alpha=Q^n\cap \X_\alpha$. For a function $u\in L^1(Q^n)$ and $x\in Q^n$ we set $u_{x^{\bar \alpha}}(x^\alpha):=u(x)$ so that by Fubini $u_{x^{\bar \alpha}}\in L^1(Q^{\alpha})$ for almost every $x^{\bar \alpha}$. For every $\alpha$ with $|\alpha|=2$ and every $T\in \DD_{n-2}(Q^n)$, we define the slices $\Sl_{\bar \alpha}^{x^{\bar \alpha}} T\in \DD^0(Q^{\alpha})$ by the requirement (see~\cite[Theorem~4.3.2]{Federer} or~\cite[Section 2.5]{GMS1}) that for every $\om\in \DD^0(Q^n)$,
\be\label{defslice}
 \lt<T,\om\,dx^{\bar\alpha}\rt>=\int_{Q^{\bar \alpha}} \lt<\Sl_{\bar \alpha}^{x^{\bar \alpha}} T, \om_{x^{\bar \alpha}}\rt>\, dx^{\bar \alpha}.
\ee
Here $dx^{\bar\alpha}$ denotes the canonical $(n-2)\,$-form. \\

\noindent \textit{Step \arabic{proof-step}.1.} We claim that for every $\alpha$ with $|\alpha|=2$ and almost every $x^{\bar \alpha}$, the slices $\Sl_{\bar \alpha}^{x^{\bar \alpha}} T[u]$ are $0\,$-rectifiable. Using White's rectifiability criterion~\cite{White1999-2,Jerrard} this would prove that $T[u]$ is $(n-2)\,$-rectifiable.\\
By~\eqref{Tuomegaij}, if $\alpha_1=\emptyset$ or $\alpha_2=\emptyset$, then $\Sl_{\bar \alpha}^{x^{\bar \alpha}} T[u]=0$ and there is nothing to prove. Indeed, if for instance $\alpha_2=\emptyset$ then $\{n_1+1,\dots,n\}\sub\bar\alpha$ so that  for every $\om\in \DD^0(Q^n)$ we have $\om \,dx^{\bar\alpha}=\om_1\wedge dx^{\bar \alpha}$ in the decomposition~\eqref{decompoomega}. Hence, using~\eqref{d2dx2=0} as in the proof of~\eqref{Tuomegaij}, we compute
\[
 \lt<T[u],\om\,dx^{\bar\alpha}\rt>=\lt<\lb u\rb,d_2d_1[\om dx^{\bar\alpha}]\rt>=\lt<\lb u\rb,d_20\rt>=0.
\]
 We conclude by the definition~\eqref{defslice} that $\Sl_{\bar \alpha}^{x^{\bar \alpha}} T[u]=0$. The case $\alpha_1=\emptyset$ is similar using $d_2d_1=-d_1d_2$.\\
We may thus assume that $\alpha_1=\{i\}$ and $\alpha_2=\{j\}$ so that $ dx^{\bar\alpha}=dx^1_{\bar\imath}\wedge dx^2_{\bar\jmath}$. Let us first prove that the operator $\pt_1\pt_2$ commutes with slicing in the sense that 
\be\label{commuteslice}
 \Sl_{\bar \alpha}^{x^{\bar \alpha}} T[u] = (-1)^{i+j} \pt_1\pt_2\lb  u_{x^{\bar \alpha}}\rb\qquad \text{ for \ae } \, x^{\bar \alpha}\in Q^{\bar \alpha}.
\ee
Indeed, for $\om\in \DD^0(Q^n)$, we compute using~\eqref{Tuomegaij} and~\eqref{expressmuu}
\[
 \lt<T[u],\om dx^1_{\bar\imath}\wedge dx^2_{\bar\jmath} \rt>=(-1)^{i+j}\int_{Q^n} \om\,d\mu_{i,j} =(-1)^{i+j}\int_{Q^n} \dfrac{\pt^2\om}{\pt x^1_i\pt x^2_j}\, u\,dx.
\]
 Then by Fubini,
\begin{align*}
  \lt<T[u],\om dx^1_{\bar\imath}\wedge dx^2_{\bar\jmath} \rt>&=(-1)^{i+j} \int_{Q^{\bar \alpha}}\lt[\int_{Q^\alpha} \dfrac{\pt^2\om_{x^{\bar \alpha}}}{\pt x^1_i\pt x^2_j}\,  u_{x^{\bar \alpha}}\,dx^{\alpha}\rt]\,dx^{\bar \alpha}\\
 &=(-1)^{i+j} \int_{Q^{\bar \alpha}}\lt<\pt_1\pt_2\lb  u_{x^{\bar \alpha}}\rb, \om_{x^{\bar \alpha}}\rt>\,dx^{\bar \alpha}.
\end{align*}
By definition~\eqref{defslice} of $\Sl_{\bar \alpha}^{x^{\bar \alpha}} T[u]$, this concludes the proof of~\eqref{commuteslice}.\\
We now prove that for almost every $x^{\bar \alpha}\in Q^{\bar \alpha}$, $T[u_{x^{\bar \alpha}}]=\pt_1\pt_2\lb  u_{x^{\bar \alpha}}\rb$ is $0\,$-rectifiable. To simplify notation we write $Q^n_{\eps}$ for $(Q^n)^\eps$. Let us prove that 
\be\label{fatouG}
 \E''(u)\ge \int_{Q^{\bar \alpha}} \E''(u_{x^{\bar \alpha}}) \,dx^{\bar \alpha}.
\ee
For this we use  Fubini and Fatou to obtain 
\begin{align*}
 \E''(u)&\ge \liminf_{k\up \oo} \int_{Q^n_{\eps_k}} \dfrac{|D[D u(\cdot, r_k f_j)](x,r_k e_i)|^\te}{ r_k^2} \,dx\\
  &=\liminf_{k\up \oo} \int_{Q^{\bar \alpha}} \un_{Q^{\bar \alpha}_{\eps_k}}(x^{\bar \alpha})\lt[\int_{Q^\alpha_{\eps_k}}\dfrac{|D[D u_{x^{\bar \alpha}}(\cdot, r_k f_j)](x^\alpha,r_k e_i)|^\te}{ r_k^2} \,dx^{\alpha}\rt] \,dx^{\bar \alpha}\\
 &\ge  \int_{Q^{\bar \alpha}} \liminf_{k\up \oo} \un_{Q^{\bar \alpha}_{\eps_k}}(x^{\bar \alpha})\lt[\int_{Q^\alpha_{\eps_k}}\dfrac{|D[D u_{x^{\bar \alpha}}(\cdot, r_k f_j)](x^\alpha,r_k e_i)|^\te}{ r_k^2} \,dx^{\alpha}\rt] \,dx^{\bar \alpha}\\
 &=\int_{Q^{\bar \alpha}} \liminf_{k\up \oo} \lt[\int_{Q^\alpha_{\eps_k}}\dfrac{|D[D u_{x^{\bar \alpha}}(\cdot, r_k f_j)](x^\alpha,r_k e_i)|^\te}{ r_k^2} \,dx^{\alpha}\rt] \,dx^{\bar \alpha}\\
 &=\int_{Q^{\bar \alpha}} \E''(u_{x^{\bar \alpha}}) \,dx^{\bar \alpha}.
\end{align*}
This proves~\eqref{fatouG}.
Therefore, for almost every $x^{\bar \alpha}\in Q^{\bar \alpha}$, $\E''(u_{x^{\bar \alpha}})<\oo$. By Theorem~\ref{theo:dirac2}, for every such $x^{\bar \alpha}$, there exist countable sequences $x^l\in Q^\alpha$ and $m_l\in \R\sm\{0\}$ with $T[u_{x^{\bar \alpha}}]=\sum_l m_l \delta_{x^l}$ and such that 
\be\label{estimsliceMtheta}
 \sum_{l} |m_l|\les \|u\|_{\oo}^{1-\theta} \E''(u_{x^{\bar \alpha}}) \qquad \text{ and } \qquad  \M_\theta(T[u_{x^{\bar \alpha}}])=\sum_{l} |m_l|^\theta\les \E''(u_{x^{\bar \alpha}}).
\ee
In particular, $T[u_{x^{\bar \alpha}}]$ is $0\,$-rectifiable. Moreover, if $u$ is integer-valued then by Theorem~\ref{theo:dirac2}, so is $T[u_{x^{\bar \alpha}}]$.
\medskip

\noindent \textit{Step \arabic{proof-step}.2.} We now prove the second inequality in~\eqref{estimTheorem}. For this we notice that if $(\Sigma,\xi,m)$ are as in~\eqref{def:rectifiableT} for $T=T[u]$, then as in the proof of Proposition~\ref{equivTmu}, we can write $\xi=\xi^1\wedge \xi^2$ for some simple (and unitary) $(n_l-1)\,$-vectors $\xi^l\in \X_l$. We write
\[
 \xi^1=\sum_i \xi^1_i e_{\bar \imath} \qquad \text{and } \qquad \xi^2=\sum   \xi^2_j f_{\bar\jmath}
\]
so that by triangle inequality
\begin{multline}\label{decompMthetaij}
 \M_\theta(T[u])=\int_{\Sigma} |m|^\theta\,d\HH^{n-2}=\int_{\Sigma} |m|^\theta |\xi^1\wedge\xi^2|\,d\HH^{n-2}
 \le \sum_{i,j} \int_{\Sigma} |m|^\theta |\xi^1_i \xi^2_j|\,d\HH^{n-2}.
\end{multline}
Fix $i\in \{1,\cdots, n_1\}$, $j\in \{1,\cdots, n_2\}$ and let $\alpha_1=\{i\}$, $\alpha_2=\{j\}$. By the co-area formula for rectifiable sets (see~\cite[Theorem~3.2.22]{Federer} or~\cite[Theorem~5.4.9]{KrantzPark}),
\[
 \int_{\Sigma} |m|^\theta |\xi^1_i \xi^2_j|\,d\HH^{n-2}=\int_{Q^{\bar \alpha}} \sum_{x^\alpha\in \Sigma\cap (Q^\alpha+x^{\bar \alpha})} |m(x)|^\theta dx^{\bar \alpha}= \int_{Q^{\bar \alpha}}\M_\theta (\Sl_{\bar \alpha}^{x^{\bar \alpha}} T[u])dx^{\bar \alpha}.
\]
Using~\eqref{commuteslice} and~\eqref{estimsliceMtheta}  we find
\[
 \int_{\Sigma} |m|^\theta |\xi^1_i \xi^2_j|\,d\HH^{n-2}\les \int_{Q^{\bar \alpha}}\E''(u_{x^{\bar \alpha}})dx^{\bar \alpha}\stackrel{\eqref{fatouG}}{\le} \E''(u).
\]
Plugging this in~\eqref{decompMthetaij} proves 
\[
 \M_\theta(T[u])\les \E''(u).
\]

\medskip

\noindent \textit{Step \arabic{proof-step}.3.} We finally show that $T[u]$ is tensor-rectifiable.  Since it is a local statement it is enough to prove that around each point $\bar x\in Q^n$ there is a ball $B_r(\bar x)\sub Q^n$ such that $T[u]\restr B_r(\bar x)$ is tensor rectifiable. As $T[u]$ is  normal and rectifiable, $T_{\bar x,r}=T[u]\restr B_r(\bar x)$  is also a normal and rectifiable current for every $\bar x\in Q^n$ and almost every $r>0$ (depending on $\bar x$) such that $B_r(\bar x)\sub Q^n$. In particular $T_{\bar x,r}$ is a rectifiable flat chain (notice however that in general $T_{\bar x,r}$ is not a cycle). Moreover, the tangent $(n-2)\,$-vector $\xi$ to $\Sigma$ can be written as $\xi^1\wedge \xi^2$ with $\xi^l\in \X^l$ so that for $\HH^{n-2}$-every $x\in \Sigma\cap B_r(\bar x)$ the approximate tangent $(n-2)$-plane $T_x \Sigma$ is of the form $L^1(x)\times L^2(x)$ for some the hyperplanes $L^l(x)=(\xi^l)^\perp\cap\X^l$ for  $l=1,2$. Therefore we may appeal to~\cite[Theorem~1.3]{GM_tfc} which yields that $T_{\bar x,r}$ is tensor rectifiable. \medskip

\noindent{\textit{Step \stepcounter{proof-step}\arabic{proof-step}. The general case.}}\\ 
In this final step we use a covering argument to prove the claim in a more general domain $\Om$. For this we need to introduce localized versions of the energy. For any open set $A=A_1+A_2\sub \Om$, we set 
\[
\E''(u,A):= \liminf_{k\up \oo}\sum_{1\le i \le n_1,1\le j\le n_2} \int_{A^{\eps_k}} \dfrac{|D[D u(\cdot, r_k f_j)](x,r_k e_i)|^\te}{ r_k^2} \,dx.\]
We then set for $u\in L^\oo(\Om)$, $u_A:=u\un_A\in \DD_n(A)$ so that as elements of $\DD_{n-2}(A)$ (this is of course not true in $\DD_{n-2}(\Om)$), 
\[
 T[u_A]=T[u]\restr A.
\]
Let $\mathcal{Q}'$ be a Whitney partition of $\Om$, see e.g.~\cite[Appendix J]{Grafakos}. By definition, there exists $\lambda=\lambda(d)>0$ such that 
\[
 \mathcal{Q}=\{\lambda Q \ : \ Q\in \mathcal{Q'}\}
\]
is a cover of $\Om$ with finite overlap, \ie
\be\label{finiteoverlap}
 \un_\Om\le \sum_{Q\in\mathcal{Q}} \un_Q\les  \un_\Om.
\ee
Therefore, letting 
\[
 F_k(x)=\sum_{1\le i \le n_1,1\le j\le n_2}\dfrac{|D[D u(\cdot, r_k f_j)](x,r_k e_i)|^\te}{ r_k^2},
\]
we have
\begin{multline*}
 \E''(u)=\liminf_{k\up \oo} \int_{\Om^{\eps_k}} \un_\Om(x) F_k(x) \,dx
 \ges \liminf_{k\up \oo}\sum_{\mathcal{Q}} \int_{\Om^{\eps_k}} \un_Q(x) F_k(x) \,dx\\
 \ge \liminf_{k\up \oo}\sum_{\mathcal{Q}} \int_{Q^{\eps_k}} F_k(x) \,dx
 \ge \sum_{\mathcal{Q}} \liminf_{k\up \oo}\int_{Q^{\eps_k}} F_k(x) \,dx
 =\sum_{\mathcal{Q}} \E''(u,Q).
\end{multline*}
As a consequence, for every $Q\in \mathcal{Q}$ we have $\E''(u_Q)=\E''(u,Q)<\oo$ so that by {\it Step 2}, $T[u_Q]$ is tensor rectifiable with 
\[
 \M_\theta(T[u_Q])\les \E''(u,Q).
\]
On the one hand this yields that  $T[u]$ is also tensor rectifiable. On the other hand, by~\eqref{finiteoverlap},
\[
 \M_\theta(T[u])\le \sum_{ \mathcal{Q}} \M_\theta(T[u]\restr Q)=\sum_{ \mathcal{Q}} \M_\theta(T[u_Q])\les \sum_{ \mathcal{Q}} \E''(u,Q)\les \E''(u). 
\]
This concludes the proof.
\end{proof}

\section{The case $\theta=1$}\label{sec:theta=1}
When $\theta=1$, we only consider the case $n_1=n_2=1$. In this last section we assume without loss of generality that $\Om=(-1,1)^2$.
\subsection{The general case}\label{sec:theta=1general}
In the case $\te=1$ (and $n_1=n_2=1)$ the defect measure does not in general concentrate on a set of Hausdorff dimension $n_1+n_2-2 =0$. The examples of Remark~\ref{rq:alldim} show that it may concentrate on a set with Hausdorff dimension~$s$ for any $s\in [0,1]$. We establish that 1 is the largest possible dimension provided $u$ is  integrable and $\E(u)<\oo$.
\begin{proof}[Proof of Theorem~\ref{theo:lines}]~
Let  $u\in L^1_{loc}(\Om)$ be such that  $\E(u)<\oo$. By~\eqref{controlmu}, we know that $\mu=\pt_1\pt_2 u$ is a measure with
 $|\mu|(\Om)\lesssim \E(u)$. 
By Lemma~\ref{lemma:4points} we have for $k\ge 1$ and almost every $x\in \Om^{\eps_k}$, 
\be\label{id4pts2}
\mu(Q_{x,r_k}) =Du(x+r_k e_2,r_k e_1) - Du(x,r_ke_1) \qquad \text{and } \qquad |\mu|(\pt Q_{x,r_k})=0. 
\ee
Recalling the definition of $q(x,z)$ in~\eqref{defq} and the subadditivity of $s\mapsto s^{\te_l}$,  we deduce that
\[
|\mu(Q_{x,r_k})|^{\te_1} (|Du(x,z^k_2)| +|Du(x+z^k_1,z^k_2)|)^{\te_2}\le q(x,z^k).
\]
Here we used the notation $z^k_1:=r_k e_1$, $z^k_2:=r_k e_2$ where $(e_1,e_2)$ is the standard basis of $\R^2$. Denoting $z^k:=z^k_1+z^k_2=(r_k,r_k)$, the inequality~\eqref{eq:goodqk} then leads to 
\be\label{firstestimlines}
  \limsup_{k\up \oo} \int_{\Om^{\eps_k}} \dfrac{|\mu(Q_{x,r_k})|^{\te_1} (|Du(x,z^k_2)| +|Du(x+z^k_1,z^k_2)|)^{\te_2}}{r_k^2} \,dx\les \E(u).
\ee
The rest of the proof is of the same flavor as the proof of Lemma~\ref{lem:Diracmeasure} but the construction of the partitions of $\Om$ into rectangles in Steps 2.1---2.3 below is more involved.\medskip\\
{\it Step 1.}
For $y\in\Om$, we define 
 \[
N_k(y):= \sup\lt\{ \dfrac{|\mu|(Q_{x,r_k})}{r_k}: x\in \Om^{\eps_k}\text{ such that } y\in Q_{x,r_k}\rt\}.
 \]
Then for $\eta>0$, we consider the set  
\[
\Sigma_\eta:=\lt\{y\in \Om\, :\, \limsup_{k\up \oo} N_k(y)\ge \eta\rt\}.
\]
By Besicovitch covering theorem~\cite[Theorem~2.17]{Am_Fu_Pal}, we have
$\HH^1(\Sigma_\eta)\lesssim |\mu|(\Om)/\eta$. Next, the sets $\{\Sigma_\eta\}_{\eta>0}$ form a decreasing family  of Borel sets   and their union is  
\[
\Sigma:=\lt\{y\in \Om\, :\, \limsup_{k\up \oo} N_k(y)>0\rt\}. \]
We deduce that $\Sigma$ is a Borel subset of $\Om$ and that the measure $\HH^1\restr \Sigma$ is $\sigma\,$-finite. \medskip

\noindent
{\it Step 2.}  
Let us show that $|\mu|(\Om\bks \Sigma)=0$. Let $\eps>0$. By definition, a point $y$ of $\Om^\eps$ belongs to $\Om\bks \Sigma$ if and only if 
\be
\label{scdestimlines}
\lim_{k\up \oo}\lt[ \sup \lt\{ \dfrac{|\mu|(Q_{x,r_k})}{r_k}\, :\, Q_{x,r_k}\sub \Om\text{ with }y\in Q_{x,r_k} \rt\}\rt] = 0.
\ee

By Egoroff theorem,  there exists a measurable set $\mc Y \sub \Om^\eps\sm\Sigma$ with 
\[
|\mu|(\mc Y)\ge |\mu|( \Om^\eps\sm\Sigma)/2
\]
such that~\eqref{scdestimlines} holds uniformly in $\mc Y$. In particular, there exists a sequence $t_k>0$ with $t_k\dw0$ and 
\be
\label{thirdestimlines}
|\mu|(Q_{x,r_k}) \le t_k r_k\quad\text{for every }x\text{ such that }Q_{x,r_k}\cap \mc Y\neq\void\text{ and every }k\ge1.
\ee 
Substituting $\max(t_k,r_k)$ for $t_k$ we assume without loss of generality that $t_k\ge r_k$.\medskip

\noindent
{\it Step 2.1. Covering of $\mc Y$.} Let us fix $k\ge 1$ and  $x\in Q_{0,r_k}$ and let us consider the sets 
{\setlength{\jot}{4pt}
\begin{align*}
 {\mc P}^k_x&:=\{ Q\ : \ Q=Q_{x+r_k z,r_k} \ \text{ for some  }z\in \Z^2\text{ such that }x+r_k z\in \Om^{\eps_k}\},\\
 \QQ^k_x&:=\{ Q\in {\mc P}^k_x : Q\cap \mc Y\neq \void\}.
\end{align*}
}
We then define $\mc Y^k_x:=\cup  \QQ^k_x$. Notice that by~\eqref{thirdestimlines},
\be
\label{fourthestimlines}
\text{for } k\text{ large enough, }\mc Y\sub \mc Y^k_x\text{ for every }x\in Q_{0,r_k}.
\ee
Our task is now to build a covering of some $\mc Y^k_x$ (with $x$ depending on $k$) by a collection of boxes $Q$ such that $\sum|\mu(Q)|$ tends to $0$ as $k$ goes to $+\oo$.   For this, we introduce a large number $\Lambda\ge 1$ and we cover $\QQ^k_x$ with a disjoint union, $\QQ^k_x
\sub{\mc G}^k_x\cup \mathcal{B}^k_x$ defined as follows. For $Q$ of the form $Q_{y,r_k}$, let us note  $y=x_Q$ its bottom left corner. We set, 
{\setlength{\jot}{5pt}
\begin{eqnarray}
\label{defGk}
{\mc G}^k_x & := & \lt\{Q\in  \QQ^k_x \, :\, |Du(x_Q,z^k_2)| +|Du(x_Q+z^k_1,z^k_2)| \ge \Lambda t_k r_k\rt \},\\
\label{defBk}
\mathcal{B}^k_x& := & \lt\{Q\in  \QQ^k_x\, :\,  |Du(x_Q,z^k_2)| +|Du(x_Q+z^k_1,z^k_2)|  < \Lambda t_k r_k\rt \}.
\end{eqnarray}
}

\noindent
{\it Step 2.2. Estimation of $|\mu|$ on the good set ${\mc G}^k_x$ and selection of $x=x^k$.}\\
 We first bound the average over $x$ of the sums $\sum_{{\mc G}_x^k}|\mu(Q)|$. By~\eqref{defGk} and using $\te= \te_1+\te_2=1$ we get,
 \begin{multline*}
\dfrac1{r_k^2} \int_{Q_{0,r_k}} \sum_{Q\in{\mc G}_x^k}|\mu(Q)|\,dx \\
\le \dfrac1{r_k^2} \int_{Q_{0,r_k}} \sum_{Q\in{\mc G}_x^k}|\mu(Q)|^{\te_1}|\mu(Q)|^{\te_2} \dfrac{\lt(|Du(x_Q,z^k_2)| +|Du(x_Q+z^k_1,z^k_2)|\rt)^{\te_2}}{(\Lambda t_k r_k)^{\te_2}} \,dx.
\end{multline*}
By~\eqref{fourthestimlines} we can use~\eqref{thirdestimlines} for $k$ large enough. We obtain,  
 \begin{align*}
\dfrac1{r_k^2} \int_{Q_{0,r_k}} \sum_{Q\in{\mc G}_x^k}&|\mu(Q)|\,dx \\
&\le \, \dfrac1{\Lambda^{\te_2}} \dfrac1{r_k^2} \int_{Q_{0,r_k}} \sum_{Q\in{\mc P}_x^k}|\mu(Q)|^{\te_1}( |Du(x_Q,z^k_2)| +|Du(x_Q+z^k_1,z^k_2)|)^{\te_2} \,dx\\ 
&= \dfrac1{\Lambda^{\te_2}} \int_{\Om^{\eps_k}} \dfrac{|\mu(Q_{y,r_k})|^{\te_1} (|Du(y,z_2^k)| +|Du(y+z_1^k,z_2^k)|)^{\te_2}}{r_k^2} \,dy,
 \end{align*}
where we used Fubini for the last identity. By~\eqref{firstestimlines} the last integral is of the order of $O(\E(u)/\Lambda^{\te_2})$ as $k\up \oo$. Hence, there exists a sequence $x^k\in Q_{0,r_k}$ such that 
\be\label{forgotestimlines}
\limsup_{k\up \oo}  \sum_{Q\in{\mc G}_{x^k}^k}|\mu(Q)|\, \les \,\dfrac{\E(u)}{\Lambda^{\te_2}}.
\ee
Moreover,  we may assume that~\eqref{id4pts2} holds for every square $Q\sub \Om$ of the form $Q_{x^k+r_k z ,r_k}$ for some $z\in \Z^2$. 
 
From now on, we select $x=x^k$ and we \,drop the subscripts $x^k$: we write ${\mc P}^k$ for ${\mc P}_{x^k}^k$, $\QQ^k$ for $\QQ_{x^k}^k$, ${\mc G}^k$ for ${\mc G}_{x^k}^k$ and $\mathcal{B}^k$ for $\mathcal{B}_{x^k}^k$.\medskip

\noindent
{\it Step 2.3. Covering of $\cup\mathcal{B}^k$}.\\
Let $z_2\in \Z$, we denote ${\mc P}^k_{z_2}$ the row of squares $Q\in{\mc P}^k$ such that $x_Q=x^k+r_k(z_1,z_2)$ for some $z_1\in \Z$. The set ${\mc P}^k_{z_2}$ is totally ordered  by the relation ``$<$'' defined by
\[
Q_{y_k+r_k(z_1,z_2)}< Q_{y_k+r_k(z_1',z_2)}\qquad\text{ whenever }\,z_1< z_1'.
\]
\begin{figure}[h]
\centering
\begin{tikzpicture}[scale=.6]
\newcommand{\carre}[2]{\draw[fill, color=gray!45]  (#1,#2) rectangle (#1 + 1,#2 + 1); }
\clip (0.3,0.4) rectangle (19.6,11.4);
\carre12 \carre22 
\foreach \x in {7,...,10}\carre{\x}2
\carre{5}{4} \carre{6}{4} \carre{7}{4}
\foreach \x in {12,...,16}\carre{\x}{5}
\carre2{7} \carre{3}{7} \carre{6}{7} \carre{7}{7} \carre{8}{7} \carre{12}{7} \carre{15}{7} \carre{17}{7} \carre{18}{7}
\draw[very thick] (2,7) rectangle (9,8); \draw[very thick] (12,7) rectangle (19,8);
\draw (5.5,8.5) node{$R_{z_2,1}$}; \draw (15.5,8.5) node{$R_{z_2,2}$};
\draw (2.6,6.5) node {$Q_1^-$}; \draw (8.5,6.5) node {$Q_1^+$};
\draw (12.6,6.5) node {$Q_2^-$}; \draw (18.5,6.5) node {$Q_2^+$};
\carre{3}{9}\carre{7}{9} \carre{8}{9}   \carre{15}{9} \carre{16}{9} \carre{17}{9} \carre{18}{9}
\draw (5.5,2.5) node{$\mathcal{B}_{z_2,1}$}; \draw (15.5,2.5) node{$\mathcal{B}_{z_2,2}$};
\draw[->] (5.1,2.9) -- (3.5,7.5); \draw[->] (5.3,2.9) -- (7.5,7.5);
\draw[->] (15.1,2.9) -- (12.5,7.5); \draw[->] (15.2,2.9) -- (15.5,7.5); \draw[->] (15.3,2.9) -- (17.9,7.5);
\foreach \x in {0,...,20}
	\draw[very thin] (\x,0)--(\x,12);
\foreach \y in {0,...,12}
	\draw[very thin] (0,\y)--(20,\y);
\end{tikzpicture}
   \caption{The covering $\mathcal{B}$ of the bad set. The gray squares are the elements of $\mathcal{B}^k$. \label{Fig:badset}}
\end{figure}
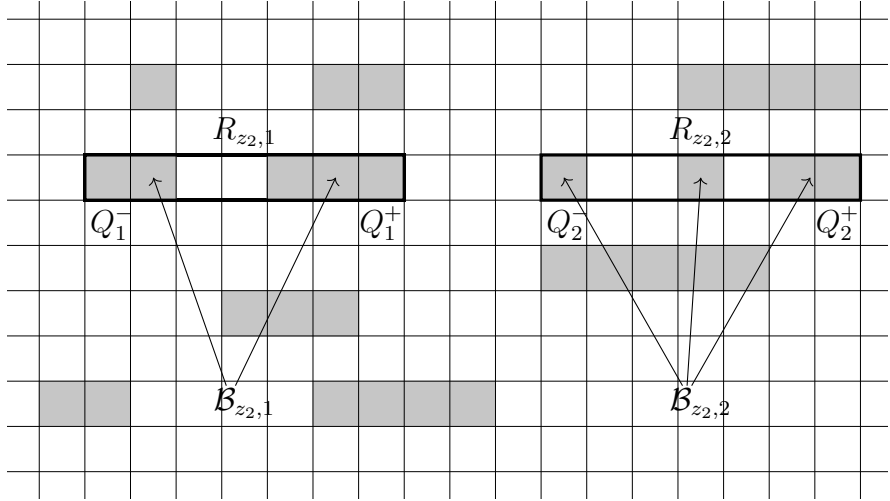
%
\noindent
 Let us define 
\[
\mathcal{B}^k_{z_2}:=\mathcal{B}^k\cap {\mc P}^k_{z_2}.
\]
It is easy to see that $\mathcal{B}^k_{z_2}$ admits a partition into a sequence $\mathcal{B}^k_{z_2,1},\cdots,\mathcal{B}^k_{z_2,s_{z_2}}$ with the following properties.
\begin{enumerate}[(a)]
\item
$s_{z_2}\lesssim 1/(\Lambda^2t_k)$;
\item 
every subset $\mathcal{B}^k_{z_2,s}$ satisfies $\diam\lt(\bigcup_{\mathcal{B}^k_{z_2,s}}Q\rt) \lesssim \Lambda^2 t_k$;
\item
the sets $\mathcal{B}^k_{z_2,1},\cdots,\mathcal{B}^k_{z_2,s_{z_2}}$ are ordered as follows: if $1\le s_1 < s_2\le s_{z_2}$ then $\max \mathcal{B}^k_{z_2,s_1}<\min \mathcal{B}^k_{z_2,s_2}$.
\end{enumerate}
 Let $s\in\{1,\cdots ,s_{z_2}\}$ we denote $Q^-_s=\min \mathcal{B}^k_{z_2,s}$, $Q^+_s=\max \mathcal{B}^k_{z_2,s}$ and we define the box $R_{z_2,s}$ by gluing together the elements of $\QQ^k_{z_2}$ between $Q^-_s$ and $Q^+_s$, namely (see Figure~\ref{Fig:badset}),
\[
R_{z_2,s}:=
\bigcup\lt\{Q\in \QQ^k_{z_2}:  Q^-_s \le Q\le  Q^+_s\rt\}.
\]
By construction the sets $R_{z_2,s}$ are disjoint, satisfy $\diam(R_{z_2,s})\les \Lambda^2 t_k$, their number, for $z_2$ fixed is $s_{z_2}\lesssim 1(\Lambda^2t_k)$ (so that their total number is estimated by  $1/(r_k \Lambda^2t_k)$)
and their union covers $\cup\mathcal{B}^k_{z_2}$. Moreover, for $s$ fixed, denoting $x^-_s=x_{Q^-_s}$, $x^+_s=x_{Q^+_s}$, we have 
\begin{multline*}
|\mu(R_{z_2,s})| \, \stackrel{\eqref{id4pts}}=\, 
|u(x^-_s) - u(x^+_s+z_1^k) +  u(x^+_s+z^k) - u(x^-_s+z_2^k)  |
\\
=\, \lt|Du(x^+_s+ z_1^k, z_2^k) - Du(x^-_s, z_2^k) \rt| 
\, \st{\eqref{defBk}}\le \, 2 \Lambda t_k r_k,
\end{multline*}
by triangle inequality and because ${Q^-_s}$ and ${Q^+_s}$ belong to $\mathcal{B}^k$.  Denoting $\widetilde{\mathcal{B}}^k$ the collection of the sets $R_{z_2,s}$ for $z_2\in \Z$ and $1\le s\le s_{z_2}$, we deduce from the above discussion that 
\be\label{fifthestimlines}
\sum_{R\in\widetilde{\mathcal{B}}^k}|\mu(R)|\,\lesssim\, \dfrac1{r_k}\, \dfrac1{\Lambda^2 t_k}\, \Lambda t_k r_k\, \lesssim\, \dfrac1\Lambda,\qquad\text{and}\qquad \max\limits_{R\in \widetilde{\mathcal{B}}^k} \diam(R)\,\les\,\Lambda^2 t_k.
\ee
Moreover, by construction $\bigcup\mathcal{B}^k\,\sub\, \bigcup\widetilde{\mathcal{B}}^k$.\medskip\\
Eventually we define $\widetilde{{\mc G}}^k$ as the set of elements $Q\in {\mc G}^k$ such that $Q\not\sub  \bigcup\widetilde{\mathcal{B}}^k$ and  set $\widetilde{\QQ}^k:=\widetilde{{\mc G}}^k\cup \widetilde{\mathcal{B}}^k$. \medskip

\noindent
{\it Step 2.4. Sending $k\up \oo$.}  By construction the elements of $\widetilde{\QQ}^k$ are disjoint rectangles in $\Om$ and recalling~\eqref{fourthestimlines},~\eqref{forgotestimlines},\eqref{fifthestimlines}, we have 
\be\nonumber
\mc Y \,\sub\, \bigcup\widetilde{\QQ}^k,\qquad \sum_{Q\in\widetilde{\QQ}^k}|\mu(Q)|\lesssim\dfrac{\E(u)}{\Lambda^{\te_2}}+\dfrac1\Lambda \quad\text{and}\quad \max\limits_{Q\in \widetilde{\QQ}^k} \diam(Q)\,\st{k\up \oo}\longto\,0.
\ee
Applying Lemma~\ref{lemma:tech} to the measure $\mu$ and to the family $\{\widetilde{\QQ}^k\}$, we obtain
\[
|\mu|(\mc Y)\,\le\, \liminf_{k\up \oo} \sum_{Q\in \widetilde{\QQ}^k}|\mu(Q)|\, 
\,\les\, \dfrac{\E(u)}{\Lambda^{\te_2}} + \dfrac1\Lambda.
\]
Since $\Lambda\ge 1$ is arbitrary, we obtain that $|\mu|(\Om^\eps\sm \Sigma)\le 2 |\mu|(\mc Y)=0$ and sending $\eps$ to 0 we conclude that $\mu=\mu\restr \Sigma$. Recalling Step 1, $\HH^1\restr \Sigma$ is $\sigma\,$-finite and the theorem is proved.
\end{proof}

\subsection{The case of Lipschitz continuous functions}\label{sec:Lip}
From now on we assume that $u$ is Lipschitz continuous with $\|\nb u\|_\oo\le 1$. The results of this section are stated for $v=\nb u$ instead of $u$ (we always assume that $u$ and $v$ are related by the identity $\nb u=v$). \smallskip

\noindent
Let us recall that $\Om=(-1,1)^2$ and  letting $K=(\R\times\{0\})\cup (\{0\}\times \R)$, we are interested in the mappings $v\in L^{\oo}(\Om,\R^2)$ such that 
\begin{align}\label{curl0main}
&\nb\times v=\pt_1 v_2-\pt_2 v_1=0,&\\
\label{inKmain}
 &v(x)\in K \text{ for almost every }x\in\Om.&
\end{align}
Setting $\mu[v]=\pt_1 v_2$ (which coincides with $\mu[u]$) we consider the set\footnote{This set was already defined in the introduction.} 
\[
S^\oo(\Om):= \{v\in L^\oo(\Om,\R^2) : \|v\|_\oo\le 1,\ \text{~\eqref{curl0main},\eqref{inKmain} hold and }\mu[v]\in \MM(\Om)\}.
\]

\subsubsection{Derivation of the differential inclusion}\label{sec:DerivdifLip}
Our first result states that if $u$ is Lipschitz continuous with $\|\nb u\|_\oo\le 1$ and $\E(u)<\oo$ then $\nb u\in S^\oo(\Om)$. We actually prove a stronger statement.
\begin{proposition}\label{prop:differincl}
 Let $u_\eps$ be a sequence of $1\,$-Lipschitz functions such that
 \[\E_0:=\liminf_{\eps\to 0} \E_\eps(u_\eps)<\oo.
 \]
Then, up to extraction and subtraction of constants, $u_\eps$ converges uniformly to a $1\,$-Lipschitz function $u$. Moreover $\nb u\in S^\oo(\Om)$ with $|\mu[\nb u]|(\Om)\les\E_0$.
\end{proposition}
\begin{proof}~
Since the functions $u_\eps$ are $1\,$-Lipschitz, up to extraction and subtraction of constants, they converge uniformly to a $1\,$-Lipschitz function $u$. Let $v:=\nb u$.
Using Fatou and the  rescaling  $z=\eps z'$ we obtain 
\[
 \int_{\R^2}\dfrac{\rho(z)}{|z|^2} \lt[\liminf_{\eps\to 0} \int_{\Om^\eps} \dfrac{|Du_\eps(x,\eps z_1)|^{\theta_1}|Du_\eps(x,\eps z_2)|^{\theta_2}}{\eps^2}\,dx\rt] dz
 \le\E_0. 
\]
Arguing as in~\cite[Proposition~M]{GM1}, we find the  existence of  $\sigma_1\in \X_1\sm\{0\}$ and $\sigma_2\in \X_2\sm\{0\}$ such that 
\begin{align}\label{enereps}
 \liminf_{\eps\to 0}\int_{\Om^\eps}\dfrac{|Du_\eps(x,\eps\sigma_1)|^{\theta_1}|Du_\eps(x,\eps\sigma_2)|^{\theta_2}}{\eps^2}\,dx&\les \E_0,\\
\label{cross}
 \liminf_{\eps\to 0}\int_{\Om^\eps}\dfrac{|D[Du_\eps(\cdot, \eps \sigma_1)](x,\eps\sigma_2)|}{\eps^2}\,dx&\les\E_0.
\end{align}

\noindent
\textit{Step 1. $\mu[v]$ is a measure.}\\
 We now prove that $\mu[v]\in \MM(\Om)$  (with $|\mu[v]|(\Om)\les \E_0$). Fix $\vhi\in C^{\oo}_c(\Om)$, set $z^\eps:=- \eps (\sigma_1+\sigma_2)$ and consider
\[
 \eta_\eps:=\dfrac1{|\sigma_1||\sigma_2|\eps^2}\un_{Q_{0,-z^\eps}}.
\]
Then $u_\eps*\eta_\eps$ still converges to $u$ uniformly and we have 
\begin{align*}
\lt<\mu[v],\vhi\rt>= \int_\Om u\,\pt_1\pt_2 \vhi&=\lim_{\eps\to 0} \int_{\Om^\eps} (u_\eps* \eta_\eps) \pt_1\pt_2 \vhi\\
 &=\lim_{\eps\to 0} \int_{\Om^\eps}\dfrac1{|\sigma_1||\sigma_2|\eps^2} \lt[\int_{\Om^\eps} \un_{Q_{y,z^\eps}}(x) u_\eps(y)\,dy\rt] \pt_1\pt_2\vhi(x)\,dx\\
&= \lim_{\eps\to 0} \int_{\Om^\eps} \dfrac1{|\sigma_1||\sigma_2|\eps^2}\lt[\int_{Q_{y,z^\eps}} \pt_1\pt_2\vhi(x)\,dx \rt] u_\eps(y)\,dy \\
&=\lim_{\eps\to 0} \int_{\Om^\eps} \dfrac1{|\sigma_1||\sigma_2|\eps^2}D[D\vhi(\cdot,-\eps \sigma_1)](y,-\eps\sigma_2)\,u_\eps(y)\,dy \\
 &=\lim_{\eps\to 0} \int_{\Om^\eps} \dfrac1{|\sigma_1||\sigma_2|\eps^2}D[Du_\eps(\cdot,\eps \sigma_1)](x,\eps\sigma_2)\, \vhi(x)\,dx.
\end{align*}
In combination with~\eqref{cross} we obtain that indeed $\mu[v]\in \MM(\Om)$ with $|\mu[v]|(\Om)\les \E_0$.\medskip

\noindent\textit{Step 2. Proof that $v(x)\in K$ almost everywhere.}\\
Let us denote here $\om:=Q^2=(0,1)^2$. We claim the following.
\begin{claim} Let $\alpha\in(0,1)$, there exist $\eps_0,c,\delta>0$ satisfying the following property.\\
 Let $p=(p_1,p_2)\in \R^2$ with  $|p_1|,|p_2|\ge \alpha$ and let $u:\om\to\R$ be a $1$-Lipschitz function. If 
\be\label{hyp1:diffincl}
|u(x)-p\cdot x|\le\delta\qquad\text{for every }x\in \om,
\ee
then there holds, for $0<\eps<\eps_0$,
\be
\label{hyp2:diffincl}
\int_{\om^{\eps|\sigma|}} \dfrac{|Du(x,\eps\sigma_1)|^{\theta_1}|Du(x,\eps\sigma_2)|^{\theta_2}}{\eps^2} \,dx+\int_{\om^{\eps|\sigma|}}\dfrac{|D[Du(\cdot, \eps \sigma_1)](x,\eps \sigma_2)|}{\eps^2}\,dx\ge c.
\ee
\end{claim}

\noindent\textit{Step 2.1. Proof that $v(x)\in K$ almost everywhere, assuming the claim.}\\
Let us first derive from the claim the conclusion $v\in K$ almost everywhere. For this we argue by contradiction and  assume this is not the case. Recalling that $v\notin K$ means that $v_1 v_2\neq 0$ we see that  there exists $0<\alpha<1$ such that the set $M=\{|v_1|>\alpha, |v_2|>\alpha\}$ has positive Lebesgue measure. For $x\in Q^2$ such that $\bar x+ rx\in\Om$ we set 
\[
u^{r,\bar x}(x)=\dfrac{u(\bar x+ rx)-u(\bar x)}r\qquad\text{and similarly}\qquad u_\eps^{r,\bar x}(x)=\dfrac{u_\eps(\bar x+ rx)-u_\eps(\bar x)}r.
\] 
Then, at every point of differentiability $\bar x$ of $u$ we have
\[
\lim_{r\to 0}|u^{r,\bar x}(x)-v(\bar x)\cdot x|= 0\qquad\text{uniformly with respect to }x.
\]
By Egoroff we may further assume that this limit is uniform in $M$. Let $\delta>0$ be given by the claim. We deduce that for $r$ small enough, there holds  for every $\bar x\in M$ and every $x\in Q^2$ such that $\bar x+ rx\in\Om$, 
\[
 |u^{r,\bar x}(x)-v(\bar x)\cdot x|\le \delta/2.
\]
Let us fix such $r$. Since $\|u_\eps-u\|_\oo\to0$ we deduce that there exists $\eps_1>0$ depending on $\delta$ and $r$ such that for $0<\eps<\eps_1$, we have 
\[
 |u_\eps^{r,\bar x}(x)-v(\bar x)\cdot x|\le \delta
\]
for $\bar x\in M$ and $x\in Q^2$ such that $\bar x+ rx\in\Om$.\\
Therefore, $u_\eps^{r,\bar x}$ satisfies~\eqref{hyp1:diffincl} with $p=v(\bar x)$ and moreover $\min(|p_1|,|p_2|)\ge\alpha$ as in the statement of the  claim. Furthermore, since $u$ is $1\,$-Lipschitz, the functions $u^{r,\bar x}$ are also $1\,$-Lipschitz.\smallskip

Let us now show that we can find $\bar x\in M$ such that~\eqref{hyp2:diffincl} is violated for $u_\eps^{r,\bar x}$ and $\eps$ small enough. Thanks to the claim, this would provide the desired contradiction. From~\eqref{enereps},~\eqref{cross}, Markov inequality, Fubini and Fatou, we find $\bar x\in M$ such that
\[
\liminf_{\eps\to 0} \int_{Q_{\bar x,r}} \dfrac{|Du_\eps(x,\eps\sigma_1)|^{\theta_1}|Du_\eps(x,\eps\sigma_2)|^{\theta_2}}{\eps^2}\,dx+\int_{Q_{\bar x,r}}\dfrac{|D[Du_\eps(\cdot, \eps \sigma_1)](x,\eps \sigma_2)|}{\eps^2}\,dx\les \E_0 r^2.
\]
Notice that using the change of variable $x=ry$ and denoting $\eps'=\eps/r$, we have 
\[
 \dfrac1r\int_{Q_{\bar x,r}} \dfrac{|Du_\eps(x,\eps\sigma_1)|^{\theta_1}|Du_\eps(x,\eps\sigma_2)|^{\theta_2}}{\eps^2}\,dx=\int_{Q^2} \dfrac{|Du^{r,\bar x}_{r\eps'}(y,\eps'\sigma_1)|^{\theta_1}|Du^{r,\bar x}_{r\eps'}(y,\eps'\sigma_2)|^{\theta_2}}{(\eps')^2}\,dy.
\]
Therefore 
\[
 \liminf_{\eps\to 0}\int_{Q^2} \dfrac{|Du^{r,\bar x}_{r\eps}(x,\eps\sigma_1)|^{\theta_1}|Du^{r,\bar x}_{r\eps}(x,\eps\sigma_2)|^{\theta_2}}{\eps^2}\,dx\les \E_0 r.
\]
Arguing similarly for the second term and choosing $r>0$ and then $\eps>0$ small enough, we get that $u^{r,\bar x}_{\eps r}$ contradicts the claim. Hence $v\in K$ almost everywhere in $\Om$. This establishes the proposition assuming that the claim stated at the beginning of Step~2 holds true.\medskip

\noindent\textit{Step 2.2. Proof of the claim.}\\
As a preliminary remark, notice that if the claim holds for some constants $\alpha,\eps_0,c,\delta,$ it also holds for $\alpha,\eps_0,c',\delta',$ for $0<c'\le c$ and $0<\delta'\le\delta$.\\ 
We argue by contradiction. Let $\alpha\in(0,1)$,  $c>0$ and $\delta>0$. Let $p=(p_1,p_2)$ be such that  $|p_1|,|p_2|\ge\alpha$. Assume that there exists  $u$ such that~\eqref{hyp1:diffincl} holds true but not~\eqref{hyp2:diffincl} for some sequence  $\eps=\eps_k$ with $\eps_k\dw0$.  We show below that this leads to a contradiction.\medskip

Taking into account the preliminary remark we decrease $\delta>0$ or $c>0$ or both if necessary to ensure 
\be\label{delta_c}
\delta<\dfrac\alpha4\qquad\text{ and }\qquad c<\dfrac{\alpha^3}{2^6}.
\ee
Moreover, to simplify a bit the notation we denote $\eps$ for $\eps_k$ and we assume from now on that $\sigma_1=e_1$ and $\sigma_2=e_2$.\medskip

Recalling that assumption~\eqref{hyp2:diffincl} does not hold and using the following identity  valid for every integrable function $f$,
\[
 \int_\om f=\int_{Q_{0,\eps}} \lt(\sum_{z\in (\eps \Z)^2\cap (\om- x)} f(x+z)\rt)\,dx,
\]
we see that there exists $x^\eps\in Q_{0,\eps}$ such that with the notation $x^{(z)}= x^\eps+ z$, we have 
\be\label{prf_of_claim_prop:differincl_0}
 \sum_{z\in (\eps \Z)^2\cap [\om- x^\eps]^\eps} \lt[|Du(x^{(z)},\eps e_1)|^{\theta_1}|Du(x^{(z)},\eps e_2)|^{\theta_2}
 +|D[Du(\cdot, \eps e_1)](x^{(z)},\eps e_2)|\rt]<c.
\ee
Let us assume without loss of generality that $x^\eps=0$. Analogously to Lemma~\ref{lemma:4points}, we set
\be
\label{prf_of_claim_prop:differincl_1}
 w(x)=u(x)+u(0)-u(x_1,0)-u(0,x_2).
\ee
We thus have
\[
u(x)= u_1(x_1)+u_2(x_2)+w(x)
\]
with $u_1(x_1)=u(x_1,0)-u(0)$ and $u_2(x_2)= u(0,x_2)$.
We check by direct computation that for $x\in\om^\eps$ and $s,t\in\R$ with $|s|,|t|\le \eps$ we have the identity
\[
 D[Du(\cdot, s e_1)](x,t e_2)=D[Dw(\cdot, s e_1)](x,t e_2).
\]

\noindent 
Let us now establish the estimate 
\be\label{claimBVweps}
 \sum_{z\in (\eps \Z)^2\cap \om^\eps} \eps [|D w(x^{(z)},\eps e_1)|+|D w(x^{(z)},\eps e_2)|]<2 c.
\ee
We focus for definiteness on the first term. Using~\eqref{prf_of_claim_prop:differincl_1}, we compute 
\[
D w(x^{(z)},\eps e_1)=D u(x^{(z)},\eps e_1) - D u_1(x^{(z)}_1,\eps) = D u(x^{(z)},\eps e_1) - D u(x^{(z)}_1,\eps e_1).
\]
This leads to
\begin{align*}
 \sum_{z\in (\eps \Z)^2\cap\om^\eps} \eps |D w(x^{(z)},\eps e_1)|&= \sum_{z\in (\eps \Z)^2\cap\om^\eps} \eps |D u(x^{(z)},\eps e_1)-D u(x^{(z)}_1,\eps e_1)|\\
 &\le \sum_{z\in (\eps \Z)^2\cap\om^\eps} \eps \sum_{j\in (\eps\Z)\cap (0,x^{(z)}_2-\eps)} |D[Du(\cdot, \eps e_1)](x^{(z)}_1+ je_2,\eps e_2)|\\
 &\le \sum_{z\in (\eps \Z)^2\cap\om^\eps}|D[Du(\cdot, \eps e_1)](x^{(z)},\eps e_2)|\stackrel{\eqref{prf_of_claim_prop:differincl_0}}< c.
\end{align*}
To get the first inequality of the last line, we used Fubini (for sums) and the fact that for $z_1\in\eps\Z$ fixed, the cardinal of the set $\{z_2 \in\eps \Z:(z_1,z_2)\in \om^\eps\}$ is bounded by $1/\eps$.\\ 
Let us now define
\[
 F:=\lt\{z\in (\eps\Z)^2\cap \om^\eps\,:\,|D w(x^{(z)},\eps e_1)|+|D w(x^{(z)},\eps e_2)|\ge\dfrac{2^4 c \eps}{\alpha^2}\rt\}
\]
so that from~\eqref{claimBVweps},
\be\label{volF}
 \HH^0( F)< \dfrac{\alpha^2}{8\eps^2}.
\ee
Next, let 
\[
 E_1:=\lt\{z_1\in(\eps\Z)\cap Q^1 \ : \ |Du_1(z_1,\eps)|\ge \dfrac12 |p_1| \eps\rt\}
\]
and define similarly $E_2$. Let us estimate from below $\HH^0( E_1)$. Recall that  $u_1(x_1)=u(x_1,0)-u(0)$ and thus~\eqref{hyp1:diffincl} implies
\[
 \|u_1-p_1x_1\|_{L^{\oo}(Q^1)}\le \delta
\]
with $|p_1|\ge \alpha$. Moreover since $u_1$ is $1\,$-Lipschitz, there holds $|Du_1(x_1,\eps)|\le \eps$. With these observations, we compute
\begin{multline*}
 |p_1|-\delta\le |u_1(1)-u_1(0)|=\Big|\sum_{z_1\in (\eps \Z)\cap(0,1-\eps)} Du_1(z_1,\eps)\Big|\\
 \le \sum_{E_1} |Du_1(z_1,\eps)|+\sum_{E_1^c} |Du_1(z_1,\eps)|\le \HH^0( E_1) \eps +\dfrac{|p_1|}2 \HH^0( E_1^c) \eps.
 \end{multline*}
Now since $\HH^0( E_1^c)=\eps^{-1}-\HH^0(E_1)$, we deduce that
\be\label{volE}
 \HH^0(E_1)\ge \lt(\dfrac{ \dfrac{|p_1|}2-\delta}{1-\dfrac{|p_1|}2}\rt) \dfrac1{\eps}\,\st{\eqref{delta_c}}>\dfrac\alpha{4\eps}.
\ee
A similar estimate holds also for $\HH^0(E_2)$.  Then by~\eqref{volF} and~\eqref{volE},
\be\label{volgood}
 \HH^0 ((E_1\times E_2)\cap F^c)\ge \HH^0(E_1) \HH^0(E_2) -\HH^0(F)>\dfrac{\alpha^2}{4\eps^2}-\dfrac{\alpha^2}{8\eps^2}=\dfrac{\alpha^2}{8\eps^2}.
\ee
Eventually, let us estimate $ |D u(x^{(z)},\eps e_l)|$ for $l=1,2$ and $z\in (E_1\times E_2)\cap F^c$. Using the the triangle inequality and the definitions of $E_1$, $E_2$ and $F$, we compute 
\begin{multline}\label{goodestim:diffincl}
 |D u(x^{(z)},\eps e_l)|=|D u_l(x^{(z)}_l,\eps)-Dw(x^{(z)},\eps e_l)|\ge |D u_l(x^{(z)}_l,\eps)|-|Dw(x^{(z)},\eps e_l)|\\
 \ge \dfrac{|p_l|}2 \eps - \dfrac{2^4 c}{\alpha^2} \eps \ge \lt(\dfrac\alpha2 - \dfrac{2^4 c}{\alpha^2}\rt)\eps\st{\eqref{delta_c}}>\dfrac{\alpha\eps}4.
\end{multline}
We conclude with~\eqref{prf_of_claim_prop:differincl_0} that
\[
 c>
 \sum_{z\in (E_1\times E_2)\cap F^c}|Du(x^{(z)},\eps e_1)|^{\theta_1}|Du(x^{(z)},\eps e_2)|^{\theta_2}
 \,\st{\eqref{goodestim:diffincl}}\ge\,\HH^0\lt((E_1\times E_2)\cap F^c\rt) \dfrac{\alpha\eps}4
\,\stackrel{\eqref{volgood}}>\,\dfrac{\alpha^3}{2^5\eps} 
\]
which gives a contradiction for $\eps=\eps_k$ small enough. This concludes the proof of the claim and therefore of the proposition. \end{proof}

\subsubsection{Rectifiability of the defect measure}\label{sec:rectifLip}
Let us state the main result of this section which is a detailed version of Theorem~\ref{thm:rectiftheta=1}.
\begin{theorem}\label{theo:rectifLipmain}
 Let $v\in S^\oo(\Om)$ and assume that $\mu=\mu[v]\in\cM(\Om)$. Then, for $|\mu|\,$-\ae$\bar x$, there exists $v^\oo=(v_1^\oo,v_2^\oo)\in(\R\sm\{0\})^2$ such that letting
 \[
 L:=\Span \lt(v_2^\oo,v_1^\oo\rt)\qquad\text{ and }\qquad c(\bar x):=\dfrac{|v_1^\oo| |v_2^\oo|}{|v^\oo|},
 \]
 we have for every $\vhi\in C_c(\R^2)$,
 \be\label{eq:rectifLipclaim}
  \lim_{r\to 0} \dfrac1r\int_{\R^2} \vhi\lt(\dfrac{x-\bar x}{r}\rt) d|\mu|(x)= c(\bar x)\int_L\vhi\, d\HH^1.
 \ee
As a consequence $\mu= m\HH^1\restr \Sigma$ for some $1$-rectifiable set $\Sigma$ and some Borel function $m$ with $|m|=c$. Moreover, denoting
\[
 \nu:= \dfrac{\sign m}{|v^\oo|}(-v_1^\oo,v_2^\oo),
\]
a unit normal to $\Sigma$ at $\bar x$ and then
\[
 V^\oo(y):=\begin{cases}
              (v^\oo_1,0) & \text{if }y\cdot \nu>0,\\
              (0,v_2^\oo) &\text{if }y\cdot \nu <0,
             \end{cases}
\]
we have
\be\label{eq:strongtracesclaim}
 \lim_{r\to 0} \frac1{r^2}\int_{B_r}\lt|v (\bar x+ y)-V^\oo(y)\rt|\,dy=0.
\ee
Therefore $v$ has traces on $\Sigma$.
\end{theorem}
\begin{remark}
 Let us point out that the sign of $m(\bar x)$ cannot be directly computed from the values of $v^\oo$. Indeed, for both $u(x_1,x_2)=\min(x_1,x_2)$ and $\tilde{u}(x_1,x_2)=\max(x_1,x_2)$ we have $v^\oo= (1,1)$ on $\Sigma= \{x_1=x_2\}$ but $\mu[u]=-\mu[\tilde{u}]= (1/\sqrt2)\HH^1\restr \Sigma$.
\end{remark}

We start with some preliminaries. Let $v\in S^\oo(\Om)$ such that $\mu[v]\in\cM(\Om)$ and  $u$ be such that $v=\nb u$.  For $t\in \R$, we let $\Gamma_t=\{u=t\}$
and $\om_t=\{u>t\}$. The function $t\mapsto |\om_t|$ is bounded and decreasing thus have bounded variation\footnote{Since $|\om_t|=|\Om|>0$ for $t<\inf u$, the function $t\in\R\mapsto |\om_t|$ is not integrable over $\R$, so, strictly speaking, it is only locally $BV$.}. By the co-area formula (see for instance~\cite{AlbBianCrip}), we have that for almost every $t$,  $\om_t$ is a set of finite perimeter with, up to a $\HH^1$-negligible set, $\pt \om_t \cap \Om= \Gamma_t$. 
We have moreover
\be\label{boundH1Gamma}
 \int_\R \HH^1(\Gamma_t)\,dt =\int_{\Om} |\nb u|\,\le |\Om|.
\ee

Our first goal is to establish that the measure $\mu$ decomposes naturally on the level sets of $u$. For this purpose we introduce a measure $\tilde \mu$ on $ \Om\times\R$ defined by the property: 
 \be\label{def_mutilde}
\int_{\Om\times \R} \vhi(x,t)\,d\tilde{\mu}(x,t)=\int_{\Om} \vhi(x,u(x))\,d\mu(x)\qquad\text{for }\vhi\in C_c( \Om\times \R).
 \ee
 Notice that the definition makes sense since $u\in \mathrm{Lip}(\Om)\sub C(\Om)$.
 The next result characterizes $\tilde\mu$  it in terms of the family of measures
 \be\label{kappa_t}
 \kappa_t:=\pt_1\pt_2\un_{\om_t}\quad\qquad\text{for }t\in\R.
 \ee 
\begin{proposition}\label{prop:decompmu}
For almost every $t\in\R$ there holds $\kappa_t\in \cM(\Om)$ and we have the identity 
\be\label{decomptildemu}
 \tilde{\mu}= \kappa_t\otimes dt.
\ee
As a consequence, for every Borel set $A\sub \Om$,
\be\label{eq:integN}
 |\mu|(A)=\int_\R |\kappa_t|(A) \,dt.
\ee
\end{proposition}
\begin{proof}~
Since $\Om$ is bounded and $u$ is Lipschitz continuous, after possibly adding a constant, we assume that $\inf u= 0$ and we set $T:=\sup u\ge0$.\medskip

 \setcounter{proof-step}{0}
\noindent{\textit{Step \stepcounter{proof-step}\arabic{proof-step}.}} Let us establish that for every $\vhi\in C^{\oo}_c (\Om)$ and every $\psi\in C^1(\R)$,
 \be\label{decompmu:firsteq}
  \int_{\Om} \vhi(x) \psi(u(x))\,d\mu(x)=\int_0^T \psi(t) \lt<\kappa_t,\vhi\rt>\,dt.
 \ee
For this we first prove that, denoting $\Psi(t):=\int_0^t \psi(s)\,ds$, we have the identity 
 \be\label{decompmu:intermediate}
   \int_{\Om} \vhi(x) \psi(u(x))\,d\mu(x)=\int_{\Om} \Psi(u) \pt_1\pt_2 \vhi\,dx.
 \ee
Since $\vhi\in C^\oo_c(\Om)$, there exists  an open set $\Om'$ compactly supported in $\Om$ and such that $\vhi\in C^\oo_c(\Om')$. Up to replacing $u$ by $\chi u$ where $\chi\in C^\oo_c(\Om)$ with $\chi=1$ on $\Om'$ we may assume that $u$ has compact support in $\Om$. We then let $u_\eps=u*\rho_\eps$ so that $\pt_1\pt_2 u_\eps\rightharpoonup \mu$ as measures, $u_\eps\to u$ in $C(\Om)$ and $\nb u_\eps\to \nb u$ almost everywhere. Therefore
\begin{align*}
\int_\Om\vhi(x)\psi(u(x))\,d\mu(x)&=\lim_{\eps\to 0}\int_\Om\vhi(x)\psi(u_\eps(x))\pt_1\pt_2u_\eps\,dx\\
 &=-\lim_{\eps\to 0}\lt[\int_\Om\pt_1\vhi(x)\psi(u_\eps(x))\pt_2 u_\eps\,dx+\int_\Om\vhi(x)\psi'(u_\eps(x)) \pt_1u_\eps\pt_2u_\eps\,dx\rt].
\end{align*}
On the one hand we have
\[
 \lim_{\eps\to 0}\int_\Om\pt_1\vhi(x)\psi(u_\eps(x))\pt_2u_\eps\,dx=\int_\Om\pt_1\vhi(x)\psi(u)\pt_2u\,dx.
\]
On the other hand, by dominated convergence theorem, we get
\[
\lim_{\eps\to 0}\int_\Om\vhi(x)\psi'(u_\eps(x))\pt_1u_\eps\pt_2u_\eps\,dx=\int_\Om\vhi(x)\psi'(u(x))\underbrace{\pt_1u\,\pt_2u}_{=0\text{ \ae}}\,dx=0.
\]
Thus
\[\int_\Om \vhi(x) \psi(u(x))\,d\mu(x)=-\int_\Om \pt_1 \vhi(x) \psi(u)\pt_2 u\,dx=-\int_\Om \pt_1 \vhi(x) \pt_2 \lt[\Psi(u)\rt]\,dx.\]
Integrating by parts once again we obtain~\eqref{decompmu:intermediate}. Finally, using the layer-cake formula, see~\cite[Theorem~1.13]{liebloss}, we deduce 
\[
 \int_\Om\Psi(u)\pt_1\pt_2\vhi\,dx=\int_0^T \lt[\int_{\om_t}\pt_1\pt_2 \vhi\,dx\rt] \psi(t)\,dt.
\]
By definition of $\kappa_t$ this concludes the proof of~\eqref{decompmu:firsteq}.
\medskip

\noindent{\textit{Step \stepcounter{proof-step}\arabic{proof-step}.}} We now prove~\eqref{decomptildemu}. In light of the definition~\eqref{def_mutilde} of $\tilde\mu$ and of the identity~\eqref{decompmu:firsteq}, we have to check that $\kappa_t$ is a measure for almost every $t$ and that the function $t\mapsto |\kappa_t|(\Om)$ is integrable with respect to the Lebesgue measure on $\R$.\\
Let us disintegrate $\tilde \mu$ along the level sets $\Gamma_t$. We obtain $\tilde \mu=\mu_t\otimes \lambda$ where $\lambda$ is a finite positive measure supported in $\ov{u(\Om)}$ and $|\mu_t|(\Om)=1$ for $\lambda\,$-almost every $t$. For $\vhi\in C^\oo_c(\Om)$ and $\psi\in C^1(\R)$, we have
\[
 \int_0^T\psi(t)\lt(\int_\Om \vhi(x)\,d\mu_t(x)\rt)\,d\lambda(t)=\int_\Om\vhi(x)\psi(u(x))\,d\mu(x)\stackrel{\eqref{decompmu:firsteq}}=\int_0^T \psi(t)\lt<\kappa_t,\vhi\rt>\,dt.
 \]
Therefore, for fixed $\vhi$, we have, as measures, 
\be\label{decompmu:last}
 \lt(\int_\Om \vhi(x)\,d\mu_t(x)\rt)\,d\lambda(t)=\lt<\kappa_t,\vhi\rt>\,dt.
\ee
Decomposing $\lambda$ into absolutely continuous and singular parts with respect to the Lebesgue measure,  we write $\lambda=f(t)\, dt+\lambda^s$. Putting this in~\eqref{decompmu:last} and identifying, we get the identities
\begin{align*}
 f(t)\lt<\mu_t,\vhi\rt>&=\lt<\kappa_t,\vhi\rt>&&\!\!\!\!\!\!\!\!\!\!\!\text{for almost every }t,\\
 \lt<\mu_t,\vhi\rt> \,d\lambda^s(t)&=0&&\!\!\!\!\!\!\!\!\!\!\!\text{as measure}.
\end{align*}
We deduce that $\kappa_t$ is a measure for almost every $t$ with $\kappa_t=f(t)\mu_t$ so that $t\mapsto |\kappa_t|(\Om)$ is integrable. Moreover the contribution of $\mu_t\otimes\lambda^s$ in the disintegration vanishes, hence $\lambda^s=0$. This concludes the proof of identity~\eqref{decomptildemu}.
\end{proof}
We deduce from Theorem~\ref{theo:dirac1}, that for almost every level $t$ the set $\om_t$ is a polygon.
\begin{proposition}\label{prop:polyomegat}
 There exists $\II_{\mathrm{poly}}\sub\R$ of full measure such that for \ae$t\in \II_{\mathrm{poly}}$, the set $\om_t$ is a finite disjoint union of open polygons with sides parallel to the coordinate axes. Moreover, for such $t$, $N(t):=|\kappa_t|(\Om)$ is the number of vertices of $\Gamma_t$ in $\Om$. More precisely,
 \be
 \label{kappa_t-N(t)}
\kappa_t =\sum_{j=1}^{N(t)} \kappa_t(x_t^j)\delta_{x^j_t},
 \ee
 where the $x^j_t$'s are the vertices in $\Om$ of the polygons forming $\om_t$ and $\kappa_t(x_t^j)=\pm1$.\\
 In the following these inner vertices are called corners and we denote 
 \[
 \CC(\om_t):=\{x_t^j:1\le j\le N(t)\}.
 \]
%
\end{proposition}
\begin{remark}~\\
 (i)~Let us point out that  Proposition~\ref{prop:polyomegat} implicitly states that  $\overline{\om}_t$ is also made of a  finite disjoint union of polygons. In particular, this excludes polygons intersecting at a corner as in Figure~\ref{Figure_polyhedron}. In other words, we excluded the multiplicities $\pm 2$ from Theorem~\ref{theo:dirac1}.\smallskip\\
 (ii)  For shortness, in the proof of Proposition~\ref{prop:polyomegat} below we deduce this simplification from the general result\footnote{We could also give a more elementary proof in our context based on the following observation. Since  $u$ is 1-Lipschitz continuous, the erosion rate of $\om_t$ as $t$ increases is bounded from below. Namely, 
\[
\left\{
\begin{array}{rcll}
\om_s\subset& \lt\{x\in\om_t:d(x,(\om_t)^c)\ge s-t\rt\}\,&\subset\subset \om_t&\text{ for }0\le t<s\le T,\smallskip\\
(\om_s)^c\subset&\lt\{x\in(\om_t)^c:d(x,\om_t)\ge t-s\rt\}\,&\subset\subset(\om_t)^c&\text{ for }0\le s<t\le T.
\end{array}
\right.
\]
}~\cite[Theorem~2.5]{AlbBianCrip}.
\end{remark}

\begin{proof}[Proof of Proposition~\ref{prop:polyomegat}]~
By Theorem~\ref{theo:dirac1} applied with $u=\un_{\om_t}$ we have for \ae $t$,
\[
\om_t= \om_t^{\text{poly}}\cup S_t
\]
where:
\begin{enumerate}[$(*)$]
\item[(i)] $\om_t^{\text{poly}}$ is a finite union of polygons with sides parallel to the coordinate axes and which may intersect only at the corners, 
\item[(ii)] $S_t$ is a union of stripes, either all vertical or all horizontal.
\end{enumerate}
Moreover, by \cite[Theorem~2.5]{AlbBianCrip}, for almost every $t$, the closure of the connected components of $\om_t^{\text{poly}}$ are actually disjoint. Besides, by~\eqref{boundH1Gamma} $S_t$ is a finite union of stripes for almost every $t$ and since  $\om_t=\{u>t\}$ is open we conclude that this set is a disjoint finite union of open polygons with sides parallel to the coordinate axes. \\
Eventually, recalling~\eqref{kappa_t} and Proposition~\ref{prop:decompmu}, we have $\kappa_t=\pt_1\pt_2\un_{\om_t}$ which yields  identity~\eqref{kappa_t-N(t)} and the related properties.
\end{proof}


Thanks to Proposition~\ref{prop:decompmu}, studying the measure $\mu=\mu[\nb u]$ reduces to studying the mapping $s\mapsto \kappa_s$. The next lemma allows us to focus on levels $t$ near which $s\mapsto N(s)$ is (almost) constant and such that we can follow individually the trajectories $s\mapsto x_s^j$. Later, establishing the approximate differentiability of these trajectories is an important step in the proof of the 1-rectifiability of $\mu$.\\
Let us introduce some further notation. 
\begin{notation}
From now on for $x\in\R^2$ and $\lambda,r>0$, we denote
\[
Q_r:=(-r,r)^2,\qquad\qquad Q_r(x):=x+Q_r\qquad\text{and}\qquad \lambda Q_r(x):=x+\lambda Q_r.
\] 
Besides, for a measurable set $J\sub\R$, we define the set of points of density of $J$ as
\[
\dens(J):=\lt\{t\in J:\lim_{\eta\dw0}\dfrac{\HH^1(J\cap(t-\eta,t+\eta))}{2\eta}=1\rt\}.
\]
Notice that, as opposed to the standard definition, we enforce $\dens(J)\sub J$. However, with the present definition, the set $J\sm\dens(J)$ is still Lebesgue negligible.
\end{notation}
\begin{lemma}\label{lem:corner}
 Let $t\in \II_{\mathrm{poly}}$ be such that $|\Gamma_t|=0$ (\emph{i.e.} a point of continuity of $t\mapsto|\om_t|$). Then, there exist $\ov r=\ov r(t)$ such that the following statements hold true.
\begin{itemize}
\item[(i)] For every $0<r\le\ov r$ there exists $\eps=\eps(t,r)>0$ such that
\[
1=|\kappa_t(x)|
\le\sum_{y\in Q_r(x)}|\kappa_s(y)|\qquad\ \text{for }s\in (t-\eps,t+\eps)\cap \II_{\mathrm{poly}}\ \text{ and }\ x\in\CC(\om_t).
\]
\item[(ii)] As a consequence, for every such $s$, $N(s)\ge N(t)$.
\item[(iii)] If moreover $t$ is a Lebesgue point of $N$ then, denoting, for $0<r\le \ov r$,
\[
\JJ_t(r):=\lt\{s\in \II_{\mathrm{poly}}:N(s)=N(t)\text{ and }|\kappa_s|(Q_r(x))=1\text{ for every }x\in\CC(\om_t)\rt\},
\]
we have $t\in\dens(\JJ_t(r))$.
\end{itemize}

\end{lemma}
\begin{proof}~
We start with the proof of $(i)$ and $(ii)$. Let $t\in \II_{\mathrm{poly}}$ such that $|\Gamma_t|=0$. We then have $|\om_t\Delta\om_s|\to 0$ as $s\to t$, that is $\un_{\om_s}\to\un_{\om_t}$ in $L^1(\Om)$. Since $\Gamma_t$ is a polygon, there exists $\ov r>0$ such that if $x\in\CC(\om_t)$, there is no other corner of $\om_t$ in $\Qrb(x)$ and the squares $Q_{\bar x}$ for $x\in\CC(\om_t)$ are pairwise disjoint.\\  Now let $0<r\le \ov r$. 
Recalling that $\kappa_s=\pt_1\pt_2\un_{\om_s}$ and that $\un_{\om_s}\to\un_{\om_t}$ in $L^1$ as $s\to t$ we get that $s\in \R\mapsto \kappa_s\in\DD'(\Om)$ is continuous at $t$. Therefore, for every $x\in\CC(\om_t)$,
\begin{align*}
|\kappa_t(x)|&
=\sup_{\stackrel{\zeta\in C^\oo_c(Q_r,[-1,1])}{|\zeta|\le1}} \int \zeta(y-x)\,d\kappa_t(y)\\
&=\sup_{\stackrel{\zeta\in C^\oo_c(Q_r,[-1,1])}{|\zeta|\le1}} \lim_{s\to t}\int \zeta(y-x)\,d\kappa_s(y)
\,\le\,\liminf_{s\to t}|\kappa_s|(Q_r(x)).
\end{align*}
Since the quantities $|\kappa_s|(Q_r(x))$ are integers we see that there exists $\eps>0$ such that for $s\in \II_{\mathrm{poly}}\cap(t-\eps,t+\eps)$ and every corner $x\in\CC(\om_t)$ there holds,
 \be\label{proof_lem:corner_1} 
 |\kappa_t(x)|\le|\kappa_s|(Q_r(x)).
 \ee 
This proves~$(i)$.  Eventually, summing~\eqref{proof_lem:corner_1} over $x\in\CC(\om_t)$ and recalling that the squares $Q_r(x)$ are disjoint, we get
\be\label{proof_lem:corner_2}
N(t)=\sum_{x\in\CC(\om_t)}|\kappa_t|(x)\le\sum_{x\in\CC(\om_t)}|\kappa_s|(Q_r(x))\le\sum_{y\in\CC(\om_s)}|\kappa_s|(y)=N(s),
\ee
which establishes $(ii)$.\medskip

\noindent
Finally, we prove $(iii)$. If $N(s)=N(t)$ the chain of inequalities in~\eqref{proof_lem:corner_2} are identities and we get that~\eqref{proof_lem:corner_1} is also an identity. Thus, assuming that $t\in \II_{\mathrm{poly}}$ is a Lebesgue point of $N$ (notice that by~\eqref{eq:integN}, $N\in L^1( \R)$), we get that $t\in\dens(\JJ_t(r))$ as claimed.
\end{proof}
Even after the preceding lemma, we still need to select further the levels sets. Indeed the continuity of the trajectories $s\mapsto x_s^i$ does not imply the continuity of the shape of the polygons forming $\om_s$ and unfortunately shape changes are associated with non-differentiability of the corner trajectories. The following example illustrates this fact
\begin{example}\label{example:bifurcation}
Let us consider the function defined on $\Om=(-1,1)^2$, by
\[
u(x_1,x_2)=
\begin{cases} 
\quad-x_1&\text{for }x_1\le0,\\
\min(x_1,x_2)&\text{for }x_1>0,\,x_2\ge0,\\
\max(-x_1,x_2)&\text{for }x_1>0,\,x_2<0.
\end{cases}
\]
For this function and $t\in(-1,1)$, the level set $\om_t$ is given by, see Figure~\ref{Fig:uctrex},
\[
 y\in \om_t\iff 
\begin{cases}
\quad y_1<x_1(t)\text{ or }y_2>x_2(t) &\text{if }t<0,\\
x_1<-t\text{ or } \lt[y_1>x_1(t)\text{ and }y_2>x_2(t)\rt]&\text{if }t\ge0.
\end{cases}
\]  
Let us highlight the following facts. 
\begin{enumerate}[(1)] 
\item $\om_t$ is connected for $t<0$, but has two connected components for $t\ge0$. 
\item The set $\om_t$ switches from locally concave to locally convex in the neighborhood of the vertex $x(t)$.
\item In any case, $\CC(\om_t)$ has a single element $x(t)=(|t|,t)$ with $\kappa_t(x(t))=1$ but  the derivative of $t\mapsto x(t)$ jumps at $t=0$. 
\item The function $s\mapsto \HH^1(\Gamma_s)$ is discontinuous at $t=0$.
\end{enumerate}
\begin{figure}[h]
\centering
\begin{tikzpicture}[scale=.8]
\begin{scope}
\clip (-.2,-.2) rectangle (5.2,5.2);
\draw[very thin] (0,0) rectangle (5,5);
\draw[dashed] (5,0) -- (2.5,2.5) -- (5,5); 
\draw[dashed] (2.5,2.5) -- (2.5,5);
\draw (1.25,3) node{$\Om$};
\draw[-{Latex[length=2mm, width=1.3mm]}] (4,2.2)--(4,2.8);
\draw[-{Latex[length=2mm, width=1.3mm]}] (2.95,4.2)--(3.55,4.2);
\draw[-{Latex[length=2mm, width=1.3mm]}] (2.05,1.5)--(1.45,1.5);
\end{scope}
\begin{scope}[xshift=6.8cm]
\clip (-.2,-.2) rectangle (5.2,5.2);
\draw[fill, color=gray!45]  (3,2) -- (3,0) -- (0,0)--(0,5)--(5,5)--(5,2)--cycle;
\draw[very thin] (0,0) rectangle (5,5);
\draw[black] (3,2) node{$\bullet$};
\draw (3,2) node[below right]{$x(t_-)$};
\draw (2,3) node{$\om_{t_-}$};
\end{scope}
\begin{scope}[xshift=13.6cm]
\clip (-.2,-.2) rectangle (5.2,5.2);
\draw[fill, color=gray!45]  (3,5) -- (3,3) -- (5,3)--(5,5)--cycle;
\draw[fill, color=gray!45]  (2,0) -- (2,5) -- (0,5)--(0,0)--cycle;
\draw[very thin] (0,0) rectangle (5,5);
\draw[black] (3,3) node{$\bullet$};
\draw (3,3) node[below]{$x(t_+)$};
\draw (4,4) node{$\om_{t_+}$};
\end{scope}
\end{tikzpicture}
\caption{\label{Fig:uctrex}Left: the vector field $\nb u$. Middle and right: some sets $\om_{t_\pm}$,   $t_-<0<t_+$.}
\end{figure}
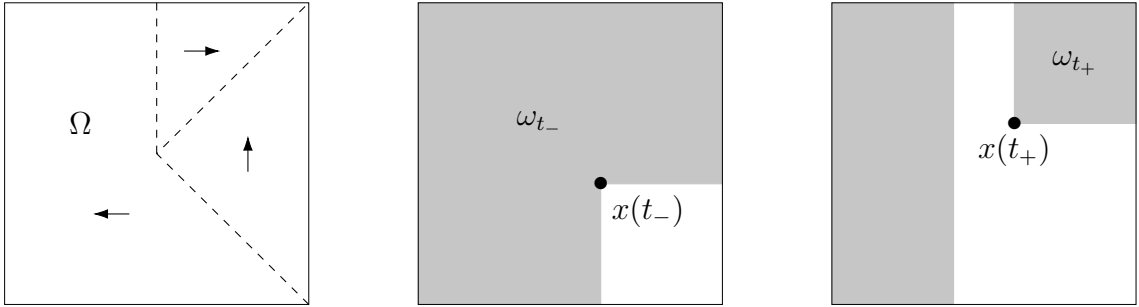
\end{example}
The third fact described in the example is the situation we wish to avoid. To this aim it turns out that the last fact is the most useful. Indeed, the co-area formula~\eqref{boundH1Gamma} provides some control on the quantities $\HH^1(\Gamma_s)$.  The idea is then to exclude the points of discontinuity of $s\mapsto \HH^1(\Gamma_s)$. More precisely, we consider a sequence of local versions of this constraint: we impose that $t$ is a Lebesgue point of $s\mapsto \HH^1(\Gamma_s\cap Q)$ for a dense countable collection of squares $Q$. However, this is still not sufficient to ensure the differentiability of the trajectories of the corners at $t$. For this, we also need to enforce the differentiability of the functions $s\mapsto |\om_s \cap Q|$.  These observations lead to the following definition. 
\begin{definition}\label{def:QQ}
We set
\[
\QQ:=\lt\{Q_r(y)\sub \Om: \text{such that }y_1,y_2\text{ and }r\text{ are rational numbers}\rt\}.
\]
Moreover, for $Q\in\QQ$ we define the function $\Vol_Q:\R\to\R_+$ by
\[
\Vol_Q(s):= |\om_s\cap Q|.
\] 
\end{definition}
The functions $\Vol_Q$  are bounded and nonincreasing, hence $BV$.  We  decompose their distributional derivatives $D\Vol_Q(s)$  in their absolutely continuous and singular parts with respect to the Lebesgue measure:
\[
D\Vol_Q=\Vol_Q'\, dt +D^s\Vol_Q.
\]
Recall that by \cite[Theorem 3.28]{Am_Fu_Pal}, $\Vol_Q$ is  differentiable almost everywhere and its derivative coincides with $\Vol_Q'$.
We are now ready to introduce the subset of levels of $t\mapsto |\om_t|$ meeting the constraints we have just outlined.  
\begin{definition}\label{def:IIQQ}
\begin{multline*}
\II_\QQ:=\big\{t\in \R: \text{ for every }Q\in\QQ,\ \Vol_Q\text{ is  differentiable at }t,\\
\text{ moreover }t\text{ is a Lebesgue point of }\Vol_Q'\text{ and of }s\mapsto \HH^1(\Gamma_s\cap Q)\Big\}.
\end{multline*}
\end{definition}

\begin{lemma}\label{lem:approxdiff}
The set $\II_\QQ$ is of full measure in $\R$.
\end{lemma}
\begin{proof}~
As a consequence of  the co-area formula (\ie~\eqref{boundH1Gamma} with $Q$ in place of $\Om$) the function $s\mapsto  \HH^1(\Gamma_s\cap Q)$ lies in $L^1(\R)$ and the set of its Lebesgue points is of full measure. Recalling that $\QQ$ is countable we get the result.  
\end{proof}

We can then define the set of good levels $t$.
\begin{lemma}\label{lem:goodGLip}
 Let $\II$ be the set of $t$ such that 
\begin{itemize}
  \item[(i)] $t\in \II_{\mathrm{poly}}$\text{ and }$|\Gamma_t|=0$,
 \item[(ii)] $t$ is a Lebesgue point of $s\mapsto N(s)$ (consequently~$t\in\dens(\JJ_t(r))$ for $r\le\ov r(t)$),
 \item[(iii)] $t\in\II_\QQ$,
 \item[(iv)] $\HH^1$-almost every point of $\Gamma_t$ is a Lebesgue point of $\nb u$.
\end{itemize}
Then  $\II$ is of full measure in $\R$. As a consequence (see~\eqref{decomptildemu}), for $|\mu|$-\ae$x$ there exists $t\in \II$ such that $x\in\CC(\om_t)$.
\end{lemma}
\begin{proof}~
Points~$(i)$ and $(ii)$ follow from Lemma~\ref{lem:corner}. Point $(iii)$ from Lemma~\ref{lem:approxdiff}.  The last point is the consequence of $\nb u
\in L^1(\Om)$ and of  the co-area formula $\un_\Om|\nb u|\,dx = (\un_{\Gamma_t}\HH^1)\otimes dt$.
 \end{proof}

We may now embark on the proof of Theorem~\ref{theo:rectifLipmain}. By Lemma~\ref{lem:goodGLip}, for $|\mu|\,$-\ae$\bar x$ there exists $\bar t\in \II$ such that $\bar x\in\CC(\om_{\bar t})$ and $\kappa_{\bar t}(\{\bar x\})=1$. Up to a translation we assume without loss of generality that $\bar t=0$ and $\bar x=0$. Moreover, possibly replacing $u$ by $\pm u\circ R$ where $R$ is a rotation of angle $k\pi/2$ for some $k\in\{0,1,2,3\}$ we assume that for some ${\ov r}>0$  there holds 
\be
\label{perfect_om0}
\om_0\cap\Qrb=\{y\in\Qrb:y_1>0,y_2>0\}.
\ee
From now on, we denote (recall the definition of $\JJ_0(r)$ from point (iii) of Lemma~\ref{lem:corner})
\[
\JJ_0:=\JJ_0(\ov r/4)\cap \II=\lt\{t\in \II : \text{there exists }x\in \Qrbq\text{ such that }\CC(\om_t)\cap\Qrb=\{x\}\rt\},
\]
and for $t\in\JJ_0$, we denote $x(t)=(x_1(t),x_2(t))$ the unique element of $\CC(\om_t)\cap\Qrb$ so that
\be\label{eq:kappaBall}
|\kappa_t|\restr\Qrb=\delta_{x(t)}\qquad \text{ for }t\in\JJ_0. 
\ee
By Lemmas~\ref{lem:corner} \&~\ref{lem:goodGLip}, there holds 
\be\label{0indensJJ0}
0\in\dens(\JJ_0).
\ee

Our first goal is to identify the trajectory $t\in\JJ_0\mapsto x(t)$ as the restriction of a mapping $t\in(-\bar t,\bar t)\mapsto h(t)$ differentiable at $0$. For instance, for the second coordinate $x_2(t)$ we could use the square $Q\sh_0:=(\ov r/2,\ov r)\times (-\ov r/4,\ov r/4)$ and the identity 
\[
x_2(t)=\dfrac{\Vol_{Q\sh_0}(t)-\Vol_{Q\sh_0}(0)}{\ov r/2}=:h_2^0(t),
\] 
which is valid for $t\in\JJ$.
Unfortunately, we have no guarantee on the differentiability of the function $t\mapsto \Vol_Q(t)$ at $0$ for a general square $Q$ unless $Q\in\QQ$.  To overcome this, we  substitute for $Q\sh_0$, a square $Q\sh_+\in\QQ$ such that (see Figure~\ref{Fig:Qhvpm}), 
\be\label{Q_h}
\{\ov r/2\}\times [-\ov r/4,\ov r/4] \ \sub\ Q\sh_+\ \sub \ ({\ov r}/4,{\ov r})\times(-{\ov r},{\ov r}).
\ee
Such square exists by density of $\QQ$ in the set of squares inside $\Om$.
Symmetrically, to control $x_1(t)$ we pick $Q\sv_+\in\QQ$ with
\be\label{Q_v}
 {[-\ov r/4,\ov r/4]}\times\{\ov r/2\} \ \sub\ Q\sv_+\ \sub \ (-{\ov r},{\ov r})\times({\ov r}/4,{\ov r}).
 \ee
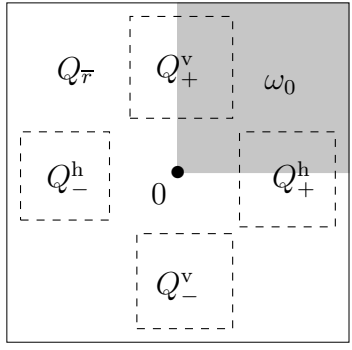
\begin{figure}[h]
\centering
\begin{tikzpicture}[scale=.9]
\clip (-.8,-.2) rectangle (5.2,5.2);
\draw[fill, color=gray!45]  (2.5,2.5) rectangle (5,5);
\draw[thin]  (2.5,2.5)node{$\bullet$} node[below left]{$0$}; 
\draw (1,4) node{$\Qrb$};
\draw (4,3.8) node{$\om_0$};
\draw[thin, dashed] (1.8,3.3) rectangle (3.3,4.8);
\draw (2.5,4) node{$Q\sv_+$};
\draw[thin, dashed] (1.9,.2) rectangle (3.3,1.6);
\draw (2.5,.8) node{$Q\sv_-$};
\draw[thin, dashed] (3.4,1.7) rectangle (4.8,3.1);
\draw (4.2,2.4) node{$Q\sh_+$};
\draw[thin, dashed] (.2,1.8) rectangle (1.5,3.1);
\draw (.9,2.4) node{$Q\sh_-$};
\draw[very thin] (0,0) rectangle (5,5);
\end{tikzpicture}
\caption{\label{Fig:Qhvpm} The set $\om_0\cap \Qrb$ and the four control squares $Q\sh_\pm$, $Q\sv_\pm$ (dashed lines) inside $\Qrb$.}
\end{figure}  
Next, denoting $\ell\sh$ and $\ell\sv$ the respective side lengths of $Q\sh_+$ and $Q\sv_+$, we set for $t\in\R$,
\be\label{defh}
h_1(t):=\dfrac{\Vol_{Q\sv_+}(0)-\Vol_{Q\sv_+}(t)}{\ell\sv}, \qquad\qquad
 h_2(t):=\dfrac{\Vol_{Q\sh_+}(0)-\Vol_{Q\sh_+}(t)}{\ell\sh}
 \ee
and then set $h:=(h_1,h_2)$.
The functions $h_l$ are bounded and nondecreasing, thus $BV$. Besides, from the definition of $\II_\QQ$, they are differentiable at $0$ and writing $Dh_l=h_l'(t)\, dt +D^sh_l$ we have that $0$ is Lebesgue point of $h_l'$ and  
\be\label{property:h2}
h'_l(0)\ge0,\quad\qquad D^sh_l\ge0,\qquad\text{and}\qquad\lim_{\eta\dw0}\dfrac1{2\eta}D^sh_l([-\eta,\eta])=0.
\ee
Next, in order to check the absence of shape changes as in Example~\ref{example:bifurcation}, we introduce two other squares $Q\sh_-, Q\sv_-
\in\QQ$ which are in symmetric position with respect to $Q\sh_-, Q\sv_+$.
\be\label{Q_hQ_v-}
\begin{array}{rcccl}
\{-{\ov r}/2\}\times[-{\ov r}/4,{\ov r}/4] & \sub & Q\sh_- & \sub & (-{\ov r},-{\ov r}/4)\times(-{\ov r},{\ov r}),\medskip\\
{ [-{\ov r}/4,{\ov r}/4]} \times\{-{\ov r}/2\}& \sub & Q\sv_- & \sub &(-{\ov r},{\ov r})\times(-{\ov r},-{\ov r}/4),
\end{array}
\ee
see again Figure~\ref{Fig:Qhvpm}. Eventually we define
\be\label{def:JJ}
\JJ:=\JJ_0\cap \JJ\sh_+\cap\JJ\sv_+\cap\JJ\sh_-\cap\JJ\sv_-,
\ee
where
\[
\begin{array}{rlcrl}
\JJ\sh_+&:=\{t\in\R:\HH^1(Q\sh_+\cap\Gamma_t)=\ell\sh\},
&\qquad\quad&\JJ\sv_+&:=\{t\in\R:\HH^1(Q\sv_+\cap\Gamma_t)=\ell\sv\},
\smallskip\\
\JJ\sh_-&:=\{t\in\R:\HH^1(Q\sh_-\cap\Gamma_t)=0\},
&&\JJ\sv_-&:=\{t\in\R:\HH^1(Q\sv_-\cap\Gamma_t)=0\}.
\end{array}
\]
The Figure~\ref{Fig:trajectory} illustrates the definition.
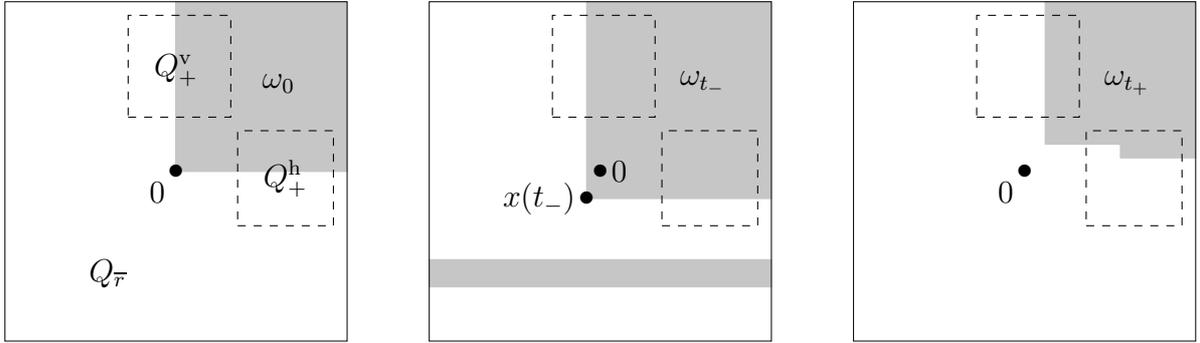
\begin{figure}[h]
\centering
\begin{tikzpicture}[scale=.9]
\begin{scope}
\clip (-.2,-.2) rectangle (5.2,5.2);

\draw[fill, color=gray!45]  (2.5,2.5) rectangle (5,5);
\draw[thin] (2.5,2.5)node{$\bullet$} node[below left]{$0$}; 
\draw (1.5,1) node{$\Qrb$};
\draw (4,3.8) node{$\om_0$};
\draw[thin, dashed] (1.8,3.3) rectangle (3.3,4.8);
\draw (2.5,4) node{$Q\sv_+$};
\draw[thin, dashed] (3.4,1.7) rectangle (4.8,3.1);
\draw (4.1,2.4) node {$Q\sh_+$};
\draw[very thin] (0,0) rectangle (5,5);
\end{scope}

\begin{scope}[xshift=6.2cm]
\clip (-.2,-.2) rectangle (5.2,5.2);

\draw[fill, color=gray!45]  (2.3,2.1) rectangle (5,5);
\draw[thin] (2.3,2.1)node{$\bullet$}node[left]{$x(t_-)$}; 
\draw (2.5,2.5) node{$\bullet$}node[right]{$0$};
\draw (4,3.8) node{$\om_{t_-}$};
\draw[thin, dashed] (1.8,3.3) rectangle (3.3,4.8);
\draw[thin, dashed] (3.4,1.7) rectangle (4.8,3.1);
\draw[fill, color=gray!45]  (0,.8) rectangle (5,1.2);
\draw[very thin] (0,0) rectangle (5,5);
\end{scope}

\begin{scope}[xshift=12.4cm]
\clip (-.2,-.2) rectangle (5.2,5.2);

\draw[fill, color=gray!45]  (2.8,5) -- (2.8,2.9) -- (3.9,2.9)--(3.9,2.7) -- (5,2.7) -- (5,5) -- cycle;   
\draw (2.5,2.5) node{$\bullet$}node[below left]{$0$};
\draw (4,3.8) node{$\om_{t_+}$};
\draw[thin, dashed] (1.8,3.3) rectangle (3.3,4.8);
\draw[thin, dashed] (3.4,1.7) rectangle (4.8,3.1);
\draw[very thin] (0,0) rectangle (5,5);
\end{scope}
\end{tikzpicture}
\caption{\label{Fig:QhQv}Left: the set $\om_0\cap\Qrb$ together with the squares $Q\sh_+$, $Q\sv_+$ (dashed lines). Middle: a set $\om_{t_-}\cap\Qrb$ for some  $t_-<0$ such that $t_-\in\JJ$. Right: a set $\om_{t_+}\cap\Qrb$ for some $t_+>0$ with $t_+\not\in \JJ\sh_+$ (and also $t_+\not\in\JJ_0$).}
\end{figure} 

The following result proves that we have achieved our first objective.

\begin{lemma}\label{lem:step1}
There holds  $x(t)=h(t)$ for $t\in\JJ$, we have $0\in\dens(\JJ)$ and consequently,
\be
\label{dx/dt}
\lim_{\substack{t\to0,\, t\ne0,\,t\in\JJ}} \lt|\dfrac{x(t)}t-h'(0)\rt|=0.
\ee
Moreover, denoting $\Qtrb:=(-\ov r/4,\ov r)^2$, there holds 
\be\label{omtQtrb}
\om_t\cap \Qtrb=(x_1(t),\ov r)\times(x_2(t),\ov r)\qquad\quad\text{for }t\in\JJ.
\ee
\end{lemma}
The example of Figure~\ref{Fig:trajectory} illustrates the lemma.
\begin{figure}[H]
\centering
\begin{tikzpicture}[scale=.9]
\begin{scope}
\clip (-.2,-.2) rectangle (5.2,5.2);
\draw (2.5,2.5) node{$\bullet$} node[below right]{$0$}; 
\draw (1.1,4) node{$\Qrbq$};
\pgfmathsetmacro{\a}{.8}
\pgfmathsetmacro{\b}{-.05}
\pgfmathsetmacro{\c}{0}
\pgfmathsetmacro{\d}{.01}
\draw [very thin, smooth]	plot[variable=\x, domain={-2.2:2.43}]
		({\x+2.5}, {2.5+\a*\x+\b*\x*\x+\c*\x*\x*\x+\d*\x^5});
\draw [very thick, smooth]	plot[variable=\x, domain={-.3:.3}]
		({\x+2.5}, {2.5+\a*\x+\b*\x*\x+\c*\x*\x*\x+\d*\x^5});
		\draw [very thick, smooth]	plot[variable=\x, domain={.4:.65}]
		({\x+2.5}, {2.5+\a*\x+\b*\x*\x+\c*\x*\x*\x+\d*\x^5});
\draw [very thick, smooth]	plot[variable=\x, domain={.9:1.3}]
	({\x+2.5}, {2.5+\a*\x+\b*\x*\x+\c*\x*\x*\x+\d*\x^5});
\draw [very thick, smooth]	plot[variable=\x, domain={1.6:1.7}]
		({\x+2.5}, {2.5+\a*\x+\b*\x*\x+\c*\x*\x*\x+\d*\x^5});
\draw [very thick, smooth]	plot[variable=\x, domain={1.9:2.3}]
		({\x+2.5}, {2.5+\a*\x+\b*\x*\x+\c*\x*\x*\x+\d*\x^5});
\draw [very thick, smooth]	plot[variable=\x, domain={-1:-.4}]
		({\x+2.5}, {2.5+\a*\x+\b*\x*\x+\c*\x*\x*\x+\d*\x^5});
\draw [very thick, smooth]	plot[variable=\x, domain={-1.5:-1.2}]
		({\x+2.5}, {2.5+\a*\x+\b*\x*\x+\c*\x*\x*\x+\d*\x^5});
\draw [dashed]	plot[variable=\x, domain={-2.5:2.5}]
		({\x+2.5}, {2.5+\a*\x});
\draw (.5,1.5) node{$L$};
\draw[very thin] (0,0) rectangle (5,5);
\end{scope}

\begin{scope}[xshift=6.2cm]
\clip (-.2,-.2) rectangle (5.2,5.2);
\draw[fill, color=gray!45]  (1.8,1.8) rectangle (5,5);
\draw[thin] (1.8,1.8)node{$\bullet$}node[left]{$x(t_-)$}; 
\draw (2.5,2.5) node{$\bullet$}node[right]{$0$};
\draw (4,3.8) node{$\om_{t_-}$};
\draw[very thin] (0,0) rectangle (5,5);
\end{scope}

\begin{scope}[xshift=12.4cm]
\clip (-.2,-.2) rectangle (5.2,5.2);
\draw[fill, color=gray!45]  (3.2,5) -- (3.2,3.2) -- (3.9,3.2)--(3.9,2.8) -- (5,2.8) -- (5,5) -- cycle;  
\draw[thin, dashed] (3.2,5) -- (3.2,2.95) node{$\bullet$}node[above left]{$h(t_+)$} -- (5,2.95); 
\draw (2.5,2.5) node{$\bullet$}node[below left]{$0$};
\draw (4,3.8) node{$\om_{t_+}$};
\draw[very thin] (0,0) rectangle (5,5);
\end{scope}
\end{tikzpicture}
\caption{\label{Fig:trajectory}Left:  example of trajectories $\{h(t)\}$ (thin) and $\{x(t)\}_{t\in\JJ}$ (thick) with their tangent $L$ at $0$ (dashed). Middle: $t_-\in\II$ so $x(t_-)=h(t_-)$. Right: $t_+\not\in\JJ$ and $x(t_+)$ is not defined.}
\end{figure}
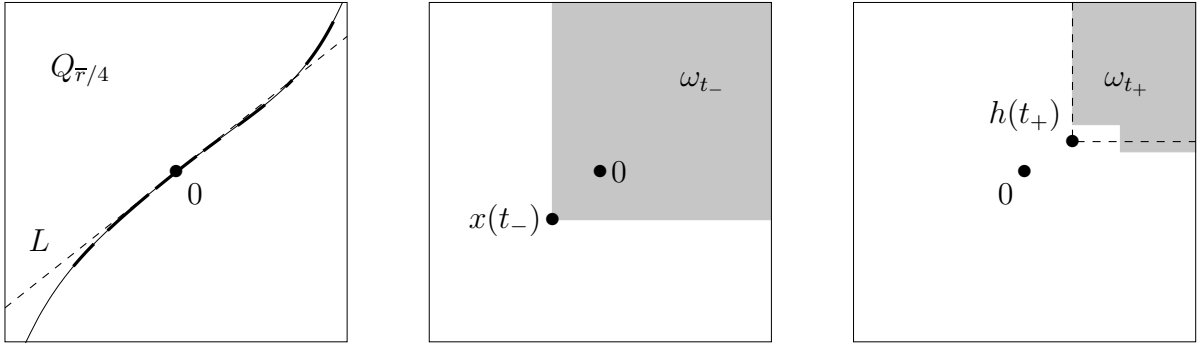 

\begin{proof}[Proof of Lemma~\ref{lem:step1}]~\\
\noindent
\textit{Step 1. Proof of the identity $x(t)=h(t)$ for $t\in\JJ$.}\\
Let $t\in\JJ_0$ we have $\CC(\om_t)=\{x(t)\}$ with $x(t)\in\Qrbq$. If moreover $t\in\JJ\sh_+$, we see that
\[
\Gamma_t\cap Q\sh_+ = [\{x_2\}\times \R]\cap Q\sh_+.
\]
Hence
\[
\om_t\cap Q\sh_+ =
\begin{cases}
\text{ either}& \{y\in Q\sh_+:y_2>x_2(t)\},\smallskip\\
\quad\text{or}& \{y\in Q\sh_+:y_2<x_2(t)\}.
\end{cases}
\]
Recalling~\eqref{perfect_om0} we see that for $t=0$ the first case holds and by monotonicity of $t\mapsto \om_t$, we get that the first case holds for every $t\in\JJ_0\cap\JJ\sh_+$. Recalling the  definition~\eqref{defh} of $h_1(t), h_2(t)$ we deduce that $x_2(t)=h_2(t)$ for $t\in\JJ_0\cap\JJ\sh_+$. Arguing similarly with $Q\sv_+$ we obtain $x_1(t)=h_1(t)$ for $t\in\JJ_0\cap\JJ\sv_+$. We conclude that $x(t)=h(t)$ for $t\in\JJ$.
\medskip

\noindent
\textit{Step 2. Proof that $0\in\dens(\JJ)$ and of~\eqref{dx/dt}.}\\
First, by~\eqref{perfect_om0} we have $0\in\JJ$ and recalling~\eqref{0indensJJ0}, we have $0\in\dens(\JJ_0)$. Next, for $t\in\JJ_0$, there is no corner of $\om_t$ in $Q\sh_+$ so $Q\sh_+\cap\om_t$ is a (finite) union of stripes and $\HH^1(Q\sh_+\cap\Gamma_t)=k_t\ell\sh$ for some integer $k_t\ge 0$. Since by assumption, $0$ is a Lebesgue point of  $t\mapsto \HH^1(Q\sh_+\cap\Gamma_t)$, we have $0\in\dens(\JJ\sh_+)$. With the same argument, the same property holds for the three other sets $\JJ\sv_+$, $\JJ\sh_-$, $\JJ\sv_-$ and we deduce that $0\in\dens(\JJ)$. Eventually,~\eqref{dx/dt} follows from the differentiability of $h$ at $0$.  \medskip

\noindent
\textit{Step 3. Proof of the identity~\eqref{omtQtrb}.}\\
Using again that for $t\in \JJ_0$, we have $\CC(\om_t)\cap \Qrb=\{x(t)\}$ we see that, for $t\in\JJ$,
\[
\Gamma_t\cap\lt[(z_1,\ov r)\times(z_2,\ov r)\rt]=\lt[(x_1(t),\ov r)\times\{x_2(t)\}\rt]\, \cup\,\{x(t)\}\, \cup\, \lt[\{x_1(t)\} \times(x_2(t),\ov r)\rt].
\]
\[
\text{where}\hfill\qquad\begin{cases}
z_1:=\max\lt(\inf\{y_1,y\in Q\sv_-\},\inf\{y_1,y\in Q\sv_+\}\rt),\smallskip\\
z_2:=\max\lt(\inf\{y_2,y\in Q\sh_-\},\inf\{y_2,y\in Q\sh_+\}\rt).
\end{cases}
\]
From the constraints~\eqref{Q_h},\eqref{Q_v}\&\eqref{Q_hQ_v-} on $Q\sh_\pm$ and $Q\sv_\pm$, we have $z_1,z_2<-\ov r/4$ and we get 
\[
\Gamma_t\cap\Qtrb=(x_1(t),\ov r)\times\{x_2(t)\}\ \cup\,\{x(t)\}\ \cup\ \{x_1(t)\} \times(x_2(t),\ov r).
\]
Eventually, we deduce~\eqref{omtQtrb} by monotonicity of $t\mapsto\om_t$. \end{proof}
Let us describe further the trajectory $t\in\JJ \mapsto x(t)$ and express its derivative in terms of the values of $\nb u$ on the sets $\Gamma_t$. For this we introduce some more notation. For $t\in\JJ$, we denote 
\be\label{Gammat}
\Gamma_t\sh:=(x_1(t),\ov r)\times\{x_2(t)\}\quad\qquad\text{and}\quad\qquad\Gamma_t\sv:=\{x_1(t)\}\times(x_2(t),\ov r),
\ee 
hence, 
\[
\Gamma_t\cap \Qtrb=\Gamma_t\sh\cup\{x(t)\}\cup\Gamma_t\sv.
\]
Remark for later use that for $t\in\JJ$, since $x(t)\in \Qrbq$  the lengths of the segments $\Gamma_t\sh$ and $\Gamma_t\sv$ are bounded from below by $3\ov r/4>0$.\medskip

We start with a simple result. 
\begin{lemma}\label{lem:Dxi_Dh_i}
For every $s,t\in \JJ$, $s<t$ and $l=1,2$, there holds
 \[
  t-s\le x_l(t)-x_l(s)=h_l(t)-h_l(s) = D h_l((s,t)).
 \]
 \end{lemma}
 \begin{proof}~
Without loss of generality, we assume that $l=1$. Let $y_2\in(\ov r/4,\ov r)$. By Lemma~\ref{lem:step1}, for $s,t\in\JJ$, we have $(x_1(s),y_2)\in\Gamma_s\sv$ and  $(x_1(t),y_2)\in\Gamma_t\sv$. Hence, using the fact that $u$ is 1-Lipschitz continuous, we get
\[
t-s=u(x_1(t),y_2)-u(x_1(s),y_2)\le |x_1(t)-x_1(s)|,
\]
and since $x_1$ is nondecreasing  we can remove the absolute value in the last term. For the other inequality, we recall that $x_1=h_1$ on $\JJ$ and that $h_1$ is $BV$ and continuous at any point of $\JJ_0\supset \JJ$.
\end{proof}
We now identify the derivatives $h_l'(0)$ in terms of $v=\nb u$.
\begin{lemma}
\label{lem:h_i'}
Let $l\in\{1,2\}$. The function $h_l$ is differentiable at any point $t\in\dens(\JJ)$ and 
\[
h_l'(t)\ge 1.
\]
Moreover:
\begin{enumerate}[(i)]
\item For almost every $y_2\in (x_2(t),\ov r)$  there holds 
\[
v_1(x_1(t)e_1 + y_2e_2)=0\qquad\text{ and }\qquad h'_2(t) v_2(x_1(t)e_1 + y_2e_2)=1.
\] 
As a consequence  $v=\nb u$ is well defined and \ae on $\Gamma_t\sh$ with $v_1=0$ and $v_2>0$.
\item Symmetrically, for almost every $y_1\in (x_1(t),\ov r)$  there holds 
\[
h'_1(t) v_1(y_1e_1+x_2(t)e_2)=1\qquad\text{ and }\qquad v_2(y_1 e_1 + x_2(t)e_2)=0.
\] 
As a consequence  $v$ is constant \ae on $\Gamma_t\sv$ with $v_1>0$ and $v_2=0$.
\end{enumerate}
In particular, since $0\in\dens(\JJ)$ we have for almost every $y_1,y_2\in (0,\ov r)$,
\[
v(y_1,0)=(0,v_2^\oo)\qquad\text{ and }\qquad v(0,y_2)=(v_1^\oo,0),
\]
where we have used the notation
\be\label{idvooh'(0)}
v_1^\oo:=\dfrac1{h_1'(0)},\qquad\qquad v_2^\oo:=\dfrac1{h_2'(0)}.
\ee
\end{lemma}
\begin{proof}~
 We consider $l=1$, the case $l=2$ being identical. Let $t\in\dens(\JJ)$.  For almost every $y_1\in (x_1(t),\ov r)$, the function $u$ is differentiable at $(y_1,x_2(t))\in\Gamma_t\sh$. Since $u(\tilde y_1,x_2(t))=t$ for $\tilde y_1$ in the neighborhood of $y_1$, we have $\pt_1u(y_1,x_2(t))=0$. Moreover,  writing $u(\tilde y_1,x_2(s))=s$ for $s\in\JJ$, the chain rule leads to
\[
h_2'(t)\pt_2 u(y_1,x_2(t))=1,
\]
as claimed. Consequently $\pt_2 u(y_1,x_2(t))\ne 0$ and up to a $\HH^1$-negligible set, $v=\nb u$ is constant on $\Gamma_t\sh$. Eventually, the inequality $h_2'(t)\ge1$ follows from Lemma~\ref{lem:Dxi_Dh_i}.
\end{proof}

\noindent
We are now ready to establish the first part of Theorem~\ref{theo:rectifLipmain} namely~\eqref{eq:rectifLipclaim}.
\begin{proposition}\label{prop:theorectif_1}
For every $\vhi\in C_c(\R^2)$,
\[
\lim_{r\to 0}\dfrac1r\int_{\R^2}\vhi(r^{-1} y) \,d|\mu|(y)=\int_{\R}\vhi(s h'(0))\,ds.
\]
\end{proposition}
\begin{remark}
 Let us point out that~\eqref{eq:rectifLipclaim} indeed follows from Proposition~\ref{prop:theorectif_1} since we have, with the change of variable $s=\sigma/|h'(0)|$ and using~\eqref{idvooh'(0)},
 \[
 \int_\R\vhi(s h'(0))\,ds= \dfrac1{|h'(0)|}\int_\R \vhi\lt(\sigma\dfrac{h'(0)}{|h'(0)|}\rt)\, d\sigma =\dfrac{v_1^\oo v_2^\oo}{|v^\oo|}\int_L \vhi\,d\HH^1,
 \]
 with 
 \[
 L:=\Span h'(0)=\Span(1/v_1^\oo,1/v_2^\oo).
 \]
Moreover by the rectifiability criterion of~\cite[Theorem~16.7]{Mattila}, the existence of approximate tangent measures to $|\mu|$ provided by the proposition yields the $\HH^1$-rectifiability of $|\mu|$ and thus of $\mu$. 
\end{remark}
\begin{proof}[Proof of Proposition~\ref{prop:theorectif_1}.]~
Fix $\vhi\in C_c(\R^2)$ with $\supp \vhi\sub Q_\ell$ for some $\ell>0$.
   By~\eqref{decomptildemu} we have for $r>0$,
\[
\dfrac1r\int_\Om \vhi\lt(r^{-1}y\rt) d|\mu|(y)= \dfrac1r\int_{\R} \int_{Q_{\ell r}} \vhi\lt(r^{-1}y\rt)\, d|\kappa_t|(y)\,dt.
\]
Notice that for $r<\ov r/\ell$ we have $Q_{\ell r}\sub\Qrb$. For almost every $t\in \R$, if $|\kappa_t|(Q_{\ell r})\neq 0$ then $\Gamma_t\cap Q_{\ell r}\neq \emptyset$, \ie there exists $x\in Q_{\ell r}$ such that $u(x)=t$. Using that $u(0)=0$ and the fact that $u$ is $1-$Lipschitz we deduce that 
\[
 |\kappa_t|(Q_{\ell r})\neq 0 \qquad \Longrightarrow \qquad |t|=|u(x)|\le |x|\le \ell r.
\]
Thus,
\[
\int_{Q_{\ell r}} \vhi\lt(r^{-1}y\rt)\, d|\kappa_t|(y)=
\begin{cases}
\displaystyle \int_{Q_{\ell r}} \vhi\lt(r^{-1}y\rt)\, d|\kappa_t|(y)&\text{ if }|t|< \ell r,\\
\qquad 0&\text{ for }|t|\ge \ell r.
\end{cases}
\]
Hence for $r< \ov r/\ell$,
\begin{multline}\label{prop:rectifLipclaim1}
\dfrac1r\int_\Om \vhi\lt(r^{-1}y\rt) d|\mu|(y)= \dfrac1r\int_{- \ell r}^{ \ell r} \int_{\Qrb} \vhi\lt(r^{-1}y\rt)\, d|\kappa_t|(y)\,dt\\
\stackrel{\eqref{eq:kappaBall}}{=} \dfrac1r \int_{\JJ\cap (- \ell r, \ell r)}\vhi\lt(r^{-1}x(t)\rt)\,dt + \dfrac1r \int_{(-\ell r , \ell r)\sm\JJ} \int_{\Qrb}\vhi\lt(r^{-1}y\rt)\,d|\kappa_t|(y)\,dt\\
=:q_1(r)+q_2(r).
\end{multline}
Let us treat the first term. We write
\begin{align*}
  q_1(r)
  &=  \dfrac1r \int_{\JJ\cap (- \ell r, \ell r)}\vhi(r^{-1} th'(0))\, dt +   \dfrac1r  \int_{\JJ\cap (- \ell r, \ell r)}[\vhi\lt(r^{-1}x(t)\rt)-\vhi(r^{-1}th'(0))]\, dt\\
&=:q_{1,1}(r)+q_{1,2}(r).
\end{align*}
Using the change of variable $t=rs$ and recalling that $0\in\dens(\JJ)$, we can pass to the limit in $q_{1,1}(r)$: 
\be\label{prop:rectifLipclaim16}
 q_{1,1}(r)=\int_{\frac1r \JJ\cap (-\ell,\ell)}\vhi(sh'(0))\, ds\ \st{r\dw0}\longto\ 
\int_\R\vhi(sh'(0))\, ds.
\ee
Notice that we used that $h'_i(0)\ge 1$  by Lemma~\ref{lem:h_i'} and thus $sh'(0)\notin Q_\ell$ for $|s|\ge\ell$.\\
To show that $q_{1,2}(r)$ is negligible. We introduce a modulus of continuity $\eta_\vhi\in C(\R_+,\R_+)$ of $\vhi$, increasing and with $\eta_\vhi(0)=0$. We also define 
\[
 \delta(r):=\sup\lt\{ \lt|\dfrac{x(t)}{t}-h'(0)\rt| : t\in \JJ,\, 0<|t|<r\rt\}.
\]
Remark that by~\eqref{dx/dt}, $\delta(r)\to 0$ as $r\to 0$. We then have 
\begin{align*}
|q_{1,2}(r)|
 &\le  \dfrac1r \int_{\JJ\cap (-\ell r,\ell r)}\eta_\vhi\lt(r^{-1}(x(t)-th'(0))\rt)\, dt\\
 &\le \ell \eta_\vhi(\ell \delta(\ell r)).
\end{align*}
Sending $r$ to $0$, we get $\lim_{r\dw0} q_{1,2}(r)=0$ and with~\eqref{prop:rectifLipclaim16} we conclude  that
\be\label{prop:rectifLipclaim18}
\lim_{r\dw0} q_1(r)=\int_\R\vhi(sh'(0))\, ds.
\ee
We still have to establish that the second term in the right-hand side of~\eqref{prop:rectifLipclaim1} is negligible. Let us write 
\[
|q_2(r)|\le\dfrac{\|\vhi\|_\oo}r\int_{(-\ell r,\ell r)\sm\JJ }|\kappa_t|(\Qrb)\, dt\le \dfrac{\|\vhi\|_\oo}r\int_{(-\ell r,\ell r)\sm\JJ} N(t)\, dt.
\]
Recalling that $ 0\in \dens(\JJ)$ and that 0 is a Lebesgue point of the $L^1$ function $t\mapsto N(t)$, we get that $q_2(r)\to0$ as $r\to0$. Together with~\eqref{prop:rectifLipclaim1}$\And$\eqref{prop:rectifLipclaim18} this yields the result.
\end{proof}

\noindent
We can finally end the proof of Theorem~\ref{theo:rectifLipmain} by establishing the trace identity~\eqref{eq:strongtracesclaim}. Denoting 
\[
 \nu:= \dfrac1{|v^\oo|}(-v_1^\oo,v_2^\oo) = \dfrac1{|h'(0)|}[h'(0)]^\perp,
\]
we set
\[
 V^\oo(y)=\begin{cases}
              (v^\oo_1,0)=\lt(\dfrac1{h'_1(0)},0\rt) & \text{if } y\cd  \nu >0, \bigskip\\
              (0,v^\oo_2)=\lt(0,\dfrac1{h'_2(0)}\rt) & \text{if } y\cd \nu<0.
             \end{cases}
\]
Notice that $\nb\we V^\oo=0$ in the sense of distributions. 
\begin{proposition}\label{prop:theorectif_2} There holds
\be\label{L1_trace}
\lim_{r\dw0}\dfrac1{r^2}\int_{B_r}|v-V^\oo|\,=0.
\ee
\end{proposition}
\begin{proof}~
Let $0<r<\ov r$ and let us denote (see Figure~\ref{Fig:BhBv}),
\[
B_r\sh:=\{y\in B_r: y\cd \nu<0\} \qquad \textrm{and} \qquad B_r\sv:=\{y\in B_r: y\cd \nu>0\}. 
\]
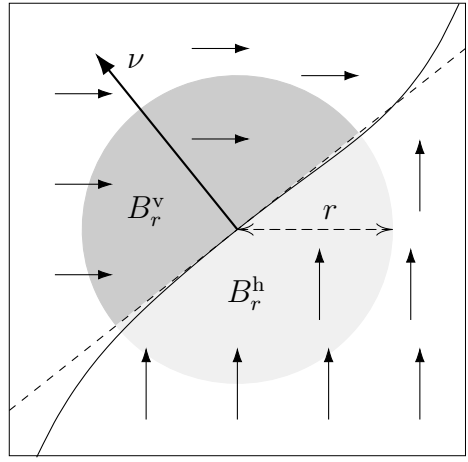
\begin{figure}[H]
\centering
\begin{tikzpicture}[scale=1.2]
\clip (-.2,-.2) rectangle (5.2,5.2);
\pgfmathsetmacro{\a}{.8}
\pgfmathsetmacro{\b}{-.05}
\pgfmathsetmacro{\c}{0}
\pgfmathsetmacro{\d}{.01}
\pgfmathsetmacro{\r}{1.7}
\pgfmathsetmacro{\R}{2.5}
\pgfmathsetmacro{\angle}{atan(\a)}
\pgfmathsetmacro{\C}{cos(\angle)}
\pgfmathsetmacro{\S}{sin(\angle)}

\draw[thin,fill, color=gray!40] (2.5+\r*\C,2.5+\r*\S) arc (\angle:\angle+180:\r) -- cycle;
\draw[thin,fill, color=gray!12] (2.5-\r*\C,2.5-\r*\S) arc (\angle+180:\angle+360:\r) -- cycle;
\draw[very thin] (0,0) rectangle (5,5);

\draw[-{>[scale=3,length=2,width=1.5]},very thin, dashed] (2.5,2.5)--({2.5+\r},2.5);
\draw[-{>[scale=3,length=2,width=1.5]},very thin, dashed] ({2.5+\r},2.5) -- (2.5,2.5);
\draw (2.6,1.8) node{$B_r\sh$};
\draw (1.5,2.7) node{$B_r\sv$};

\draw ({2.5+.6*\r},{2.5}) node[above]{$r$};

\draw[-{Latex[length=3mm, width=2mm]},thick] (2.5,2.5)--++({-\R*\S},{\R*\C});
\draw ({2.5-.85*\R*\S},{2.5+.85*\R*\C}) node[above right]{$\nu$};

\pgfmathsetmacro{\vx}{.5/cos(\angle)}
\pgfmathsetmacro{\vy}{.5/sin(\angle)}
\draw[-{Latex[length=2mm, width=1.3mm]}](.5,2)--++(\vx,0);
\draw[-{Latex[length=2mm, width=1.3mm]}](.5,3)--++(\vx,0);
\draw[-{Latex[length=2mm, width=1.3mm]}](.5,4)--++(\vx,0);
\draw[-{Latex[length=2mm, width=1.3mm]}](2,3.5)--++(\vx,0);
\draw[-{Latex[length=2mm, width=1.3mm]}](2,4.5)--++(\vx,0);
\draw[-{Latex[length=2mm, width=1.3mm]}](3.2,4.2)--++(\vx,0);

\draw[-{Latex[length=2mm, width=1.3mm]}](1.5,.4)--++(0,\vy);
\draw[-{Latex[length=2mm, width=1.3mm]}](2.5,.4)--++(0,\vy);
\draw[-{Latex[length=2mm, width=1.3mm]}](3.5,.4)--++(0,\vy);
\draw[-{Latex[length=2mm, width=1.3mm]}](4.5,.4)--++(0,\vy);
\draw[-{Latex[length=2mm, width=1.3mm]}](3.4,1.5)--++(0,\vy);
\draw[-{Latex[length=2mm, width=1.3mm]}](4.4,1.5)--++(0,\vy);
\draw[-{Latex[length=2mm, width=1.3mm]}](4.5,2.7)--++(0,\vy);

\draw [thin, smooth]	plot[variable=\x, domain={-2.2:2.43}]
		({\x+2.5}, {2.5+\a*\x+\b*\x*\x+\c*\x*\x*\x+\d*\x^5});
\draw [dashed]	plot[variable=\x, domain={-2.5:2.5}]
		({\x+2.5}, {2.5+\a*\x});

\end{tikzpicture}
\caption{\label{Fig:BhBv}The vector field $V^\oo$, the half balls $B_r\sh$ (light gray) and $B_r\sv$ (dark gray) and, as in Figure~\ref{Fig:trajectory}, the trajectory $\{h(t)\}$ (solid curve) with its tangent at $0$ (dashed line).}
\end{figure} 

We split the integral in~\eqref{L1_trace} as
\[
\dfrac1{r^2}\int_{B_r}|v-V^\oo|\,=\dfrac1{r^2}\int_{B_r\sh}|v-(0,v_2^\oo)|\,+\dfrac1{r^2}\int_{B_r\sv}|v-(v_1^\oo,0)|\,=:q\sh(r)+q\sv(r).
\]
We show that $q\sh(r)$ goes to 0 with $r$ (the treatment of $q\sv(r)$ is identical).  Since $h_2$ is differentiable at $0$ and $h_2'(0)\ge1$, there holds $(-r,r)\sub h_2((-2r,2r))$ for $r>0$ small enough. Setting 
\[ 
\lambda:=\dfrac{h'_1(0)}{h'_2(0)}=\dfrac{v_2^\oo}{v_1^\oo}
\]
we have by Fubini,
\[
q\sh(r)\le \dfrac1{r^2} \int_{h_2((-2r,2r))}\int_{\lambda y_2}^{2r}|v(y_1,y_2)-(0,v_2^\oo)|\,dy_1\,dy_2.
\]
We split the domain of integration with respect to $y_2$ in a ``good'' and a ``bad'' set: 
\[
D_{\mathrm{g}}(r):=h_2((-2r,2r)\cap\JJ)\qquad\text{ and }\qquad D_{\mathrm{b}}(r):=h_2((-2r,2r)\sm\JJ).
\]
We have $q\sh(r)\le q\sh_{\mathrm{g}}(r)+q\sh_{\mathrm{b}}(r)$ where
\begin{align*}
q\sh_{\mathrm{g}}(r)&:= \dfrac1{r^2} \int_{D_{\mathrm{g}}(r)}\int_{\lambda y_2}^{2r}|v(y_1,y_2)-(0,v_2^\oo)|\,dy_1\,dy_2,\\
q\sh_{\mathrm{b}}(r)&:= \dfrac1{r^2} \int_{D_{\mathrm{b}}(r)}\int_{\lambda y_2}^{2r}|v(y_1,y_2)-(0,v_2^\oo)|\,dy_1\,dy_2.
\end{align*}
 \smallskip

Let us first consider the term $q\sh_{\mathrm{g}}(r)$. For $y_2\in D_{\mathrm{g}}(r)$ there exists a unique $t$ such that $h_2(t)=y_2$. We denote $T(y_2):=t$ and then $X_1(y_2):=h_1(T(y_2))=h_1(h_2^{-1}(y_2))$. Taking into account that $|v|,|v^\oo|\le 1$ we get 
\begin{align*}
q\sh_{\mathrm{g}}(r)\le & \dfrac1{r^2}\int_{D_{\mathrm{g}}(r)}\int_{X_1(y_2)}^{2r}|v(y_1,y_2)-(0,v_2^\oo)|\,dy_1\,dy_2+\dfrac2{r^2}\int_{D_{\mathrm{g}}(r)}|X_1(y_2)-\lambda y_2|\, dy_2\\
&=: q\sh_{\mathrm{g},1}(r)+q\sh_{\mathrm{g},2}(r).
\end{align*}
Recalling the definition~\eqref{Gammat} of $\Gamma_t\sh$ and $\Gamma_t\sv$,  we have by Lemma~\ref{lem:h_i'} that for $t\in \JJ$, $v$ is constant on $\HH^1$-almost all $\Gamma_t\sh$ and $\HH^1$-almost every point of $\Gamma_t\sh$ is a Lebesgue point of $v$. Let $y^*_1\in(\ov r/4,\ov r)$ such that $y^*:=(y^*_1,0)$ is a Lebesgue point of $v$. In particular $v(y^*)=(0,v_2^\oo)$ and we have the estimate,
\[
q\sh_{\mathrm{g},1}(r)\le\dfrac1{r^2}\int_{Q_{2r}(y^*)}|v(y)-v(y^*)|\,dy\ \ \st{r\dw0}\longto\ 0.
\]
Then, since $h$ is differentiable at $0$ with $h_2'(0)\ge 1$, we have 
\[
\dfrac{X_1(y_2)}{y_2}-\frac{h'_1(0)}{h_2'(0)}=\frac{X_1(y_2)}{y_2}-\lambda\  \ \st{y_2\dw0} \longto\ 0.
\]
Using again that $h_2$ is differentiable at $0$ we have for $r>0$ small enough (and using $ h_2'(0)\ge1>0$)
\[
D_{\mathrm g}(r)\subset h_2((-2r,2r)) \subset  (-4h_2'(0) r,4h_2'(0) r).
\]
We infer, for $r>0$ small enough,
\[
q^{ \mathrm h}_{{\mathrm g},2}(r)= \dfrac2r\int_{D_{\mathrm g}(r)} \dfrac{|y_2|}r\lt|\frac{X_1(y_2)}{y_2}-\lambda \rt|\, dy_2\le \dfrac{8h_2'(0)}r\int_{-4h_2'(0) r}^{4h_2'(0) r} \lt|\frac{X_1(y_2)}{y_2}-\lambda \rt|\, dy_2\  \st{r\dw0} \longto\ 0.
\] 
We conclude that
\be\label{split_int_trace_1}
\lim_{r\dw0}q\sh_{\mathrm g}(r)= 0.
\ee 
Eventually, we estimate $q\sh_{\mathrm b}(r)$. Using $|v|,|v^\oo|\le 1$ and recalling that $h_2$ is increasing, we have
\[
q\sh_{\mathrm b}(r)\le \frac2r \HH^1\lt(h_2((-2r,2r)\sm\JJ)\rt).
\]
 Recalling the decomposition of $Dh_2=h_2'\HH^1 + D^sh_2$ and that $h_2$ is increasing, we claim that for $I\subset (-\ov r,\ov r)$ Lebesgue measurable (notice that since $h_2$ is increasing, $h_2(I)$ is also measurable), there holds 
\be\label{H1(h2(I))}
\HH^1(h_2(I))\le \nu(I) + \int_I h_2'(s)\, ds,
\ee
where\footnote{$\nu$ is the smallest outer measure built from $D^sh_2$.} $\nu(I):=\inf\{D^sh_2(J):J\text{ open subset, }I\subset J\}$. Indeed, the inequality holds true for open subsets with equality and then extends to every measurable set since, by regularity of $h_2'\HH^1$, 
\[
\int_I h_2'(s)\, ds = \inf\lt\{\int_J h_2'(s)\, ds :J\text{ open subset, }I\subset J\rt\}.
\]
Applying inequality~\eqref{H1(h2(I))} with $I=(-2r,2r)\sm\JJ$ we compute 
\begin{align*}
q\sh_{\mathrm b}(r)
&\le \frac2r\nu((-2r,2r)\sm\JJ)+\frac2r\int_{(-2r,2r)\sm\JJ} h_2'(s)\, ds\\
&\le  \frac2rD^sh_2((-2r,2r)) +\frac{2h_2'(0)}r\HH^1((-2r,2r)\sm\JJ) +\frac2r\int_{-2r}^{2r} |h_2'(s)-h_2'(0)|\, ds.
\end{align*}
Sending $r$ to 0, the first term goes to 0 by~\eqref{property:h2}, so does the second term because $0\in\dens(\JJ)$ as well as the last term because 0 is a Lebesgue point of $h_2'$. We conclude that $q\sh_{\mathrm b}(r)$ goes to 0 and with~\eqref{split_int_trace_1} the proposition is established. This ends the proof of Theorem~\ref{theo:rectifLipmain}.
\end{proof}


\subsubsection{Compactness of $S^\oo(\Om)$}\label{sec:compactness}
We now prove Proposition~\ref{prop:compactnessintro} about the compactness properties of $S^\oo(\Om)$.
\begin{proof}[Proof of Proposition~\ref{prop:compactnessintro}]~

 \setcounter{proof-step}{0}

 \noindent{\textit{Step \stepcounter{proof-step}\arabic{proof-step}.
 Compactness of finite energy states.}}\\
 Let $v^k\in S^\oo(\Om)$ be such that 
 \[
  \sup_k |\mu[v^k]|(\Om)<\oo.
 \]
Since $|v_k|\le 1$, up to extraction there is $|v|\le 1$ such that $v_k$ converges  to $v$ in the weak-$*$ topology of $L^\oo$. By weak convergence we have that $\nb \times v=0$, that $\mu[v^k]$ weakly converges to $\mu[v]$ and the estimate
\[
 |\mu[v]|(\Om)\le \liminf_{k\to \oo} |\mu[v^k]|(\Om)<\oo.
\]
Therefore, to prove that $v\in S^\oo(\Om)$, we just need to show that  $v\in K$ for almost every $x$. Let $\vhi_l\in C^1(\R)$ for $l\in\{1,2\}$. We consider the entropies $\Phi_1(v)= \vhi_1(v_2) e_1$ and $\Phi_2(v)= \vhi_2(v_1) e_2$. For $l=1,2$ we have 
\be\label{eq:divphizero} 
 |\nb \cdot (\Phi_l(v^k))|(\Om)\le \|\vhi_l'\|_\oo |\mu[v^k]|(\Om).
\ee
We deduce that the two sequences $(\Phi_1(v^k))$ and $(\Phi_2(v^k))$ are compact in $H^{-1}(\Om)$ (see e.g.~\cite[Lemma~6]{DKMO01}). Thus, by the div-curl lemma we have as weak limits,
\be\label{eq:weaklim}
\lim_{k\up\oo} \vhi_1(v_2^k)\vhi_2(v_1^k)\, =\,  \lt( \lim_{k\up\oo}  \vhi_1(v_2^k)\rt)  \lt( \lim_{k\up\oo} \vhi_2(v_1^k)\rt).
\ee
Applying this with $\vhi_1=\vhi_2={\rm Id}$ we find $0 = v_2 v_1$ and thus $v\in K$.
\medskip

\noindent{\textit{Step \stepcounter{proof-step}\arabic{proof-step}. The case of vanishing defect measure: compactness in the strong topology.}}\\
We now assume that $\mu[v]=0$. We have $\pt_2 v_1=\pt_1 v_2=0$ so that  $v_l(x)= v_l(x_l)$ for $l=1,2$. Since $v\in K$ almost everywhere this leads to $v_1=0$ or $v_2=0$. We now  improve the weak convergence of $v^k$ to strong convergence of $v_1^k$ or $v_2^k$ to zero.\\
Let $(\lambda_x)_{x\in \Om}$ be the Young measure generated by $v^k$, \ie
\[
 \lim_{k\up \oo} \int_{\Om} \vhi(x, v^k(x))\,dx=\int_{\Om}\lt(\int_{B_1} \vhi(x,z)\,d\lambda_x(z)\rt)\,dx \qquad \text{for every } \vhi\in C_c(\Om\times \R^2).
\]
Arguing as in~\cite{DKMO01}, we see that $\lambda_x$ is a probability measure on $K$ with
\[
 v(x)=\int_K z\, d\lambda_x \qquad \text{for almost every } x\in \Om.
\]
After localizing, we can write~\eqref{eq:weaklim} as 
\[
\int_K  \vhi_1(z_2)\vhi_2(z_1)\, d\lambda_x(z)\, =\, \lt(\int_K  \vhi_1(z_2)\, d\lambda_x(z)\rt)\lt(\int_K  \vhi_2(z_1)\, d\lambda_x(z)\rt).
\]
Therefore $\lambda_x$ is tensorized and thus supported on $\{0\}\times \R$ or on $\R\times\{0\}$. Let us prove that, up to a set of Lebesgue measure zero, it is either always supported on $\{0\}\times \R$ or always on $\R\times\{0\}$.
We decompose $\lambda_x$ as
\be
\label{decomp}
\lambda_x=\alpha_x\delta_0\otimes\delta_0 + \lambda_x\sh\otimes \delta_0+\delta_0\otimes \lambda_x\sv,
\ee
 with $\alpha_x\ge 0$ and  $\lambda_x\sh$, $\lambda_x\sv$ positive measures on $\R$ such that $\lambda_x\sh(\{0\})=\lambda_x\sv(\{0\})=0$ and $\alpha_x+\lambda_x\sh(\R)+\lambda_x\sv(\R)=1$.
Let
\[
\om\sh=\{x\in \Om \ : \lambda_x\sh(\R)> 0\},\qquad\qquad\om\sv=\{x\in \Om \ : \lambda_x\sv(\R)> 0\}.
\]
We claim that either $|\om\sh|=0$ or $|\om\sv|=0$. Assume instead  that $|\om\sh|>0$, $|\om\sv|>0$. Notice first that since $\lambda_x$ is tensorized for almost every $x$, we have 
\be\label{om1capom2=0}
|\om\sh\cap \om\sv|=0.
\ee
Now, since $\mu[v]=0$, the estimate~\eqref{eq:divphizero} with $v$ in place of $v^k$ leads to the following identities, in the sense of distribution,
\[
 \pt_1 \lt[\int_{K} \vhi_1 (z_2)\,d\lambda_x(z)\rt]=0,\qquad\qquad  \pt_2 \lt[\int_{K} \vhi_2 (z_1)\,d\lambda_x(z)\rt]=0,
\]
for every $\vhi_1,\vhi_2\in C^1(\R)$. Using $\vhi_1(z_2)=|z_2|^2$  in the first identity and substituting the decomposition of $\lambda_x$, we get 
\[
\pt_1 \lt[ \un_{\om\sv}(x) \int_{K} |z_2|^2\,d\lambda\sv_x(z)\rt]=0.
\]
We deduce that for almost every $x_2\in (-1,1)$, the function 
\[
y_1\in (-1,1)\mapsto  \un_{\om\sv}(y_1,x_2) \int_{K} |z_2|^2\,d\lambda\sv_{(y_1,x_2)}(z)\qquad
\text{is constant (in the \ae sense).}
\] As the two factors either both vanish or are both positive we deduce that the function 
\be\label{unom2=cst}
y_1\in (-1,1)\mapsto  \un_{\om\sv}(y_1,x_2)\qquad\text{is constant for \ae\ }x_2\in (-1,1).
\ee
By assumption $\om\sv$ have positive measure, hence, by Fubini, there exists $J_2\sub (-1,1)$ measurable and with positive length such that for every $x_2\in J_2$ the set 
\[
  \{y_1\in (-1,1) : (y_1,x_2)\in \om\sv\text{ for some }x_2\in J_2\}
\]
has positive length.  We deduce from~\eqref{unom2=cst} that, up to a negligible set,
\[
(-1,1)\times J_2\sub\om\sv.
\]
Similarly using the second identity, we obtain a measurable subset $J_1\sub (-1,1)$ of positive length such that $J_1\times (-1,1)\sub\om\sh$. As a consequence  $J_1\times J_2\sub \om\sh\cap\om\sv$ and since $|J_1\times J_2|>0$ this contradicts~\eqref{om1capom2=0}. We conclude that either $|\om\sh|=0$ or $|\om\sv|=0$.

If for instance $|\om\sv|=0$ we obtain  from~\eqref{decomp} that the projection of $\lambda_x$ on the vertical axis is a Dirac delta at $0$ and thus $v_2^k\to 0$ in $L^1(\Om)$.
\end{proof}

As a final observation we construct $v\in S^\oo(\Om)$ such that $|\mu[v]|(B_r)$ goes to 0 faster than $r$ but such that $0$ is neither a Lebesgue point for $v_1$ nor for $v_2$.
\begin{proposition}\label{prop:counterexamp}
Let us consider a sequence $1>r_0>r_1>\dots$ decreasing to 0 such that $1/2>r_1/r_0>r_2/r_1>\dots$ also decreases to 0, see for instance the family of examples of Remark~\ref{rem:48} below.\\
Then, there exists $v\in S^\oo(\Om)$ such that
\be\label{contrex_|mu|_small}
 \lim_{r\dw 0}\dfrac{ |\mu[v]|(B_r)}r=0.
\ee
and for $l=1,2$
\be\label{eq:notLebesgue}
\dfrac1{|B_{r_{2k'+l}}|}\int_{B_{r_{2k'+l}}}\! \lt|v_l\rt|\,\ \st{k'\up\oo}\longto\ q:=\dfrac23-\dfrac{\sqrt3}{2\pi}>0.
\ee
\end{proposition}
\begin{proof}~ 
Let us set 
\be\label{hyp:repsk}
\eps_k:=2\dfrac{r_k}{r_{k-1}}\qquad\text{for }k\ge1.
\ee
By assumption, we have $1>\eps_1>\eps_2>\dots$ and $\eps_k$ goes to $0$.  Let $v^\eps$ be given as in Figure~\ref{Fig:counterexample}.
\begin{figure}[h]
\begin{center}
\begin{tikzpicture}[scale=1.1]
\pgfmathsetmacro{\L}{6}
\pgfmathsetmacro{\l}{1.2}
\newcommand{\Vv}[3]{	\draw[-{Latex[length=2mm, width=1.3mm]}] (#1,#2) --+(0,#3);	}
\newcommand{\Vh}[3]{	\draw[-{Latex[length=2mm, width=1.3mm]}] (#1,#2) --+(#3,0);	}
\draw[dashed] (-\L,-2*\l) rectangle (\L,2*\l);
\draw[dashed] (-\l,-\l) rectangle (\l,\l);
\draw (\l,\l) -- (\L,2*\l);
\draw (\l,-\l) -- (\L,-2*\l);
\draw (-\l,\l) -- (-\L,2*\l);
\draw (-\l,-\l) -- (-\L,-2*\l);
\pgfmathsetmacro{\av}{.15*\L}
\pgfmathsetmacro{\bv}{1.75*\l}
\pgfmathsetmacro{\cv}{.5*\l}
 \foreach \i in {-2,...,1}
	{	\Vv{\i*\av}{\bv}{-\cv}		\Vv{\i*\av}{-\bv}{\cv}  }
\pgfmathsetmacro{\ah}{.15*\L}
\pgfmathsetmacro{\bh}{2*\l}
\pgfmathsetmacro{\dbh}{.8*\l}
\pgfmathsetmacro{\ch}{.2*\l}
\foreach \j in {0,1}
	{\foreach \i in {-1,...,1}
		{	\Vh{\bh+\j*\dbh}{\i*\ah}{-\ch}		\Vh{-\bh-\j*\dbh}{\i*\ah}{\ch}   }
 	}      
\foreach \j in {2}
	{\foreach \i in {-1.5,1.5}
		{	\Vh{\bh+\j*\dbh}{(\i*\ah}{-\ch} 		\Vh{-\bh-\j*\dbh}{\i*\ah}{\ch}   }
 	}   	
\foreach \j in {3}
	{\foreach \i in {-1.5,...,1.5}
		{	\Vh{\bh+\j*\dbh}{(\i*\ah}{-\ch}		\Vh{-\bh-\j*\dbh}{\i*\ah}{\ch}  }
 	}    
\draw (.3*\L,1.6*\l) node{$-e_2$};
\draw (.3*\L,-1.6*\l) node{$e_2$};
\draw (-\l/5,\l/5) node{$v^\eps=0$};
\draw (.7*\L,0) node{$-\eps e_1$};
\draw (-.7*\L,0) node{$\eps e_1$};
\draw (-1.2*\L,0) node{$v^\eps=0$};
\pgfmathsetmacro{\ss}{.83}
\pgfmathsetmacro{\s}{.95}
\draw[<->,very thin] (-\s*\L,2.3*\l) --(\s*\L,2.3*\l);
\draw (0,2.3*\l) node[above]{$2/\eps$};
\draw[<->,very thin] (\L+.3*\l,-2*\s*\l) --(\L+.3*\l,2*\s*\l);
\draw ((\L+.3*\l,0) node[right]{$2$};
\draw[<->,very thin] (.85*\l,-\ss*\l) --(.85*\l,\ss*\l);
\draw (.85*\l,0) node[left]{$1$};
\draw[<->,very thin] (-\ss*\l,-.85*\l) --(\ss*\l,-.85*\l);
\draw (0,-.85*\l) node[above]{$1$};
\end{tikzpicture}
\caption{\label{Fig:counterexample}The vector field $v^\eps$. It takes values in $\{0,\pm\eps e_1,\pm e_2\}$.}
\end{center}
\end{figure}
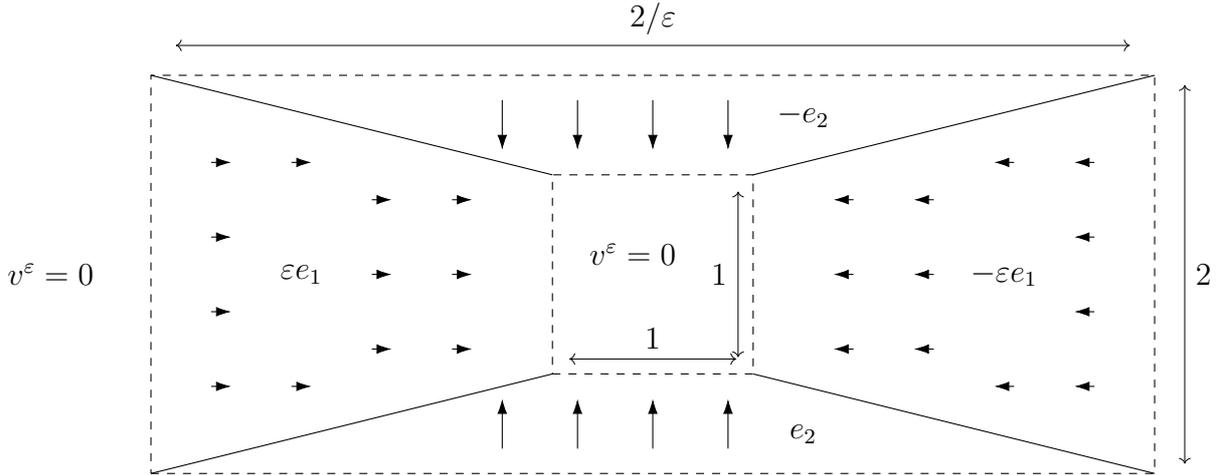
\noindent
Denoting $R$ the rotation of angle $\pi/2$ in $\R^2$, we define the vector field $v$ by
\[
 v:=\sum_{k\ge 0} \wt v^k\qquad\text{where}\ \ \quad \wt v^k(x):=R^{-k}v^{\eps_k}\lt(R^k \dfrac{x}{r_k}\rt)\quad\text{for }k\ge0\text{ and }x\in\R^2.
\]
Thanks to~\eqref{hyp:repsk}, the functions $\wt v^k$ have essentially disjoint supports (they overlap on segments), see Figure~\ref{Fig:supportofv}. More precisely, for $j>k$, $\wt v^j$ is supported in the central square $[-r_k/2,r_k/2]^2$ where $\wt v^k$ vanishes. 
\begin{figure}[H]
\begin{center}
\begin{tikzpicture}[scale=.7]
\newcommand{\suppv}[2]
	{
	\fill[color=gray!45] ({-#1},{-#2}) rectangle (#1,#2);
	}
\newcommand{\hole}[1]
	{
	\fill[color=white] ({-#1},{-#1}) rectangle (#1,#1);
	}
\pgfmathsetmacro{\c}{1}
\pgfmathsetmacro{\d}{.8}
\pgfmathsetmacro{\D}{.5}
\pgfmathsetmacro{\L}{7}
\pgfmathsetmacro{\deca}{.25}
\pgfmathsetmacro{\decb}{.5}
\pgfmathsetmacro{\s}{.99}
\pgfmathsetmacro{\ss}{.97}
\pgfmathsetmacro{\l}{\L}
\suppv{\L}{\l}
\draw[<->,very thin] (-\s*\L,1.07*\l) --(\s*\L,1.07*\l);
\draw (0,1.07*\l) node[above]{$2r_0$};
\draw[<->,very thin] ({-\L-\decb},-\s*\l) --({-\L-\decb},\s*\l);
\draw ({-\L-\decb},0) node[left]{$\phantom{\dfrac{r_1}{\eps_1}=}2r_0$};
 \hole{.5*\l}
\draw[<->,very thin] ({\L+\decb},-.5*\s*\l) --({\L+\decb},.5*\s*\l);
\draw ({\L+\decb},0) node[right]{$r_0=\dfrac{2r_1}{\eps_1}$};
\pgfmathsetmacro{\L}{.5*\c*\l}
\pgfmathsetmacro{\D}{\d*\D}
\pgfmathsetmacro{\l}{\D*\L}
\suppv{\l}{\L}
\hole{.5*\l}
\pgfmathsetmacro{\d}{.7}
\pgfmathsetmacro{\L}{.5*\c*\l}
\pgfmathsetmacro{\D}{\d*.7*\D}
\pgfmathsetmacro{\l}{\L*\D}
\suppv{\L}{\l}\
 \hole{.5*\l}
\pgfmathsetmacro{\L}{.5*\l}
\pgfmathsetmacro{\l}{\D*\L}
\suppv{\l}{\L}
\end{tikzpicture}
\caption{\label{Fig:supportofv}The support of the vector field $v$.}
\end{center}
\end{figure}
We deduce that the sum $v=\sum_{k\ge 0}\wt v^k$ is well defined. Since moreover, for $\eps>0$, 
\[
\nb \times v^\eps=0,\qquad|v^\eps|\le 1\quad\ \text{ and }\ \quad v^\eps_1v^\eps_2=0,
\]
we also have that $\nb \times v=0$ and that $v$ takes values in 
\[K\cap {\ov B}_1=([-1,1]\times\{0\})\cup(\{0\}\times[-1,1]).\]
Let us check that~\eqref{eq:notLebesgue} holds true. Let $k\ge2$ be an even integer, we have 
\[
\dfrac1{|B_{r_{k}}|}\int_{B_{r_k}} |v_2|
= \dfrac1{|B_{r_{k}}|}\int_{B_{r_k}} |\tilde v^k_2|\, + \dfrac1{|B_{r_{k}}|}\int_{(-\frac{r_k}2,\frac{r_k}2)^2} \sum_{j>k} |\tilde v^j_2|\,
=: q_k + \eta_k.
\]
For the first term, a look at Figure~\ref{Fig:vtildek} shows that as $k$ goes to $+\oo$,
\[
q_k =\underbrace{\dfrac{\lt|\{x\in B_1:|x_2|>1/2\}\rt|}{|B_1|}}_{q=\frac23-\frac{\sqrt3}{2\pi}}+O\lt(\eps_k+\dfrac{r_{k+1}}{r_k}\rt)
\ \st{k\up\oo}\longto\ q.
\]
For the remainder $\eta_k$, using the bound $|v|\le 1$ and the fact that  $\sum_{j>k} \tilde v^j$ is supported in $[-r_{k+1},r_{k+1}]\times[-r_k/2,r_k/2]$ (see again Figure~\ref{Fig:vtildek}), we obtain 
\[
\eta_k\le\dfrac2\pi \dfrac{r_{k+1}}{r_k}\st{\eqref{hyp:repsk}}=\dfrac{\eps_{k+1}}\pi\quad \st{k\up\oo}\longto\ 0.
\] 
This proves~\eqref{eq:notLebesgue} in the case $l=2$ (that is $2(k'+l)$ even). The case of odd integers is the same up to  a rotation.
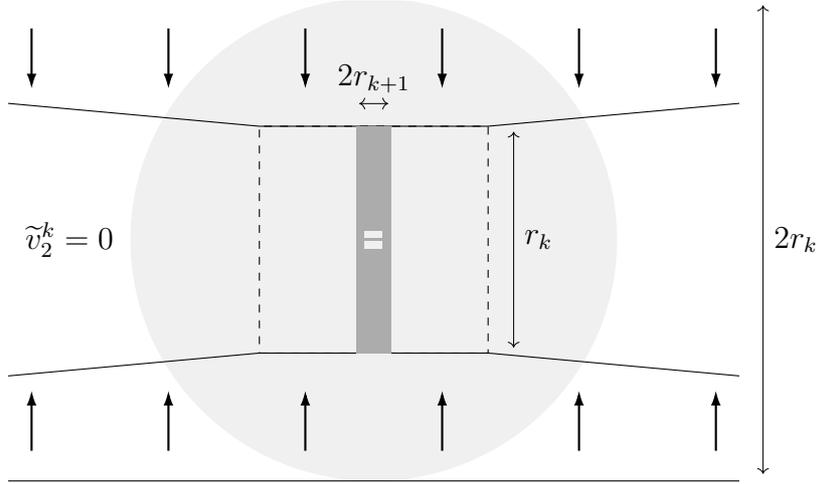
\begin{figure}[H]
\begin{center}
\begin{tikzpicture}[scale=2]
\pgfmathsetmacro{\L}{10}
\pgfmathsetmacro{\l}{.8}
\pgfmathsetmacro{\cf}{.94}
\fill[color=gray!12] (0,0) circle (2*\l); 
\newcommand{\Vv}[3]{	\draw[-{Latex[length=2mm, width=1.3mm]},thick] (#1,#2) --+(0,#3);	}
\begin{scope}
\clip (-3*\l,-2.1*\l) rectangle (3*\l,2.1*\l)  ;
\draw (-\L,-2*\l) rectangle (\L,2*\l);
\draw[dashed] (-\cf*\l,-\cf*\l) rectangle (\cf*\l,\cf*\l);
\draw (-\L,-2*\l) -- (-\cf*\l,-\cf*\l) -- (\cf*\l,-\cf*\l) -- (\L,-2*\l);
\draw (-\L,2*\l) -- (-\cf*\l,\cf*\l) -- (\cf*\l,\cf*\l) -- (\L,2*\l);
\pgfmathsetmacro{\av}{.09*\L}
\pgfmathsetmacro{\bv}{1.75*\l}
\pgfmathsetmacro{\cv}{.5*\l}
 \foreach \i in {-2.5,...,2.5}
	{	\Vv{\i*\av}{\bv}{-\cv}		\Vv{\i*\av}{-\bv}{\cv}  }
\end{scope}
\draw (-2.5*\l,0) node{${\wt v}^k_2=0$};
\pgfmathsetmacro{\ss}{.83}
\pgfmathsetmacro{\s}{.97}
\draw[<->,very thin] (3.2*\l,-2*\s*\l) --(3.2*\l,2*\s*\l);
\draw ((3.2*\l,0) node[right]{$2r_k$};
\pgfmathsetmacro{\s}{.95*\cf}
\draw[<->,very thin] (1.15*\l,-\s*\l) --(1.15*\l,\s*\l);
\draw (1.15*\l,0) node[right]{$r_k$};
\newcommand{\suppv}[2]
	{
	\fill[color=gray!65] ({-#1},{-#2}) rectangle (#1,#2);
	}
\pgfmathsetmacro{\c}{1}
\pgfmathsetmacro{\d}{.85}
\pgfmathsetmacro{\D}{.15}
\pgfmathsetmacro{\L}{\c*\cf*\l}
\pgfmathsetmacro{\decb}{.15}
\pgfmathsetmacro{\l}{\L*\D}
\suppv{\l}{\L}
\fill[color=gray!12] (-.5*\l,-.5*\l) rectangle (.5*\l,.5*\l); 
 \draw[<->,very thin] (-\l*\ss,\L+\decb) --(\l*\ss,\L+\decb);
\draw (0,\L+\decb) node[above]{$2r_{k+1}$};
\pgfmathsetmacro{\L}{.5*\l}
\pgfmathsetmacro{\D}{\d*\D}
\pgfmathsetmacro{\l}{\D*\L}
\suppv{\L}{\l}
\end{tikzpicture}
\caption{\label{Fig:vtildek}The vector field ${\wt v}^k_2 e_2\in\{0,\pm e_2\}$ in a neighborhood of $B_{r_k}$ for some (large) even integer $k$. The support of $\sum_{j>k}{\wt v}^j$ (dark gray) and the ball $B_{r_k}$ (light gray).}
\end{center}
\end{figure}

We are left with the proof of the vanishing energy limit~\eqref{contrex_|mu|_small}. Let us first notice the following scaling identity. There holds for $\eps>0$, $r>0$ and $A$ Borel subset of $\R^2$,
\[
\mu\lt[ v^\eps(\cdot/r)\rt](A)=r\mu[v^\eps]((1/r)A).
\]
In particular we have for $j\ge0$ and $r>0$,
\be\label{ctrex_scalingId}
\mu[\wt v^j](B_r)= r_j\mu[v^{\eps_j}](B_{r/r_j}).
\ee 
Next, by direct computation, we get for $\eps>0$ and $s>0$,
\[
\lt|\mu[v^\eps]\rt|(B_s)\le 4\eps s\qquad\ \text{ and }\ \qquad\lt|\mu[(v^\eps)]\rt|(\R^2)=4.
\]
With~\eqref{ctrex_scalingId}, we deduce  for $j\ge0$ and $r>0$, 
\be\label{ctrex_mu_veps}
\lt|\mu[\wt v^j]\rt|(B_r)\le 4\eps_jr\qquad\ \text{ and }\ \qquad \lt|\mu[(\wt v^j)]\rt|(\R^2)\le 4r_j.
\ee
Now we fix $r\in(0,r_0/2]$ and we denote $k=k(r)$ the unique positive integer such that $r_k/2< r\le r_{k-1}/2$. Taking into account the supports of the ${\wt v^j}$'s, we compute
\be \label{contrex_estim_mu}
 |\mu[v]|(B_r)
 =\lt|\mu[\wt v^k]\rt|(B_r) + \sum_{j\ge k+1} \lt|\mu[\wt v^j]\rt|(\R^2)\st{\eqref{ctrex_mu_veps}}
 \le 4\eps_kr +4\sum_{j\ge k+1}r_j.
\ee
To estimate the second term in the right-hand side we deduce from~\eqref{hyp:repsk}  the following chain of inequalities: $r_j<r_{j-1}/2<\dots <r_{k+1}/2^{j-k-1}=\eps_{k+1}r_k/2^{j-k}$ for  $j\ge k+1$. Hence,
\[
 \sum_{j\ge k+1} r_j < \eps_{k+1}r_k\lt(\sum_{j\ge k+1}\dfrac1{2^{j-k}}\rt)=\eps_{k+1}r_k.
\]
Using this estimate in~\eqref{contrex_estim_mu} we obtain, 
\[
 |\mu[v]|(B_r)<4\eps_kr+4\eps_{k+1}r_k<12\eps_k r.
\]
Dividing by $r$ and sending $r$ to 0, the condition $r_{k(r)}<2r$ leads to $k(r)\to+\oo$. Thus $\eps_{k(r)}\to0$ and~\eqref{contrex_|mu|_small} follows. This ends the proof of the proposition.
\end{proof}

\begin{remark}\label{rem:48}
 Let $\alpha>0$. The sequence  defined recursively by 
 \[
 r_0:=2^{-1/\alpha}\qquad\text{ and }\qquad r_{k+1}:=r_k^{\alpha+1}\quad\text{for }k\ge0
 \] 
 complies to the assumptions of the proposition. With this choice we have 
 \[
 |\mu[v]|(B_r)\le 48\,r^\beta\qquad\text{ with}\quad\beta:=2-\dfrac1{\alpha+1}.
 \] 
As we can make $\beta$ arbitrarily close to 2 by choosing $\alpha$ large enough we conclude that even a constraint of the form
 \[
\lt|\mu[v]\rt|(B_r)\le C r^{2-\eps},
 \]
 for some $\eps>0$ and $C\ge1$ does not ensure that $0$ is a Lebesgue point of either $v_1$ or $v_2$. 
\end{remark}

\subsection*{Acknowledments}
B. Merlet is partially supported by the INRIA team RAPSODI and the Labex CEMPI (ANR-11-LABX-0007-01). This work was supported by a public grant from the Fondation Mathématique Jacques Hadamard.

\bibliographystyle{alpha}
\bibliography{BibStripes}
\end{document}